\documentclass[font=11pt]{article}
\usepackage[margin=1in]{geometry}
\RequirePackage{amsthm,amsmath,amsfonts,amssymb, bbm, bm, mathtools, nccmath, mathrsfs}
\RequirePackage[numbers]{natbib}
\RequirePackage[colorlinks,citecolor=blue,urlcolor=blue]{hyperref}
\RequirePackage{graphicx}

\numberwithin{equation}{section}

\theoremstyle{plain}
\newtheorem{theorem}{Theorem}[section]
\newtheorem{lemma}[theorem]{Lemma}
\newtheorem{prop}[theorem]{Proposition}
\newtheorem{corollary}[theorem]{Corollary}

\theoremstyle{remark}
\newtheorem{remark}[theorem]{Remark}
\newtheorem{defn}[theorem]{Definition}

\newtheorem{assumptionT}{Assumption}
\newenvironment{taggedassump}[1]
 {\renewcommand\theassumptionT{#1}\assumptionT}
 {\endassumptionT}
\newtheorem{example}[theorem]{Example}

\newcommand{\beq}{\begin{equation}}
\newcommand{\eeq}{\end{equation}}
\def\beqs#1\eeqs{%
    \begin{equation}\begin{split}%
    #1%
    \end{split}\end{equation}%
}
\def\beqsn#1\eeqsn{%
    \begin{equation*}\begin{split}%
    #1%
    \end{split}\end{equation*}%
}
\def\beqsj #1\eeqsj
{
    \begin{equation}
    \setlength{\jot}{10pt}
    \begin{split}
    #1
    \end{split}\end{equation}
}
\DeclareMathOperator*{\argmin}{argmin}
\newcommand{\af}[1]{a_{#1}}
\newcommand{\bhat}{\hat b}
\newcommand{\bpi}{\bar b}
\newcommand{\calE}{\mathcal E}
\newcommand{\cgro}{c_{0}}
\newcommand{\cf}[1]{b_{#1}}

\newcommand{\corlap}{\lap_S}

\newcommand{\del}[1]{\epsilon_3^2 +\efour{#1}^2}
\newcommand{\delodd}[1]{\Es{\s}(\epsilon_3^2+\efour{\s}^2)(\epsilon_3+\efour{\s}^2)+ d^{-1/2}\efive{\s}^3}

\newcommand{\Dir}{\mathrm{Dir}}
\newcommand{\E}{\mathbb E\,}
\newcommand{\ethree}[1]{\epsilon_{3}(#1)}
\newcommand{\efour}[1]{\epsilon_{4}(#1)}
\newcommand{\efive}[1]{\epsilon_{5}(#1)}
\newcommand{\ek}[2]{\epsilon_{#1}(#2)}
\newcommand{\Es}[1]{E_{#1}}

\renewcommand{\H}{{\mathbf{H}}}
\newcommand{\I}{{\mathcal I}}
\newcommand{\iid}{\stackrel{\mathrm{i.i.d.}}{\sim}}
\newcommand{\ind}{\mathbbm{1}}
\renewcommand{\l}{\left}
\newcommand{\la}{\langle}
\newcommand{\lla}{\l\la}
\newcommand{\lap}{\hat\gamma}
\newcommand{\les}{\lesssim}
\newcommand{\Lhat}{S}
\newcommand{\Lhatb}{\delta\bhat}
\newcommand{\Lhatm}{\delta\mhat}
\newcommand{\Lhatth}{\delta\hat\theta}

\newcommand{\LTV}{L_{\mathrm{TV}}}
\newcommand{\M}{M}
\newcommand{\mhat}{\hat x}
\newcommand{\mhatS}{\mhat_{\Lhat}}
\newcommand{\MLE}{\hat b}
\newcommand{\mpi}{\bar x}
\newcommand{\nb}{p}
\newcommand{\Nmin}{{\nb_{\mathrm{min}}}}

\newcommand{\PP}{\mathbb P}
\newcommand{\prior}{\pi_0}
\newcommand{\Q}{{\mathbf{Q}}}
\renewcommand{\r}{\right}
\newcommand{\R}{\mathbb R}
\newcommand{\ra}{\rangle}
\newcommand{\rra}{\r\ra}

\newcommand{\RTV}{R_{\mathrm{TV}}}

\newcommand{\rfour}{r_4}
\newcommand{\s}{s}

\newcommand{\Sigpi}{\Sigma_\pi}
\newcommand{\sst}{\s^*}
\newcommand{\T}{{\intercal}}
\newcommand{\taus}[1]{\tau_{#1}}
\newcommand{\tilep}{\bar\epsilon_3}

\newcommand{\Unif}{{\mathrm{Unif}}}
\newcommand{\Var}{{\mathrm{Var}}}

\newcommand{\ground}{\beta}

\newcommand{\KL}[2]{\mathrm{KL}(\,#1\;||\;#2)}
\newcommand{\e}{\,\mathrm{exp}}

\newcommand{\U}{{\mathcal U}}

\newcommand{\TV}{\mathrm{TV}}

\title{The Laplace approximation accuracy in high dimensions: a refined analysis and new skew adjustment}
\author{
 Anya Katsevich\thanks{This work was supported by NSF grant DMS-2202963}\\
  \texttt{akatsevi@mit.edu}
 }

\begin{document}

\maketitle

\begin{abstract}
In Bayesian inference, making deductions about a parameter of interest requires one to sample from or compute an integral against a posterior distribution. A popular method to make these computations cheaper in high-dimensional settings is to replace the posterior with its Laplace approximation (LA), a Gaussian distribution. In this work, we derive a leading order decomposition of the LA error, a powerful technique to analyze the accuracy of the approximation more precisely than was possible before. It allows us to derive the first ever skew correction to the LA which provably improves its accuracy by an order of magnitude in the high-dimensional regime. Our approach also enables us to prove both tighter upper bounds on the standard LA and the first ever lower bounds in high dimensions. In particular, we prove that $d^2\ll n$ is in general necessary for accuracy of the LA, where $d$ is dimension and $n$ is sample size. Finally, we apply our theory in two example models: a Dirichlet posterior arising from a multinomial observation, and logistic regression with Gaussian design. In the latter setting, we prove high probability bounds on the accuracy of the LA and skew-corrected LA in powers of $d/\sqrt n$ alone.
\end{abstract}

\section{Introduction}

Developing cheap and accurate computational techniques for Bayesian inference is an important goal, as Bayesian inference tasks can be very computationally intensive. These tasks include constructing posterior credible sets, computing the posterior mean and covariance, and computing the posterior predictive distribution of a new data point. Computing all of these quantities involves either sampling from the posterior $\pi$, or taking integrals $\int gd\pi$ against the posterior. When the dimensionality of the parameter of interest is large, these tasks can be very expensive.

The most common approach in Bayesian inference to both generate samples from, and integrate against $\pi$ is Markov Chain Monte Carlo (MCMC)~\cite{bolstad2009understanding}. In principle, MCMC schemes can be made arbitrarily accurate by tuning parameters, e.g.  by extending the simulation time of the Markov chain or by decreasing step sizes. However, MCMC is computationally intensive in high dimensions and it also has other disadvantages. For example, it can be difficult to identify clear-cut stopping criteria for the algorithm~\cite{mcmc-convergence}.

Another popular approach is to find a simple distribution $\lap$ which approximates $\pi$, and to use this distribution as a proxy for $\pi$ to do all of one's inference tasks. In the ideal scenario, many integrals against $\lap$ are computable in closed form, and it is cheap to sample from $\lap$. The idea of using an approximation $\lap$ to $\pi$ is at the heart of approximate Bayesian inference methods such as variational inference~\cite{blei2017variational,wainwright2008graphical}, expectation propagation~\cite{minka2001expectation}, and the Laplace approximation, a Gaussian approximation first introduced in the Bayesian inference context by~\cite{tierney1986accurate}.

The  Laplace approximation (LA) exploits large sample properties of the posterior which have been established in the \emph{Bernstein von-Mises} theorem (BvM). The BvM is a fundamental result which in its classical form states that if the model is well-specified, then the posterior contracts around the ground truth parameter and becomes asymptotically normal in the large sample limit~\cite[Section 10.2]{vaart_1998}. That the posterior contracts as sample size increases is intuitively clear, since uncertainty in the parameter value diminishes as samples accumulate. From this concentration result, an informal mathematical argument also clearly explains why the posterior is asymptotically Gaussian. We discuss this below. Extensions of the BvM in the misspecified case and for generalized posteriors have also been obtained; see~\cite{kleijn2012bernstein} and~\cite{miller2021asymptotic}. Despite the theoretical and philosophical importance of the BvM (it has been used as a justification of Bayesian procedures from the perspective of frequentist inference), this result does not give an implementable Gaussian approximation to the posterior. This is where the LA comes in. 

To explain the LA construction, consider a posterior $\pi$ whose mass concentrates in a small neighborhood of the mode, which we call $\mhat$. When the conditions of the BvM are satisfied, this highest mode should be unique; otherwise, concentration cannot occur. Since most of the mass of $\pi$ is near $\mhat$, we should incur only a small error by replacing the log posterior with its second order Taylor expansion about $\mhat$. This gives rise to the LA, the Gaussian density $\lap$ which is given by
\beq\label{lap-def}
\lap=\mathcal N\l(\mhat,\;\nabla^2V(\mhat)^{-1}\r),\quad \mhat=\arg\min_{x\in\Theta}V(x)
\eeq where $\pi\propto e^{-V}$ is a density on $\Theta\subseteq\R^d$. The LA has proved to be an invaluable tool for Bayesian inference in applications ranging from deep learning~\cite{daxberger2021laplace} to inverse problems~\cite{stadler2012extreme} to variable selection in high-dimensional regression~\cite{barber2016laplace}.

Unlike methods such as MCMC, which can be made arbitrarily accurate, there are no parameters to tune in the LA --- it incurs some fixed approximation error. Quantifying the LA's error as a function of dimension $d$, sample size $n$, and model parameters, is a worthy task given its widespread use. It is also a challenging theoretical endeavor when dimension $d$ is large, and currently a very active research area. Major contributions have been made e.g. by~\cite{bp,spok23,helin2022non}, and we discuss these and other works in more detail below. We also study the LA error in this work.

But arguably, it is even more important to go \emph{beyond} the LA to develop new, more accurate approximations which better capture the complexity of the posterior $\pi$. For example, a known downside of the LA $\lap$ is that it is symmetric about the mode and therefore cannot capture skewness of $\pi$. Instead of constructing an entirely new kind of approximation, a natural idea is to correct the LA in some way to get a higher-order accuracy approximation. Non-rigorous skew normal approximations have been developed in~\cite{zhou2024tractable}, but we are aware of only a single work~\cite{durante2023skewed} that rigorously derives a higher-order accurate LA. We discuss this work in more detail in Section~\ref{subsec:discuss}. However, the corrected LA obtained by~\cite{durante2023skewed} is only shown to be accurate in constant dimension $d$. In modern applications involving very high-dimensional parameters, $d$ cannot be considered constant relative to sample size $n$. So far, no prior work has obtained a higher-accuracy LA which is rigorously justified in high dimensions.

In this work, we develop a powerful technique to analyze the Laplace approximation more precisely than was possible before. This technique leads us to derive the \emph{first ever correction to the LA which provably improves its accuracy by an order of magnitude, in high dimensions}. Our approach allows us to prove error bounds on this corrected LA in terms of a variety of error metrics discussed below. It also enables us to prove both lower bounds and tighter upper bounds on the standard LA. \\

\subsection{Main contributions: higher-order accuracy LA}\label{intro:main}
As stated above, many quantities used in Bayesian inference can be written as integrals $\int gd\pi$ of functions $g$ against the posterior $\pi\propto e^{-V}$. This motivates studying the approximation of posterior integrals. From here, the study of approximation to $\pi$ with respect to integral probability metrics such as total variation (TV) distance follows naturally. We derive the following expansion of $\int gd\pi$:
\beq\label{secord}\medint\int gd\pi = \medint\int gd\lap + \medint\int g\Lhat d\lap + R(g)=\medint\int gd\corlap + R(g),\eeq
 where $\Lhat$ is an explicit cubic polynomial, and $$\corlap:=(1+\Lhat)\lap$$ is our higher-order accurate LA. We now describe our results pertaining to this expansion, as well as the convenient computational features of $\corlap$. 

To do so, we first note that up to constants, the negative log posterior $V$ of $\pi\propto e^{-V}$ is given by the sum of a negative log prior $v_0$ and a negative log likelihood, which we denote $n\ell$. The function $n\ell$ is typically given by the sum of $n$ terms, corresponding to $n$ independent samples. Alternatively, $n\ell$ could be the data misfit function in Bayesian inverse problems, with $n$ denoting inverse noise level. Due to this structure, we write $V=nv$, where $v=\ell+n^{-1}v_0$.

 
 \paragraph*{1. Error bounds for $\bm\corlap$} We prove that if $g$ has log linear growth at infinity and is appropriately centered and scaled, then
 \begin{align}
\l|\medint\int gd\pi - \medint\int gd\corlap\r| &\les b_4(v)\l(\frac{d}{\sqrt n}\r)^2,\label{Del-g-2}\\
  \l|\medint\int g\Lhat d\lap\r| &\les b_3(v)\frac{d}{\sqrt n},\label{L-g}\end{align} and therefore
  \begin{align}
\l|\medint\int gd\pi - \medint\int gd\lap\r| &\les b_3(v)\frac{d}{\sqrt n}+b_4(v)\l(\frac{d}{\sqrt n}\r)^2,\label{Del-g-1}
\end{align} up to exponentially small terms. Here, ``$\les$'' means up to an absolute constant. See Theorem~\ref{thm:Vgen} for the precise statement. We refer to $b_3$ and $b_4$ as \emph{model-dependent factors}, while the terms $d/\sqrt n$ and $(d/\sqrt n)^2$ are the~\emph{universal} factors in our error bounds. The model-dependent factors are given in terms of third and fourth derivatives of $v$ in a neighborhood of the mode $\mhat$. 

Regarding the error bound~\eqref{Del-g-1} for the original LA, we note that similar bounds in the literature have been shown for a much more limited class of functions, or for a specific choice of model, or with suboptimal dimension dependence. We cite and discuss this literature in Section~\ref{intro:further}. For now, we wish to emphasize~\eqref{Del-g-2}, and the superior performance of $\corlap$ compared to that of $\lap$ in~\eqref{Del-g-1}. In fact, when $g$ is odd about $\mhat$, we do some extra work to show that $\corlap$ yields an even better approximation. Namely, we show in Theorem~\ref{thm:odd} that up to exponentially negligible terms, it holds
\beq
\l|\medint\int gd\pi -\medint \int gd\corlap\r| \les b_5(v)\l(\frac{d}{\sqrt n}\r)^3,\label{Del-g-3}\eeq where $g$ has again been normalized. The term $b_5$ involves the fifth derivative of $v$.

Although our formulation of the results suggests that $d/\sqrt n$ is a fundamental unit of error, we caution against this interpretation. This is because the model-dependent terms can in fact depend on dimension. We discuss this further below.

\paragraph*{2. Important special cases: mean and TV distance} We use~\eqref{Del-g-2} and~\eqref{Del-g-3} to prove bounds on the accuracy of $\corlap$ in approximating the mean, and in approximating set probabilities. Namely, we show that
\beqs\label{TV-mean}
\sup_{A\in\mathcal B(\R^d)}|\pi(A)-\corlap(A)|&\les b_4(v)\l(\frac{d}{\sqrt n}\r)^2,\\
 \l\|\medint\int xd\pi(x) -\medint\int xd\corlap(x)\r\|_{H_V} &\les  b_5(v)\l(\frac{d}{\sqrt n}\r)^3,
\eeqs
where $\|\cdot\|_{H_V}$ is the Mahalanobis norm, with weight matrix $H_V=\nabla^2V(\mhat)=n\nabla^2v(\mhat)$. This weighted norm should be thought of as a kind of standardization. See Corollaries~\ref{corr:corTV} and~\ref{thm:Vmean} for more precise statements. The above approximation accuracy is superior to that of $\lap$. Indeed, the work~\cite{bp} has shown that both $\TV(\pi,\lap)$ and the (weighted) distance between the means of $\pi$ and $\lap$, are bounded above by $d/\sqrt n$ up to model-dependent factors. Note that the mean of $\lap$ is simply $\mhat$, the mode of $\pi$, also known as the MAP (maximum aposteriori).

\paragraph*{3. Computability of integrals against $\bm\corlap$} A higher-order accuracy approximation is only useful if it is simple to use. Integrals $\int gd\corlap=\int g(1+\Lhat)d\lap$ can always be computed using Monte Carlo, by drawing i.i.d. samples $X_i\sim\lap$, and averaging $g(X_i)(1+\Lhat(X_i))$. Even more simply, the integrals of \emph{polynomials} $g$ against $\corlap$ are known in closed form. This is because Gaussian moments are known in closed form, and $g(1+\Lhat)$ is a polynomial, since both $g$ and $\Lhat$ are polynomials. Functions $g(x)=e^{u^\T x}q(x)$ for a vector $u$ and polynomial $q$ can also be integrated against $\corlap$ in closed form. Having a closed form expression is extremely valuable, because it means we do not have to resort to sampling from $\lap$ to compute the integral, and therefore we do not incur the additional error that arises from Monte Carlo integral approximations. 

As an important application of this, consider the function $g(x)=x$, integrating which yields the mean. We derive the following simple formula for the mean of $\corlap$:
\beq\label{mhatcorrect}\int xd\corlap(x) =\mhat -\frac12H_V^{-1}\lla\nabla^3V(\mhat),H_V^{-1}\rra,\eeq where $H_V=\nabla^2V(\mhat)$. As we see from~\eqref{TV-mean} and the subsequent discussion, computing this straightforward expression instead of using $\mhat$ reduces the mean approximation error by \emph{two orders of magnitude}, from $\mathcal O(d/\sqrt n)$ to $\mathcal O((d/\sqrt n)^3)$. 

The fact that $\corlap$ provides such an excellent approximation to the mean has prompted us to interpret this measure as a \emph{skew correction} to $\lap$; hence the letter `S'. Indeed, a crucial shortcoming of the LA $\lap$ is that it is symmetric about its mode $\mhat$, while $\pi$ itself may be skewed. The skew shifts the mean away from the mode. Thus the LA to the mean is poor when $\pi$ is skewed, but the corrected approximation~\eqref{mhatcorrect} compensates by shifting the mode in the direction of the skew. 

\subsection{Further contributions: refined bounds on uncorrected LA}\label{intro:further}
There is a natural tradeoff between accuracy and computational complexity: although the corrected LA improves on the standard LA in terms of accuracy, the latter is simpler to use. Indeed, the standard LA is so widespread in practice due to this simplicity. It is therefore extremely important to precisely characterize the error incurred by the standard LA. As we will see below, our expansion $\int gd\pi\approx \int gd\lap + \int g\Lhat d\lap$ allows us to do exactly this. 

At this point, we note that we limit our error analysis to the study of the following types of errors: 1) $\Delta_g(\lap)=\int gd\pi - \int gd\lap$ for general $g$, 2) the TV distance between $\pi$ and $\lap$, 3) the mean approximation error, and 4) the covariance approximation error. There are many other notions of distance, including various integral probability metrics, other $f$-divergences, and Wasserstein distance, which we omit for brevity. In the rest of the introduction, we primarily focus on errors 1) and 2), both in our literature review and discussion of our own work.



Before stating our results, we first review prior works on the uncorrected LA in high dimensions. Typically, bounds on the LA error stated in such works take the form $\text{Err}(\pi,\lap)\leq c(v)d^\alpha/n^\beta$. As above, we call $c(v)$ the ``model-dependent'' factor in the error, and $d^\alpha/n^\beta$ the ``universal'' factor. The fact that the model enters into the error only through the function $v$ is natural, since the posterior depends on the model (i.e. on the likelihood, prior, and data) only through $v$. It is important to note that the coefficient $c(v)$ could itself depend on $d$, affecting the overall dimension dependence of the error. For example, consider a Bayesian inverse problem in which $d$ is the discretization level of an underlying infinite-dimensional forward operator. The log likelihood and therefore $v$ clearly depend on $d$ in an essential way through the discretized forward operator.

The works~\cite{helin2022non, dehaene2019deterministic,spok23, fischer2022normal} all bound the Laplace error in high dimension. For simplicity we only report the TV bounds in these works. 
Each of the first three works obtain (under slightly varying conditions) that the TV distance is bounded as $\TV(\pi,\lap)\leq c(v)\sqrt{d^3/n}$, where $c(v)$ involves the operator norm of the third order derivative tensor $\nabla^3v$. (To be precise,~\cite{spok23} actually states bounds in terms of an effective dimension $d_{\text{eff}}\leq d$. We discuss this work in Section~\ref{subsub:tight}.) Meanwhile,~\cite{fischer2022normal} considers the setting of exponential families, and obtains bounds which scale with a higher power of $d$. 

The requirement $d^3\ll n$ (up to model dependence) was thought to be necessary for asymptotic normality of the posterior. Indeed, in addition to the above works on the Laplace approximation, a long line of works on the BvM in increasing dimension required $d^3\ll n$, e.g.~\cite{ghosal2000,spokoiny2013bernstein,boucheron2009discrete,belloni2014posterior, lu2017bernstein}. 

Remarkably, however,~\cite{bp} recently tightened the dimension dependence of the Laplace approximation to \beq\label{TV-bp}\TV(\pi, \lap)\leq c(v)(d^2/n)^{1/2},\qquad c(v)=\sup_{x\in\U}\|\nabla^3v(x)\|_{H_v},\eeq where $\U$ is a small neighborhood of $\mhat$ and $\|\cdot\|_{H_v}$ is an operator norm weighted by the matrix $H_v=\nabla^2v(\mhat)$. 
We note that some manipulations are required to bring the TV bound presented in~\cite{bp} into the form~\eqref{TV-bp}. This is explained in the work~\cite{katsBVM}, which improves on the TV bound of~\cite{bp} in the case that $v$ grows at least linearly at infinity. See Section~\ref{subsub:tight} for more on this.\\

\noindent We now discuss our own contributions to the understanding of the uncorrected LA.
\paragraph*{1. Bounds on $\mathbf{\Delta_g(\bm\lap):=\int gd\bm\pi-\int gd\bm\lap}$}
We have already presented our bound on $\Delta_g(\lap)$ in~\eqref{Del-g-1}. As noted in Section~\ref{intro:main}, previous bounds in the literature on $\Delta_g(\lap)$ for general $g$ in high dimensions are limited and suboptimal. The work~\cite{bp} bounds $\Delta_g(\lap)$ in dimension $d=1$, and~\cite{fischer2022normal} bounds $\Delta_g(\lap)$ in the case that the data distribution lies in an exponential family. The bounds obtained are very suboptimal in their dimension dependence, as noted in their Remark 3. The work~\cite{lapinski2019multivariate} has proved bounds in arbitrary dimension on a quantity related to but nevertheless different from $\Delta_g(\lap)$.\\ 

Our other contributions stem from a refined understanding of the TV error and mean approximation error of the uncorrected LA which can be derived from the expansion~\eqref{secord}. For brevity, we only describe the results for the TV error in the introduction. Similar results on the mean approximation error are described in more detail in Sections~\ref{subsub:mean} and~\ref{subsec:lb}.

Taking the supremum over $\|g\|_\infty\leq1$ in~\eqref{secord}, and using the bounds~\eqref{Del-g-2} and~\eqref{L-g}, we obtain the following TV decomposition and upper bounds:
\beq\label{LRTV}
\begin{gathered}
\TV(\pi,\lap)=\LTV+\RTV,\qquad \LTV = \frac12\medint\int|\Lhat|d\lap,\\
\LTV\leq \|\Lhat\|_{L^2(\lap)}\leq\, b_3(v)\,\frac{d}{\sqrt n},\qquad |\RTV|\les\, b_4(v)\,\l(\frac{d}{\sqrt n}\r)^2.
\end{gathered}\eeq 
Because our upper bound on $\RTV$ is an order of magnitude smaller than our upper bound on $\LTV$, we informally call $\LTV$ the ``leading order'' term and $\RTV$ the ``remainder term''. (Similarly, we call $L(g):=\int g\Lhat d\lap$ and $R(g)$ the leading order and remainder terms of~\eqref{secord}, respectively.) This is an abuse of terminology, since we have not proved $\RTV\ll \LTV$.

The decomposition~\eqref{LRTV} leads us to the following three implications.

\paragraph*{2. Lower bounds} The TV decomposition~\eqref{LRTV} gives a general framework to prove lower bounds; namely, we need only bound $\LTV$ from below and $\RTV$ from above. We have already provided a formula for the upper bound on $\RTV$. Bounding $\LTV$ from below, while still challenging, is vastly simpler than directly bounding the TV distance from below. We implement this strategy to prove a lower bound in Section~\ref{subsec:lb}. Namely, we give an example of a function $v$ for which the leading order term $\LTV$ in the TV error is bounded from \emph{below} by $d/\sqrt n$ and the remainder $\RTV$ is bounded above by $(d/\sqrt n)^2$. Together, the lower and upper bound imply $\TV(\pi,\lap)\gtrsim d/\sqrt n$. This definitively resolves the question of the tight dimension dependence of the LA: we have shown that the condition $d^2\ll n$ \emph{cannot in general be relaxed}. 
 We also prove another lower bound for a multinomial model; see Sections~\ref{intro:pmf} and~\ref{sec:pmf} for more on this.

\paragraph*{3. Tighter upper bounds} Although we have shown that $d^2\ll n$ is in general necessary for TV accuracy of the LA, this does not mean the TV bound~\eqref{TV-bp} of~\cite{bp} is tight. Indeed, if the model-dependent factor $c(v)$ is too coarse, it could contribute additional powers of $d$. Here, we use our decomposition to bound the TV distance via $|\LTV|+|\RTV|$, which allows us to prove tighter upper bounds. We can do this because the remainder $\RTV$ is higher-order than previously derived bounds, and the leading order term $\LTV$ is explicit, and can therefore be much more precisely controlled. Our $b_3(v)$ is given by $b_3(v)=\|\nabla^3v(\mhat)\|_{H_v}$. Therefore, the second bound on $\LTV$ in~\eqref{LRTV} is already tighter than $c(v)$ from~\eqref{TV-bp}, since we do not need to take the supremum over $x\in\U$. But in fact, we can sometimes do even better. Namely, using the tighter inequality $\LTV\leq \|\Lhat\|_{L^2(\lap)}$, we can improve the overall bound by \emph{an order of magnitude}, as we show in Section~\ref{sec:pmf}; see also the discussion in Section~\ref{subsub:tight}.
\paragraph*{4. Computability of the TV error} Computing an upper bound on the TV error $\TV(\pi,\lap)$ of the Laplace approximation is important in practice, and a challenging problem in high dimensions. All the bounds we are aware of either have very suboptimal dimension dependence or else involve an operator norm as in the bound~\eqref{TV-bp} of~\cite{bp}. The computational complexity of computing an operator norm scales exponentially with dimension. See Section~\ref{subsub:compute} for references and more details. Here, we take an important step towards solving this problem by providing a bound on the leading order term $\LTV$ which is computable in polynomial time.

\subsection{Illustration with a multinomial model}\label{intro:pmf}
In Section~\ref{sec:pmf}, we consider a tractable example which clearly demonstrates the improved understanding of the uncorrected LA which can be gleaned from the leading order decomposition~\eqref{secord}. Namely, we take a multinomial likelihood and conjugate Dirichlet prior, so that the posterior is also Dirichlet, which is relatively simple to analyze. The example has the convenient and rare property that many quantities can be explicitly computed or at least very precisely bounded. This applies both to true quantities such as the posterior mean, and quantities arising in our upper bounds, such as operator norms of the derivatives of $v$.

We first study the decomposition $\TV(\pi,\lap)=\LTV+\RTV$ of the TV error from~\eqref{LRTV}. We explicitly work out the coefficients appearing in our bounds on $\LTV$ and $\RTV$, in terms of the probability mass function of empirically observed frequencies. The resulting TV bound is an order of magnitude smaller than that of~\cite{bp}, due in large part to the fine control on $\LTV$. This demonstrates point 3 of Section~\ref{intro:further} above. We then prove a lower bound on the TV distance, which scales as $d/\sqrt n$ when the observed frequencies are not too close to uniform. The lower bound cleanly demonstrates the role of $d$, $n$, and model specific factors standing in the way of accuracy of the LA. Moreover, in certain regimes the lower bound matches the upper bound on $\LTV$. 

We then study the decomposition of the mean $\bar x:=\int xd\pi(x)$ obtained by substituting $g(x)=x$ into~\eqref{secord}. We can write $\mpi = \mhat + \Lhatm + R_{\mpi}$, where $R_{\mpi}$ is defined implicitly by this equation, and $\Lhatm = \int x\Lhat(x)d\lap(x)$. In this example all four vectors $\mpi,\mhat,\Lhatm,R_{\mpi}$ are known exactly. We compute the true norms of the vectors $\Lhatm$ and $R_{\mpi}$ and compare them to our bounds. We find that our bound on $\Lhatm$ is tight, while our bound on $R_{\mpi}$ is very coarse.

Another nice feature of this example is that both our upper and lower bounds, as well as the true values of various quantities, are especially illuminating. They convey the essence of the way in which the posterior deviates from its Gaussian Laplace approximation.

\subsection{Application to logistic regression}
We next apply our theory to the posterior in a logistic regression model with a flat or Gaussian prior on the coefficient vector. We demonstrate the power of our leading order decomposition using a combination of theory and numerical simulation.

First, we work out the explicit formulas for the skew-corrected LA $\corlap$, the mean correction~\eqref{mhatcorrect}, and the polynomial-time computable upper bound on $\LTV$, in terms of the data and the known quantities in the model, such as the sigmoid function. We then numerically demonstrate the superior performance of $\corlap$ over $\lap$ in approximating set probabilities under $\pi$ and the mean of $\pi$.  Next, we prove high-probability versions of our upper bounds on the uncorrected and corrected LA errors which are explicit up to absolute constants. Finally, we combine a rigorous upper bound on $\RTV$ with a numerical lower bound on $\LTV$ to show that $\RTV\ll\LTV$ and $\LTV\asymp\TV(\pi,\lap)\gtrsim d/\sqrt n$ with high probability. This semi-numerical argument also justifies using the bound on $\LTV$ as a proxy for a bound on $\TV(\pi,\lap)$.


We now highlight our high probability bounds, and explain the need for such bounds.

\subsubsection{High probability bounds}
In order to apply our theory to obtain bounds on $\int gd\pi-\int gd\lap$ and $\int gd\pi - \int gd\corlap$ for the logistic regression model, we need to quantify the model-dependent factors such as $b_3(v)$ and $b_4(v)$ from~\eqref{Del-g-1},~\eqref{Del-g-2}. Unlike in the multinomial model, this is a very complex task for logistic regression. For example, in the case of a flat prior, computing $b_4(v)$ requires computing
\beq\label{Xi}\sup_{u\neq0}\frac{\frac1n\sum_{i=1}^n\sigma'''(X_i^{\T}\mhat)(X_i^{\T}u)^4}{\l(\frac1n\sum_{i=1}^n\sigma'(X_i^{\T}\mhat)(X_i^{\T}u)^2\r)^{2}},\eeq where the $X_i$ are the feature vectors, $\mhat$ is the posterior mode (MAP), and $\sigma(t)=(1+e^{-t})^{-1}$ is the sigmoid function. It is unlikely that such a quantity can be either efficiently computed or tightly analytically bounded for a given set of observed feature vectors $X_i$. 

However, if we can bound the terms $b_3(v),b_4(v)$ \emph{with high probability}, then we can still obtain useful information from the bounds~\eqref{Del-g-1} and~\eqref{Del-g-2}. We assume that the $X_i$ are drawn i.i.d. from a Gaussian distribution
and prove that there exists an absolute constant $C$ such that with high probability (w.h.p.) over the joint feature-label distribution, all of the model-dependent factors are bounded by $C$. In particular,~\eqref{Xi} is also bounded by $C$ w.h.p. This allows us to derive high-probability bounds depending only on absolute constants, $d$, and $n$. For example, by taking $g(x)=x$ or $g(x)=\ind_A$ in~\eqref{Del-g-1} and~\eqref{Del-g-2}, we show the following bounds all hold w.h.p. (up to an exponentially negligible term):
\beqs\label{logbd}
\TV(\pi,\lap) &\leq C\epsilon,\qquad \TV(\pi,\corlap)\leq C\epsilon^2,\\
\sqrt n\l\|\mpi-\mhat\r\|&\leq C\epsilon,\qquad \sqrt n\l\|\mpi-\l(\mhat+\Lhatm\r)\r\|\leq C\epsilon^{3}.
\eeqs 
Here, $\epsilon=d/\sqrt n$ and $\mhat+\Lhatm$ is the corrected mean approximation $\int xd\corlap(x)$ from~\eqref{mhatcorrect}. See Corollary~\ref{corr:logreg} for the formal statement. The bounds~\eqref{logbd} are extremely useful because they allow us to evaluate the typical order of magnitude of the approximation error, and in particular, to conclude that the model-dependent terms \emph{do not contribute a power of $d$}. This information is very hard to obtain by trying to directly bound the $v$-dependent terms for a particular draw of the data, since doing so would require solving the optimization problem~\eqref{Xi}.\\

We note that the high-probability bounds~\eqref{logbd} are similar to the BvM, in that both treat the posterior as a random object, due to the data being randomly drawn from some distribution. In fact,~\cite{dehaene2019deterministic} uses the term BvM to refer to high-probability Laplace bounds. See our earlier work~\cite{katsBVM} for a (true) BvM in the regime $d^2\ll n$ for the two examples considered here: logistic regression and the multinomial model.

\subsection{Other related work}\label{intro:other}
The first natural tool at our disposal to analyze the LA is the Laplace \emph{method}, which the LA is based on. The Laplace method gives a whole asymptotic expansion of integrals of the form $\int ge^{-nv}$; see~\cite[Chapter 9.5]{wongbook}. However, it does not immediately yield a bound on e.g. the TV distance between $\pi\propto e^{-nv}$ and its Laplace approximation. See~\cite{schillings2020convergence} for a thorough summary of the limitations of using the Laplace method to quantify the LA error. This work is the first to bound the TV and Hellinger errors of the LA. The authors show that $\TV(\pi,\hat\gamma),H(\pi,\hat\gamma)\sim n^{-1/2}$, but consider dimension to be constant. 

Another work related to ours, but in constant dimension, is that of~\cite{bilodeau2023tightness}. There, the authors proves lower bounds on the error of the Laplace approximation to the normalizing constant for a few illustrative statistical models.

Our decomposition~\eqref{secord} is akin to a first order asymptotic expansion. Several works have derived full-blown asymptotic expansions of posterior distributions, or of integrals agains a posterior. Note that this is different from the expansion provided by the Laplace method due to the posterior normalization. In this vein, a similar work to ours in one dimension is~\cite{johnson1970asymptotic}. There, the author derives an expansion of the posterior cdf in powers of $n^{-1/2}$, as well as an expansion of posterior moments. The coefficients of the expansion are given in terms of derivatives of the log likelihood and log prior. See also~\cite{ghosh1982expansions} and~\cite{bertail2002johnson} for slight variations of these results.

Expansions of posterior integrals in high dimensions are considered in~\cite{shun1995laplace}. There, the authors derive formal series expansions of the normalizing constant $\int_{\R^d} e^{-nv(x)}dx$, where the coefficients again depend on derivatives of the log posterior. 

We mention that Edgeworth-type expansions of posterior integrals have also been obtained; see~\cite{weng2010bayesian,kolassa2020validity}. The coefficients in these expansions are expressed in terms of moments or cumulants of the posterior. Therefore, applying the expansion requires the ability to integrate certain functions against the posterior. 

Finally, we note that the most basic assumption in the analysis of the LA is that there is a unique global minimizer of $v$. The work~\cite{hasenpflug2022wasserstein} goes beyond this to study the convergence of $\pi\propto e^{-nv}$ to its limiting measure in the case of many (potentially infinite) global minima.

\paragraph*{Organization}The rest of the paper is organized as follows. In Section~\ref{sec:V}, we state our assumptions and main results, centering around the decomposition~\eqref{secord}. Section~\ref{sec:discuss} contains further discussion of our results. Section~\ref{sec:pmf} studies the multinomial model, and Section~\ref{sec:log} studies the logistic regression model. Finally in Section~\ref{sec:overview}, we outline the proof of the decomposition. Deferred proofs can be found in the Appendix.\\

\paragraph*{Notation} The notation $a\les b$ means there is an absolute constant $C>0$ such that $a\leq Cb$. The letter $C$ always denotes an absolute constant. The symbol $\ind$ denotes the indicator function. A tensor $T$ of order $k$ is an array $T=(T_{i_1i_2\dots i_k})_{i_1,\dots,i_k=1}^d$. For two order $k$ tensors $T$ and $S$ we let $\la T, S\ra$ be the entrywise inner product. We say $T$ is symmetric if $T_{i_1\dots i_k}= T_{j_1\dots j_k}$, for all permutations $j_1\dots j_k$ of $i_1\dots i_k$. 

Let $H$ be a symmetric positive definite matrix. For a vector $x\in\R^d$, we let $\|x\|_H$ denote $\|x\|_H = \sqrt{x^\T Hx}$.  For an order $k\geq2$ symmetric tensor $T$, we define the $H$-weighted operator norm of $T$ to be
\beq\label{TH}\|T\|_H:=\sup_{\|u\|_H=1}\la T, u^{\otimes k}\ra = \sup_{\|u\|_H=1}\sum_{i_1,\dots,i_k=1}^dT_{i_1i_2\dots i_k}u_{i_1}u_{i_2}\dots u_{i_k}\eeq  By Theorem 2.1 of~\cite{symmtens}, for symmetric tensors, the definition~\eqref{TH} coincides with the standard definition of operator norm:
$$\sup_{\|u_1\|_H=\dots=\|u_k\|_H=1} \la T, u_1\otimes\dots\otimes u_k\ra = \|T\|_H=\sup_{\|u\|_H=1}\la T, u^{\otimes k}\ra.$$ 
When $H=I_d$, the norm $\|T\|_{I_d}$ is the regular operator norm, and in this case we omit the subscript. For a symmetric, order 3 tensor $T$ and a symmetric matrix $A$, we let $\la T, A\ra\in\R^d$ be the vector with coordinates
\beq\label{laTA}
\la T, A\ra_i=\sum_{j,k=1}^dT_{ijk}A_{jk},\quad i=1,\dots,d.
\eeq

\section{Main Results}\label{sec:V}
In Section~\ref{assumptions}, we discuss the form of the log posterior $V=nv$ in Bayesian inference and give some concrete examples. We then state our assumptions on $v$, $n$, and $d$. In Section~\ref{sec:V:epsilon}, we introduce the key quantities in terms of which our bounds are stated, and which should be small for these bounds to be useful. In Section~\ref{sec:V:results}, we state our main results on the error of the skew-corrected LA, and the leading order decomposition of the uncorrected LA error. In Section~\ref{subsec:lb}, we present our lower bounds. Omitted proofs from this section can be found in Appendix~\ref{app:main}.

\subsection{Setting and assumptions on the potential}\label{assumptions}We consider a general setting in Bayesian inference that includes posteriors and generalized posteriors more broadly; this perspective is similar to that of~\cite{bp} on the Laplace approximation, and~\cite{miller2021asymptotic} on the BvM. Namely, let
$$\pi(x)\propto e^{-n\ell(x)}\prior(x),\quad x\in\Theta\subset\R^d.$$ Here, $\prior$ is a \emph{prior} on the parameter of interest $x$, and $L:=e^{-n\ell}$ is a \emph{generalized likelihood}. We assume without loss of generality that the support of $\prior$ is $\Theta$; otherwise, redefine $\Theta$ to be its intersection with the support of $\prior$. Hence for some function $v_0$, we can write $\prior$ as $\prior(x)\propto e^{-v_0(x)}$ for all $x\in\Theta$. Correspondingly, $\pi$ can be written as $\pi\propto e^{-n\ell(x)-v_0(x)}=e^{-nv(x)}$, where $$v(x)=\ell(x)+\frac1nv_0(x),\quad x\in\Theta.$$ We also let $V(x)=nv(x)=n\ell(x)+v_0(x)$. Of course, the accuracy of the Laplace approximation to $\pi\propto e^{-V}$ does not depend on how $V$ is decomposed, and none of $n,v,\ell, v_0$ is uniquely determined by the equation $V=nv=n\ell + v_0$. In line with this, our assumptions and results can all be restated in terms of the function $V$ alone. However, to strike a balance between interpretability and generality, we have chosen to present them in terms of $n$ and $v$ separately. For further clarity, we begin by discussing the form of $\ell$ in a few typical settings which are covered by our theory.
\begin{example}[Classical setting]\label{ex:standard}
In the classical setting of Bayesian inference, we observe independent data $y_i, i=1,\dots, n$, and model the distribution of $y_i$ as $y_i\sim p_i(\cdot\mid x)$. In the most standard case of i.i.d. data we have $p_i=p$ for all $i$; however, this will not be required in our assumptions. The likelihood is then given by $L(x)=\prod_{i=1}^np_i(y_i\mid x)$. To write $L=e^{-n\ell}$ for some function $\ell:\Theta\to\R$, it must hold that $L$ is positive on $\Theta$. This occurs if for example the support of $p_i(\cdot\mid x)$ is the same for all $x\in\Theta$, and if $y_i\sim p_i(\cdot\mid x^*)$ for some true $x^*$. However this is not the only condition which ensures the likelihood is positive and in general, we do \emph{not} need to assume the model is well-specified. 

In this setting, the function $\ell$ is given by
$$\ell(x) = -\frac1n\log L(x) = -\frac1n\sum_{i=1}^n\log p_i(y_i\mid x),\quad x\in\Theta.$$
\end{example}
\begin{example}[Bayesian inverse problem]
In typical Bayesian inverse problems,  we are given an observation $y$, which is modeled as $y=f(x)+\epsilon$. Here $\epsilon$ is additive noise, $f:\Theta\to\R^k$ is the forward model, and $x$ is the signal we hope to recover. In this setting, $n$ represents the inverse of the noise amplitude. For example, we could have $\epsilon \sim\mathcal N(0, n^{-1}I_k)$, in which case $e^{-n\ell(x)} = e^{-n\|y-f(x)\|^2/2}$, and hence
$$\ell(x)=\frac12\|y-f(x)\|^2.$$ As before, $\prior\propto e^{-v_0}$ is a prior on $x\in\Theta$.\end{example}
\begin{example}[Fractional or tempered posterior]``Fractional'' or ``tempered'' posteriors are posteriors in which the likelihood $L(x)$ has been replaced by $L(x)^\alpha$ for some $\alpha\in(0,1)$. Fractional posteriors have been observed to be robust to misspecification~\cite{miller2018robust}, and their asymptotic properties (e.g. concentration) have been studied in~\cite{bhattacharya2019bayesian}. In this setting, $\ell(x)$ (whatever it may be) is simply replaced by $\alpha\ell(x)$. 
\end{example}
\begin{example}[Gibbs posterior]\label{ex:loss}In Gibbs posteriors, the log likelihood is replaced by a more general loss~\cite{bissiri2016general}. This gives rise to the generalized likelihood $L=e^{-n\ell}$, where $\ell$ is given by
\beq\label{eq:loss}\ell(x) = \frac1n\sum_{i=1}^n\rho_i(x, y_i).\eeq Here, $\rho_i$ is the loss function for the data point $y_i$; for example, $\rho_i(x,y_i)=\rho(x-y_i)$  if $x$ represents a location parameter. A Gibbs posterior may be appropriate in settings where the data-generating process is not known precisely. Alternatively, $\ell(x)$ could play the role of a risk to be minimized on average, in which case sampling $x$ from the generalized posterior $\pi\propto e^{-n\ell}\prior$ may be preferable to sampling from the traditional posterior modeling the data-generating process~\cite{jiang2008gibbs}.
%
%
\end{example}
Now that we have discussed a few typical forms of the function $\ell$, we return to the general setting of a measure $\pi\propto e^{-nv}$, and specify our assumptions on the function $v$. Later in Remark~\ref{rk:stat}, we discuss how to check the assumptions in terms of $\ell$ and $v_0$, in the case in which $v=\ell+n^{-1}v_0$. Note that the function $v$ depends on $d$ since it is defined on $\Theta\subseteq\R^d$. It also clearly depends on $n$. Although it is not explicitly required, our bounds make the most sense if $v$ is $\mathcal O(1)$ relative to $n$, informally speaking. This is satisfied if, for example, $\ell$ is an average of $n$ terms, as in Examples~\ref{ex:standard} and~\ref{ex:loss}, and $v_0$ is fixed with $n$ or $v_0 = n\tilde v_0$, where $\tilde v_0$ is fixed with $n$.

We will discuss the requirement on the size of $d$ relative to $n$ in Section~\ref{sec:V:epsilon}. For now, however, it is useful to keep in mind while reading the below assumptions that we should at least have $d^2\ll n$. 
\begin{taggedassump}{A1}[Regularity and unique strict minimum]\label{assume:1}
We have $v \in C^4(\Theta)$, and $v$ has a unique global minimizer $\mhat\in\Theta$ with $H_v=\nabla^2v(\mhat )\succ 0$.
\end{taggedassump}
For two of our results, we require the stronger assumption that $v\in C^5(\Theta)$; we will indicate when this is the case. The quantities $c_5(\s),\efive{\s}$ in Definitions~\ref{c3c4def},~\ref{epsdef} below will be used in conjunction with the stronger regularity. Now that we have assumed a unique strict minimizer $\mhat$ exists, we define a local neighborhood of $\mhat$:
\begin{defn}[Local region]\label{Udef}
We define the local region
$$\U(\s)=\{x\in\R^d\;: \;\|x-\mhat\|_{H_v}\leq\s\sqrt{d/n}\}.$$ 
\end{defn}
Next, we impose a lower bound on the growth of $v$.
\begin{taggedassump}{A2}[Lower bound on growth of $v$]\label{assume:c0} There exist $\s_0>0$ and $1\geq\cgro>C/\sqrt d$ for some absolute constant $C$ such that 
\beq\label{assume:c0:eq}
 v(x) - v(\mhat) \geq \cgro\sqrt {d/n}\|x-\mhat\|_{H_v}\quad\forall \;x\in\Theta\setminus\U(\s_0).
\eeq   \end{taggedassump}
Thus $v$ must grow at least linearly at infinity, though the slope of this linear growth can be very small, of order $\sqrt{d/n}$. 
\begin{remark}[Interpretation of Assumption~\ref{assume:c0}]To understand why the coefficient of linear growth is of the form $\cgro\sqrt{d/n}$, consider the boundary of the set $\U(\s_0)$, i.e. points $x$ such that $\|x-\mhat\|_{H_v}=\s_0\sqrt{d/n}$. Note that, if $\s_0$ is not too large then a Taylor approximation of $v$ shows that
 \beq\label{vsuff}v(x)-v(\mhat) \geq\frac14\|x-\mhat\|_{H_v}^2 = \frac{\s_0}{4}\sqrt{d/n}\|x-\mhat\|_{H_v},\qquad \forall\;\|x-\mhat\|_{H_v}=\s_0\sqrt{d/n}.\eeq Therefore, we know that the inequality in~\eqref{assume:c0:eq} holds at least for $x$ on the boundary of $\Theta\setminus\U(s_0)$, with $\cgro=\s_0/4$, provided $\s_0$ is small enough.  The assumption requires that this linear growth should continue out to infinity. However, $\cgro$ may be smaller than $\s_0/4$, meaning, it is not necessary that the slope $(\s_0/4)\sqrt{d/n}$ on the boundary is maintained out to infinity.\end{remark}
 \begin{remark}[Sufficient conditions for Assumptions~\ref{assume:1} and~\ref{assume:c0}]\label{vconvA2}
If $v$ is strongly convex, then $v$ has a unique strict global minimum $\mhat$, and therefore Assumption~\ref{assume:1} is satisfied. More general conditions under which Assumption~\ref{assume:1} is satisfied are difficult to state. Meanwhile, we show in Appendix~\ref{app:main} that Assumption~\ref{assume:c0} is satisfied with $\s_0=4$ and $\cgro=1$ if $v$ is convex and $c_3\sqrt{d/n}+c_4(4)d/n\leq3/8$. The quantities $c_3$ and $c_4(\cdot)$ are defined in~\eqref{nabla34} below, and they control the size of the third and fourth derivative of $v$ in a neighborhood of $\mhat$. This condition on $c_3$ and $c_4$ ensures that the boundary condition~\eqref{vsuff} is satisfied with $\s_0=4$. From there, we automatically obtain~\eqref{assume:c0:eq} with $\s_0=4$ and $\cgro=1$, because for convex functions, the infimum $\inf_{x\in\Theta\setminus\U(\s_0)}(v(x) - v(\mhat))/ \|x-\mhat\|_{H_v}$ is achieved at the boundary of $\U(\s_0)$.

As we discuss below, the quantities $c_3d/\sqrt n$ and $c_4(\s_0)d^2/n$ must both be small in order for our bounds to be useful. Therefore, the condition $c_3\sqrt{d/n}+c_4(\s_0)d/n\leq3/8$ is no more stringent than what is needed for our bounds to be useful. 

It is important to note that convexity or strong convexity of $v$ are \emph{not} required for the assumptions to hold. In one dimension, a prototypical example of a function $v$ satisfying our conditions is an asymmetric double well in which the deeper well is centered at $\mhat$, the shallower well is centered at some $x_1$, and $(v(x_1)-v(\mhat))/|x_1-\mhat|$ is not too small.

\end{remark}
\begin{remark}[Checking the assumptions in terms of likelihood and prior]\label{rk:stat}
Recall that $v = \ell+ n^{-1}v_0$ in Bayesian inference, where $\prior\propto e^{-v_0}$ is the prior and $L=e^{-n\ell}$ is the generalized likelihood. The above remarks show that if for example $v\in C^4(\Theta)$, is strongly convex, and its derivatives are not too large, then both assumptions are satisfied. To check that $v\in C^4(\Theta)$, we simply check that $\ell$ and $v_0$ are both $C^4(\Theta)$. Regarding the size of the derivatives of $v$, see Remark~\ref{rk:deriv:ell}. To get strong convexity of $v$, we could have for example that $\ell$ is convex and $v_0$ is strongly convex or vice versa. Let us discuss some typical settings where this holds.

In the standard setting of Example~\ref{ex:standard}, the function $\ell$ is (strongly) convex if $x\mapsto p_i(y_i|x)$ is (strongly) convex for each $i$. For example, when the model is given by an exponential family in canonical form, $p_i=p$ is convex in $x$~\cite{Sundberg2011}. When the model is given by a generalized linear model with canonical link function, $p_i(y_i|x) = p(y_i|x^\T z_i)$ is also convex in $x$ (here, $z_i$ are feature vectors)~\cite[p. 143]{agresti2015foundations}. In the case of a Gibbs posterior as in Example~\ref{ex:loss}, $\ell$ is (strongly) convex if the $\rho_i$ from~\eqref{eq:loss} are (strongly) convex in $x$. The standard example of a strongly log-concave prior is a Gaussian prior, and there are many more log concave priors. Note that the assumptions do not prevent the log prior from scaling with $n$. 

For worked-out examples, see Sections~\ref{sec:pmf} and~\ref{sec:log}, where we check the conditions for a multinomial and logistic regression model, respectively.\end{remark}

\subsection{Key quantities}\label{sec:V:epsilon}
We first define the Laplace approximation (LA), as well as the corrected LA.
\begin{defn}[LA and corrected LA]
Let $H_V=nH_v=\nabla^2V(\mhat)$, where $\mhat$ is the unique global minimizer of $V=nv$. The LA to $\pi\propto e^{-V}$ is 
$$\lap=\mathcal N(\mhat, H_V^{-1}).$$ The \emph{corrected} LA is the signed density $\corlap$ given by 
\beqs\label{Lhatdef}
d\corlap=(1+\Lhat)d\lap,\qquad \text{where}\qquad\Lhat (x) = -\frac16\lla\nabla^3V(\mhat),\,(x-\mhat)^{\otimes 3}\rra.\eeqs
\end{defn}
See Section~\ref{discuss:intuit} for the intuition behind this correction, and a discussion of the typical size of $\Lhat(x)$ for $x$ located in the region of concentration of $\lap$. We now define our first key small quantity, which will arise when bounding various leading order terms.
\begin{defn}[Quantity bounding leading order terms]\label{def:tilep}
Let $W(x)=V(\mhat + H_V^{-1/2}x)$, and define $\tilep\geq0$ as the square root of
\beq
\tilep^2 =\frac16\|\nabla^3W(0)\|_F^2 + \frac14\|\la\nabla^3W(0), I_d\ra\|^2,\eeq where $\|\cdot\|_F$ is the Frobenius norm, and $\la\nabla^3W(0), I_d\ra$ is defined as in~\eqref{laTA}.
 \end{defn}
 \noindent To see that $\tilep$ is small, note that $\nabla^3V\sim n$, and $H_V^{-1/2}\sim 1/\sqrt n$. Therefore by the chain rule, $\nabla^3W(0)\sim 1/\sqrt n$. In~\eqref{tulip-ep} below, we give an explicit upper bound on $\tilep$. If it is unsatisfying to define $\tilep$ in terms of $\nabla^3W(0)$, see (A.1) in Appendix~\ref{app:main} for the formula for $\tilep$ in terms of $\nabla^3V(\mhat)$ and $H_V^{-1/2}$. We highlight for future discussion that $\tilep$ is an \emph{explicit polynomial function of the entries of $\nabla^3V(\mhat)$ and $H_V^{-1/2}$.}\\
 
\noindent The reason this particular $\tilep$ arises is that it is the $L^2(\lap)$ norm of the correction function $\Lhat$.
\begin{lemma}[Formula for $\|\Lhat\|_{L^2(\lap)}$]\label{lma:Lbd}
 Let $\tilep$ be as in Definition~\ref{def:tilep}. It holds $\|\Lhat\|_{L^2(\lap)}= \tilep$.
 \end{lemma}
\noindent See Lemma~\ref{lma:prop:r2k} in Section~\ref{sec:prop:r2k} for the proof, which uses the representation $\|\Lhat\|_{L^2(\lap)}^2=\E[\la\tfrac{1}{3!}\nabla^3W(0), Z^{\otimes 3}\ra^2]$.  
Next, we introduce notation for the operator norms of derivatives of $v$. 
\begin{defn}[Local operator norm bounds on $\nabla^3v,\nabla^4v$,$\nabla^5v$]\label{c3c4def}Let
\begin{equation}\label{nabla34}
\begin{split}
c_k(\s)=\sup_{x\in\U(s)}\|\nabla^kv(x)\|_{H_v},\quad k=3,4,5,
  \end{split}\end{equation} and in particular, 
  $$c_3:=c_3(0)=\|\nabla^3v(\mhat)\|_{H_v}.$$
  \end{defn}
\noindent Next, we define the key quantities which arise when bounding the \emph{remainder} term.
   \begin{defn}[Quantities bounding the remainder term]\label{epsdef}
Recall that $V=nv$ and $H_V=\nabla^2V(\mhat)=nH_v$. We let
\beqs\label{e345def}
\ethree{\s} &:= c_3(\s)\frac{d}{\sqrt n} =d \sup_{x\in\U(s)}\|\nabla^3V(x )\|_{H_V},\\
  \efour{\s} &:= c_4(\s)^{1/2}\frac{d}{\sqrt n}=d\sup_{x\in\U(s)} \|\nabla^4V(x )\|_{H_V}^{1/2},\\
 \efive{\s} &:= c_5(\s)^{1/3}\frac{d}{\sqrt n} =d\sup_{x\in\U(s)} \|\nabla^5V(x )\|_{H_V}^{1/3}\eeqs and in particular,
 \beq\label{eps3def}\epsilon_3:=\ethree{0}=c_3(0)\frac{d}{\sqrt n}=d\|\nabla^3V(\mhat)\|_{H_V}.\eeq
 \end{defn}
 To see why the last equality is true in each line, note that $\|nT\|_{nA}= n^{-k/2}\|nT\|_{A}=n^{1-k/2}\|T\|_A$ for an order $k$ tensor $T$ and a matrix $A$. Our bounds stated in the next section are small if \beq\label{ebound}\epsilon_3\ll1,\quad\efour{\s}\ll1,\quad d^{-1/2}\efive{\s}^3\ll1,\eeq for an $\s$ to be specified. 
 \begin{remark}[Growth of the model-dependent factors]\label{rk:cgrowth}
If $c_3,c_4(\s),c_5(\s)$ are all bounded by absolute constants, then $\epsilon_k\sim d/\sqrt n$ for $k=3,4,5$. However, as noted in the introduction, $c_3,c_4(\s),c_5(\s)$ may in fact grow with dimension, and they may grow at different rates from one another. Correspondingly $\epsilon_3,\efour{\s},\efive{\s}$ may have different orders of magnitude, and they may be larger than $d/\sqrt n$. We will see this behavior in both of our examples in Sections~\ref{sec:pmf} and~\ref{sec:log}.

It may be possible that in some very favorable cases, one of $c_3,c_4(\s),c_5(\s)$ could \emph{decay} as $d$ increases. However, we will essentially disregard this possibility for simplicity, assuming for concreteness that $c_3(0)\gtrsim1$, and hence $\epsilon_3\gtrsim d/\sqrt n$. Therefore, the minimal requirement needed to satisfy~\eqref{ebound} is that $d^2\ll n$. 
\end{remark}
\begin{remark}[Choice of notation]\label{rk:eps-vs-c} In the bounds in Section~\ref{sec:V:results}, we will often see the quantity $\epsilon_3^2 +\efour{\s}^2$, which can be written as $ (c_3^2+c_4(\s))d^2/n$. One may wonder why we do not express our bounds in terms of the $c_k$'s, which would enable us to group terms of like power of $d^2/n$. The reason is that $c_3^2$ and $c_4(\s)$ could have different order of magnitude, so we wish to emphasize $\epsilon_3=c_3d/\sqrt n$ and $\efour{\s}=c_4(\s)^{1/2}d/\sqrt n$ as separate quantities, each of which must be small in order for our bounds to be useful. 
\end{remark}
Next, let us compare $\epsilon_3$ to $\tilep$. Using the inequalities $\|T\|_F\leq d\|T\|$ and $\|\la T, I_d\ra\|\leq d\|T\|$ shown in Lemma F.3, and noting that $\|\nabla^3W(0)\|=\|\nabla^3V(\mhat)\|_{H_V}$, we obtain $\tilep^2 \leq d^2\|\nabla^3V(\mhat)\|^2_{H_V} = \epsilon_3^2$ and hence
\beq\label{tulip-ep}
\tilep \leq \epsilon_3=c_3(0)\frac{d}{\sqrt n}
\eeq
Finally, we remark on how the prior affects the size of the derivative operator norms. 
\begin{remark}[Effect of the prior on derivative operator norms]\label{rk:deriv:ell}
Recall that in the context of Bayesian inference, $V(x)=n\ell(x)+v_0(x)$. Suppose the prior $\prior\propto e^{-v_0}$ is log concave, so that $\nabla^2v_0(x)\succ0$ for all $x\in\Theta$. Thus,
$$\|\nabla^kV(x)\|_{H_V}\leq n\|\nabla^k\ell(x)\|_{H_V} + \|\nabla^kv_0(x)\|_{H_V},\quad k=3,4,5.$$ and $H_V=\nabla^2V(\mhat)=n\nabla^2\ell(\mhat)+\nabla^2v_0(\mhat)$. Now, informally, the weighting in an $H_V$-weighted tensor operator norm amounts to ``dividing'' by some power of $H_V$. Thus increasing the strength of prior regularization (say, by changing $v_0$ to $cv_0$ for some $c>1$) increases $H_V$ in the positive definite ordering, and therefore decreases $\|\nabla^k\ell(x)\|_{H_V}$. At the same time, however, $\|\nabla^kv_0(x)\|_{H_V}$ may increase. The clearest case is when $\prior$ is Gaussian and hence $v_0$ is quadratic. We then have that $\nabla^kv_0=0$ for all $k\geq3$ and hence increasing the strength of prior regularization only decreases $\|\nabla^kV(x)\|_{H_V}$. (Technically, the location of the minimizer $\mhat$ changes along with the prior, so one must also show that $\nabla^2\ell(\mhat)$ is not significantly affected as $\mhat$ changes.) Thus a strongly regularizing Gaussian prior, in which $v_0$ is scaled by $n$, is favorable from the perspective of our bounds. This should be intuitively clear, since as the contribution to $V$ from the quadratic $v_0$ increases, the posterior becomes more Gaussian, and hence the LA becomes more accurate.
\end{remark}
\subsection{Main results: error bounds and leading order decompositions}\label{sec:V:results}

We now turn to our main results, in which we use the following shorthands to simplify the presentation:
\beqs
\Delta_g(\corlap) = R(g)&=\int gd\pi -\int g(1+\Lhat)d\lap,\\
L(g) &=\int g\Lhat d\lap,\\
\Delta_g(\lap)= L(g) +R(g)&=\int gd\pi-\int gd\lap.
\eeqs
We also use the following notation throughout this section.
\beq\label{Esdef}
\Es{\s}=\e\l(\l(\del{\s}\r)\s^4\r),\quad \taus{\s} =de^{-\cgro\s d/8},\quad \sst = \max(\s_0, (8/\cgro)\log(2e/\cgro)).
\eeq
Here, $\s_0$ and $\cgro$ are from Assumption~\ref{assume:c0}, and $\s$ is a user-chosen radius. The values $\epsilon_3$ and $\efour{\s}$ are as in Definition~\ref{epsdef}. 

There are two implications of a leading order error decomposition. First, it allows us to incorporate the leading order term into the approximation to improve the accuracy. Second, we can use it to make more precise statements about the original, uncorrected error. We emphasize the two perspectives with our choice of notation: $\Delta_g(\corlap)$ emphasizes the former and $R(g)$ the latter, even though they are equal to one another.

\subsubsection{Bounds on $\Delta_g$}
Consider a function $G$ such that for some $\af{G}>0$, it holds
 \beq\label{gcond}|G(x)-\medint\int Gd\lap|\leq\af{G}\e(\cgro\sqrt {d}\|x-\mhat\|_{H_V}/2)\quad\forall x\in\R^d,\eeq where $\cgro$ is from Assumption~\ref{assume:c0}. We then define
 \beq\label{gnorm}g(x) = \frac{G(x)-\medint\int Gd\lap}{\sqrt{\Var_{\lap}(G)}},\qquad \af{g} = \frac{\af{G}}{\sqrt{\Var_{\lap}(G)}}.\eeq
\begin{theorem}[Bound on $\Delta_g(\corlap)$; decomposition of $\Delta_g(\lap)$]\label{thm:Vgen}
Suppose $v$ satisfies Assumptions~\ref{assume:1},~\ref{assume:c0}. Let $G$ satisfy~\eqref{gcond} for some $\af{G}\geq0$, and let $g$, $\af{g}$ be as in~\eqref{gnorm}. Let $\s\geq\sst$ be such that $\U(\s)\subset\Theta$, and suppose $\epsilon_{3},\efour{\s}\les1$. Then we have the following bound on the accuracy $\Delta_g(\corlap)$ or equivalently, on the remainder $R(g)=\Delta_g(\lap)-L(g)$:
\beq\label{g-bd-2}
\begin{split}
|R(g)|=|\Delta_g(\corlap)|&\les\Es{\s}\l(\del{\s}\r) +  (\af{g}\vee1)\taus{\s}.
\end{split}
\eeq
Moreover, it holds $|L(g)|\leq \tilep$, and hence
\beq\label{g-bd-1}|\Delta_g(\lap)|\leq|L(g)|+|R(g)|\leq\tilep+|R(g)|.\eeq In particular, $\Delta_g(\lap)=\Delta_g(\corlap)$ when $g$ is even about $\mhat$, because $L(g)=0$ in this case. \end{theorem}
The bound $|L(g)|\leq\tilep$ follows immediately from Lemma~\ref{lma:Lbd} and Cauchy-Schwarz. To see that $L(g)=0$ if $g$ is even about $\mhat$, note that $\Lhat$ is odd about $\mhat$ and therefore so is $g\Lhat$. 
\begin{theorem}[Improved bound on $\Delta_g(\corlap)$ for odd $g$]\label{thm:odd}Suppose $v$ and $g$ are as in Theorem~\ref{thm:Vgen}, and additionally that $g$ is odd about $\mhat$ and $v\in C^5(\Theta)$. Then
\begin{align}
|\Delta_g(\corlap)|\les \delodd{\s}+ (\af{g}\vee1)\taus{\s}.\label{godd-bd}
\end{align}
\end{theorem}
%
We presented the bounds in terms of the standardized function $g$ for simplicity. However, it is straightforward to see that $L(G)=\sqrt{\Var_{\lap}(G)}L(g)$ and $R(G)=\Delta_G(\corlap)=\sqrt{\Var_{\lap}(G)}R(g)$. Note also that if $G$ is even (odd) about $\mhat$, then $g$ is even (odd) about $\mhat$. See Section~\ref{subsub:gcond} for a discussion of the condition~\eqref{gcond}. 
\begin{remark}[Size of $\Es{\s}$ and $\taus{\s}$]\label{rk:s} We show in Section~\ref{subsub:s} that if $d$ is large then in typical cases,  we can choose $\s$ to be an absolute constant. This ensures $\Es{\s}=\mathcal O(1)$ and $\taus{\s}$ is negligible compared to the other terms in our bounds. 
\end{remark}
\begin{remark}[Presentation of the bounds]\label{rk:intro}
Recall from Definition~\ref{epsdef} that $\epsilon_3=c_3d/\sqrt n$, $\efour{\s}=c_4(\s)^{\frac12}d/\sqrt n$, and $\efive{\s}=c_5(\s)^{\frac13}d/\sqrt n$. Similarly, $\tilep\leq \epsilon_3=c_3d/\sqrt n$. Thus the above bounds on $R(g)=\Delta_g(\corlap)$ and on $L(g)$ can be stated in terms of powers of $d/\sqrt n$, with a $v$-dependent coefficient in front. This is how we presented the bounds in~\eqref{Del-g-2},~\eqref{L-g},~\eqref{Del-g-1}, and~\eqref{Del-g-3} of the introduction. We did this to simplify the presentation, but we now prefer to use the quantities $\tilep,\epsilon_3,\efour{\s},\efive{\s}$, for the reason stated in Remark~\ref{rk:eps-vs-c}. 
\end{remark}

%


\begin{remark}[Orders of magnitude in the righthand sides]\label{rk:order} As indicated in Remark~\ref{rk:intro}, we have $\epsilon_3,\efour{\s},\efive{\s}\les_v \epsilon:=d/\sqrt n$ up to the $v$-dependent coefficients, and $\tilep\leq\epsilon_3\les_v\epsilon$ as well. Therefore, to informally understand the orders of magnitude of the righthand sides, we simply count how many $\epsilon$'s are multiplied together. This leads to the following simple, informal summary of the bounds in Theorems~\ref{thm:Vgen} and~\ref{thm:odd}:
\beqs\label{bds-g-informal}
 |\Delta_g(\lap)|&\les_v \begin{cases}\epsilon\vspace{5pt}\\ 
\epsilon^2,\quad\text{g even about $\mhat$},\end{cases}\\
|\Delta_g(\corlap)|&\les_v\begin{cases}\epsilon^2\vspace{5pt}\\ 
\epsilon^3 ,\quad\text{g odd about $\mhat$}.\end{cases}
\eeqs Here, we have omitted the exponentially small term $\taus{\s}$, and we have assumed $\Es{\s}\leq C$. We caution that~\eqref{bds-g-informal} is only meant as a very high-level summary, and a careful study of the model-dependent terms is necessary for a true understanding of the order of magnitude of our bounds.  
\end{remark}



To demonstrate the breadth of applicability of Theorem~\ref{thm:Vgen}, we first consider approximating the moment generating function (mgf) of $\pi$, which is given by the integral against $\pi$ of an exponentially growing function. 
\begin{example}[Approximating the mgf]\label{corr:mgf}
Let $M_\pi(u) = \E_{X\sim\pi}[\e(u^\T H_V^{1/2}(X-\mhat))]$ be the rescaled mgf of $\pi\propto e^{-nv}$, where $v$ is as in Theorem~\ref{thm:Vgen}. Then for any $\|u\|\leq \cgro\sqrt d/2$, it holds
\beqs\label{eq:mgf}
\bigg|\frac{M_\pi(u)}{\e(\|u\|^2/2)} - 1&+\lla\nabla^3W(0), \; \tfrac16u^{\otimes3}+\tfrac12I_d\otimes u\rra\bigg|\\
&\les \sqrt{e^{\|u\|^2}-1}\,\Es{\s}\l(\del{\s}\r) +  \taus{\s},
\eeqs where $W$ is the function $W(x)=V(\mhat+H_V^{-1/2}x)$. Let $M_{\corlap}(u), M_{\lap}(u)$ be the mgfs of $\corlap$ and $\lap$, respectively, rescaled the same way as in the definition of $M_\pi$. We have $M_{\lap}(u)=\e(\|u\|^2/2)$, so without access to the corrected measure $\corlap$, we would simply approximate $M_{\pi}(u)$ by $e^{\|u\|^2/2}$. Using $\corlap$ gives us an improved approximation to $M_\pi(u)$. Specifically, the quantity inside the absolute values is given by $(M_\pi(u) - M_{\corlap}(u))/M_{\lap}(u)$.
Note that the mgf of $\corlap$ is computable in closed form. Indeed, integrals of the form $\int e^{w^\T x}(1+\Lhat(x))d\lap(x)$ can be written as an integral of $1+\Lhat$ against a new Gaussian distribution with shifted mean.
\end{example}

Next, we apply Theorems~\ref{thm:Vgen} and~\ref{thm:odd} to three classes of observables: indicators, linear functions, and quadratic functions. This gives refined bounds on the TV, mean, and covariance error, respectively.
\subsubsection{Bounds on TV distance}

As a direct application of Theorem~\ref{thm:Vgen}, we obtain a bound on $\TV(\pi,\corlap)$, as well as a leading order decomposition of $\TV(\pi,\lap)$. In preparation for the following corollary, define
\beq\label{Ldef}
\LTV=\frac12\int|\Lhat| d\lap =\frac{1}{12}\E|\la\nabla^3V(\mhat), (H_V^{-1/2}Z)^{\otimes3}\ra|,\eeq  where $Z\sim\mathcal N(0, I_d)$.
\begin{corollary}[TV error of $\corlap$, and TV decomposition for $\lap$]\label{corr:corTV}
Let $\RTV:=\TV(\pi,\lap)-\LTV$, where $\LTV$ is as in~\eqref{Ldef}. Under the conditions on $v$ from Theorem~\ref{thm:Vgen}, we have
\beqs\label{TVcorlap}
|\RTV|\leq \TV(\pi,\corlap)&\les\,\Es{\s}\l(\del{\s}\r) +\taus{\s}.\eeqs
Moreover, it holds $0\leq\LTV\leq\tilep/2$, and hence
$$
\TV(\pi,\lap)\leq\tilep/2+|\RTV|.
$$
Furthermore, if $A$ is a Borel set on $\R^d$ which is symmetric about $\mhat$, then
\beq
|\pi(A)-\lap(A)|\leq \TV(\pi,\corlap).\eeq
\end{corollary}


The informal and coarse summary of this theorem's guarantee on the TV accuracy of $\corlap$ is that $\TV(\pi,\corlap)\les_v\epsilon^2$, where $\epsilon=d/\sqrt n$. Regarding the TV accuracy of $\lap$, we have $|\pi(A)-\lap(A)|\les_v\epsilon$ for all Borel sets $A$, and the improved bound $|\pi(A)-\lap(A)|\les_v\epsilon^2$ when $A$ is even about $\mhat$. Moreover, $\LTV$ is the ``leading order'' term of $\TV(\pi,\lap)$ in the sense that $\TV=\LTV+\RTV$ and
\beqsn
|\LTV|\les_v\epsilon,\quad |\RTV|&\les_v\epsilon^2.\eeqsn That $|\LTV|\les_v\epsilon$ follows from the fact that $\tilep\leq c_3\epsilon$, recalling~\eqref{tulip-ep}.

\subsubsection{Mean and covariance error}\label{subsub:mean}
Let $\mhatS=\int xd\corlap(x)$ be the mean of $\corlap$ and $\mpi=\int xd\pi(x)$ be the mean of $\pi$. Note that, by definition, we have $$\mhatS=\mhat + \Lhatm,\qquad \Lhatm := \int x\Lhat d\lap(x).$$ From the perspective of a leading order decomposition of the mean accuracy of $\lap$, we write
$$\mpi = \mhat + \Lhatm + R_{\mpi}.$$
\begin{corollary}[Mean error for $\corlap$, and mean decomposition for $\lap$]\label{thm:Vmean}
Assume the conditions on $v$ from Theorem~\ref{thm:Vgen}. Then $\Lhatm=\int x\Lhat d\lap(x)$ is given by
\beqs\label{Lm-def}
\Lhatm&= -\frac12H_V^{-1}\la\nabla^3V(\mhat ), H_V^{-1}\ra,
\eeqs
and the following bound holds on the approximation error $\mpi-\mhatS$ (equivalently, on the remainder $R_{\mpi}$):
\begin{align}
\big\|R_{\mpi}\big\|_{H_V}=  \big \|\mpi- \mhatS\big\|_{H_V} &\les \Es{\s}\l(\del{\s}\r) +\taus{\s}.\label{cor-mean}
 \end{align} If $v\in C^5(\Theta)$, then~\eqref{cor-mean} can be tightened to
\beqs\label{Rc5}
\big\|\mpi- \mhatS\big\|_{H_V} \les\delodd{\s}+\taus{\s}.\eeqs 
Finally, it holds $\|\Lhatm\|_{H_V}\leq\tilep$, and hence
 \beq\label{unc-mean-L}
 \|\mpi-\mhat\|_{H_V}\leq\tilep+\|R_{\mpi}\|_{H_V},
\eeq
where $\|R_{\mpi}\|_{H_V}$ is bounded as~\eqref{cor-mean} or~\eqref{Rc5}, depending on the regularity of $v$.
\end{corollary} 
 Finally, we bound the covariance approximation error. We only consider the error for $\lap$, since there should be no improvement from $\corlap$. This is essentially because the matrix-valued function $(x-\mhat)(x-\mhat)^\T$ is even about $\mhat$, and therefore the leading order term cancels.
\begin{corollary}[Covariance error bound for $\lap$]\label{thm:Vcov}Assume the conditions of Theorem~\ref{thm:Vgen}. Let $\Sigpi = \int (x-\mpi)(x-\mpi)^\T d\pi(x)$ be the covariance of $\pi$, and recall that $H_V^{-1}$ is the covariance of $\lap$. Then we have the upper bound 
\beq\label{eq:Vcov}\|\Sigpi-H_V^{-1}\|_{H_V^{-1}}=\|H_V^{1/2}(\Sigpi -H_V^{-1})H_V^{1/2}\|\les \Es{\s}^2\l(\del{\s}\r) +\taus{\s}.\eeq
\end{corollary}

The informal and coarse summary of Corollaries~\ref{thm:Vmean} and~\ref{thm:Vcov} are that 
$$\|\Lhatm\|_{H_V}\les_v\epsilon,\quad\|R_{\mpi}\|_{H_V}=\|\mpi- \mhatS\|_{H_V}\les_v\epsilon^3,\quad\|\Sigpi -H_V^{-1}\|_{H_V^{-1}}\les_v\epsilon^2,$$ where $\epsilon=d/\sqrt n$.

\begin{remark}[Mean and covariance bound: unweighted norm]
If a lower bound on $\lambda_{\min}(H_V)$ is known, it can be used in Corollaries~\ref{thm:Vmean} and~\ref{thm:Vcov} to obtain upper bounds on the unweighted norms $\|\mpi - \mhat\|$, $\|\mpi-\mhatS\|$, and $\|\Sigpi - H_V^{-1}\|$.
\end{remark}
\subsection{Lower bound}\label{subsec:lb}We now use our leading order decompositions to demonstrate a function $v$ for which the TV and mean errors of the uncorrected LA are bounded from \emph{below} by $d/\sqrt n$. (For a lower bound construction using a different proof method, see Section~\ref{sec:TV:pmf} on the multinomial model.) 

Let
$\phi(x)=\E[\psi(Z^\T x)]$ for a function $\psi:\R\to\R$, with $Z\sim\mathcal N(0, I_d)$, and define 
\beq\label{v-lb}v(x)=\phi(x) - \nabla\phi(x^*)^\T x = \E[\psi(Z^\T x)] - \E[\psi'(Z^\T x^*)Z^\T x]\eeq for some fixed $x^*\in\R^d$. We assume for simplicity that $x^*=e_1=(1,0,\dots,0)$, though the analysis can be generalized to other values of $x^*$. We show in Appendix~\ref{app:main} that $\mhat=x^*$, i.e. $x^*$ is the unique global minimizer of $v$.

When $\psi(t)=\log(1+e^t)$, this $v$ arises in the setting of logistic regression with standard Gaussian design. Specifically, $v$ is the expectation of the negative population log likelihood over the feature distribution. However, we do not use this fact to prove the below lower bound. 
\begin{lemma}[Upper bounds on remainder terms]\label{lma:log-reg-TV}
Suppose $\psi\in C^4(\R)$ is strictly convex, has uniformly bounded third and fourth derivatives, and is independent of both $d$ and $n$. Let $\pi\propto e^{-nv}=e^{-V}$ where $v$ is as in~\eqref{v-lb}, with $x^*=(1,0,\dots,0)$. Let $\LTV$ and $\Lhatm$ be as in~\eqref{Ldef} and~\eqref{Lm-def}, respectively. Then
\beq\label{rmder-bd-log}
|\TV(\pi,\lap)-\LTV| \les\epsilon^2+\tau,\quad \|\mpi-\mhat-\Lhatm\|_{H_V}\les \epsilon^2+\tau,\eeq where 
\beq\label{eps-tau-0}\epsilon=d/\sqrt n,\qquad\tau=d\e(-Cn^{1/4}\sqrt d).\eeq
\end{lemma}
\begin{lemma}[Lower bounds on leading order terms]\label{lma:leading-lb-log}
Consider the setting of Lemma~\ref{lma:log-reg-TV}, and define $a_{k,p}=\E[\psi^{(k)}(Z_1)Z_1^p]$, where $Z_1\sim\mathcal N(0,1)$. Then
\beqs
\LTV&\geq \frac{1}{a_{2,2}^{1/2}\sqrt n}\l((d-1)\frac{|a_{3,1}|}{8a_{2,0}} -\frac{|a_{3,3}|}{4a_{2,2}}\r),\\
\|\Lhatm\|_{H_V} &= \frac{1}{2 a_{2,2}^{1/2}\sqrt n}\l|(d-1)\frac{a_{3,1}}{a_{2,0}}+ \frac{a_{3,3}}{a_{2,2}}\r|.\eeqs
\end{lemma}
Combining Lemmas~\ref{lma:log-reg-TV} and~\ref{lma:leading-lb-log} immediately gives the following corollary.
\begin{corollary}\label{corr:lb}
In the setting of Lemma~\ref{lma:log-reg-TV}, we have
\beqs
\TV(\pi,\lap)&\geq C_1\epsilon - C_2\epsilon^2-C_2\tau\\
\|\mpi-\mhat\|_{H_V}&\geq C_1\epsilon - C_2\epsilon^2-C_2\tau,
\eeqs where $\epsilon$ and $\tau$ are as in~\eqref{eps-tau-0}, and $C_1,C_2>0$ are absolute constants.
\end{corollary}
\section{Discussion}\label{sec:discuss}
 In Section~\ref{subsec:discuss}, we discuss the significance of the results of Section~\ref{sec:V} and compare to prior work. We then compare the LA error bounds to those of Gaussian variational inference in Section~\ref{VI}, and give further details and intuition about our results in Section~\ref{subsec:further}.
 \subsection{Discussion of contributions}\label{subsec:discuss}
Here, we summarize the significance of our results, and then go into more detail about each point in the context of prior work. \\

\noindent (1) We obtain the first ever higher-order accurate LA in high dimensions, and bounded its error in terms of $\Delta_g$, TV, and mean approximation, in Theorem~\ref{thm:Vgen} and~\ref{thm:odd}, Corollary~\ref{corr:corTV}, and Corollary~\ref{thm:Vmean}.\\

\noindent (2) Our decomposition of the LA error enables us to prove \emph{lower bounds} of order $d/\sqrt n$ on the TV error of the uncorrected approximation, in Corollary~\ref{corr:lb}. This proves that the condition $d^2\ll n$ is in general \emph{necessary} to guarantee accuracy of the LA.\\

\noindent (3) In Theorem~\ref{thm:Vgen}, we obtain the first ever bound on the LA error $\Delta_g(\lap)$ in the high-dimensional regime, for any function $g$ with log linear growth. \\


\noindent (4) We obtain upper bounds on the TV, mean, and covariance error of the LA, in Corollaries~\ref{corr:corTV},~\ref{thm:Vmean}, and~\ref{thm:Vcov}, respectively. Our bounds are in some cases an \emph{order of magnitude} smaller than all previously known bounds.\\

\noindent(5) The leading order term of the LA error serves as a proxy to the overall error, and this proxy is \emph{computable}, unlike most prior bounds. \\

We now go into more detail about these points in the context of prior work. The numbers of the below subsections correspond to the numbering of the above points (e.g. Section~\ref{subsub:lower} corresponds to point 2).


\subsubsection{Higher-order accurate LA}\label{subsub:higherorder}
As we have stated above, the work~\cite{durante2023skewed} is the only other work to obtain a higher-order accurate LA, but the authors have only proved their results for fixed dimension $d$. Besides this, the bounds in~\cite{durante2023skewed} are suboptimal relative to ours even in their $n$ dependence, due to the presence of a potentially large power of $\log n$ (see e.g. the term $M_n^{c_8}$ in Theorem 4.1 of~\cite{durante2023skewed}). 

The approximation $p_{\mathrm{SKS}}$ (SKS for ``skew-symmetric'') derived in their work also suffers from a significant practical drawback relative to ours. Namely, integrals against $p_{\mathrm{SKS}}$ of simple functions such as polynomials are not known in closed form. Although the authors explain that sampling from $p_{\mathrm{SKS}}$ is as cheap as sampling from a Gaussian, a large number of samples is needed to decrease the Monte Carlo error to a level below that of the improved accuracy that $p_{\mathrm{SKS}}$ achieves. In contrast, integrals against $\corlap$ of both polynomials $q(x)$ and ``exponentially shifted" polynomials $e^{u^\T x}q(x)$ \emph{can} be computed in closed form. Therefore, one can more directly reap the reward of the decreased approximation error of $\corlap$, without this small error being overwhelmed by a possibly larger Monte Carlo error. 

There is one potential advantage of $p_{\mathrm{SKS}}$ over our $\corlap$: there is an exact and explicit procedure to sample from $p_{\mathrm{SKS}}$, and it is a well-defined probability density. In contrast, sampling from $\corlap$ is technically not well-defined since $\corlap$ is not a probability measure: $\corlap(x)$ can take negative values for some $x$ far from the bulk of the distribution. 

For the purpose of approximating integrals $\int gd\pi$, this is irrelevant both philosophically and computationally. Philosophically, if the integral $\int gd\corlap$ is a good approximation of $\int gd\pi$, then it is irrelevant that this was achieved using a measure which is not a probability density. In particular, if $\int \ind_Ad\corlap$ is negative for some set $A$, this simply means that $\pi(A)\les_v (d/\sqrt n)^2$. Practically, that we cannot sample from $\corlap$ is not an obstacle because we can sample $X_i\iid \lap$, and then estimate $\int gd\corlap=\int g(1+\Lhat)d\lap$ via $\frac1n\sum_{i=1}^n(1+\Lhat(X_i))g(X_i)$. This procedure is only necessary when the integral cannot be evaluated in closed form. 

%

However, for the purpose of constructing a credible set $\mathcal R_\alpha$ of a specified coverage $1-\alpha$ (i.e. such that $\pi(\mathcal R_\alpha)=1-\alpha$), it is indeed convenient to sample directly from the approximating measure. To see why, consider for simplicity a credible interval for the one-dimensional marginal of coordinate $x_j$ under the distribution $x\sim\pi$. If we have access to samples drawn from the approximating density, then we can simply take the interval between the $\alpha/2$ and $1-\alpha/2$ sample quantiles of that coordinate. Using our approximation, it may be possible to get an analytic expression for the density of the marginal (under the joint distribution $\corlap$), since we are taking an integral over $\R^{d-1}$ of a polynomial against a Gaussian. From here, a credible interval is simple to construct, since this a one-dimensional problem.  We leave this for future work.

\subsubsection{Lower bounds}\label{subsub:lower} We are not aware of any lower bound beyond the aforementioned recent work~\cite{bilodeau2023tightness}. There, the authors prove a lower bound on the Laplace approximation to the normalizing constant, in fixed dimension and for a few simple models. In contrast, our lower bound from Section~\ref{subsec:lb} holds in high dimensions. Moreover, the lower bound construction is based on the framework of the leading order decomposition, so it may be possible to use it to obtain lower bounds for a broader class of examples.


\subsubsection{Bounds on $\Delta_g(\lap)$}
As noted in the introduction, prior work on the Laplace approximation error to $\int gd\pi$ for general classes of $g$ is relatively sparse. The works~\cite{bp} and~\cite{fischer2022normal} both rely on Stein's method to prove their bounds on $\Delta_g(\lap)$, whereas we take a more direct approach (outlined in Section~\ref{sec:overview}). Stein's method seems to be a less powerful approach to bound $\Delta_g(\lap)$. Indeed, the work~\cite{bp} only bounds $\Delta_g(\lap)$ in $d=1$ and notes ``so far we have struggled to find enough theory that would let us compute useful bounds using [our] approach [in higher dimensions]''. The bounds of~\cite{fischer2022normal} apply in arbitrary dimension $d$, but have very suboptimal scaling with $d$.\\

\subsubsection{Tighter upper bounds}\label{subsub:tight}
As noted in the introduction,~\cite{bp} is the first work to obtain Laplace approximation bounds scaling as $d/\sqrt n$ up to model-dependent factors, followed by our earlier work~\cite{katsBVM}, which improves on the TV bound of~\cite{bp} in the case that $v$ grows linearly at infinity. We now show that our TV bound can be an order of magnitude smaller than that of~\cite{bp} and~\cite{katsBVM}. The same is true for the mean error of~\cite{bp} (and for the same reasons), so we only discuss the TV bound here. 

Although our results from Section~\ref{sec:V} do not require that $d\gg1$, we consider this case in the present discussion, since this is the interesting and nontrivial regime. More specifically, we assume $d\gg\log n$. We also assume that $\s_0$ and $\cgro$ from Assumption~\ref{assume:c0} are absolute constants, which is easily satisfied if, for example, $v$ is convex (see Remark~\ref{vconvA2}). As shown in Section~\ref{subsub:s}, in this regime there is a natural choice of the radius $\s$, which leads to the simpler bound $|\TV(\pi,\lap)-\LTV|\les\epsilon_3^2+\efour{\s}^2$.  Adding this simpler bound to the bound $\LTV\leq\tilep$ from Corollary~\ref{corr:corTV} yields
\beq\label{ourTV}
\TV(\pi,\lap)\leq \tilep+C\ethree{0}^2 + C\efour{\s}^2
\eeq where $\s=\sst=\max(\s_0, (8/\cgro)\log(2e/\cgro))$ is an absolute constant and $C$ is an unspecified absolute constant.

The work~\cite{katsBVM} assumes $v\in C^3(\Theta)$, that $v$ has a unique strict global minimizer, and that $v$ satisfies the linear growth condition~\eqref{assume:c0:eq}, with the additional requirement that $\cgro=\s_0/4$ and $\s_0$ is an absolute constant. Assuming we are in the regime $d\gg\log n$ (though this is not required in~\cite{katsBVM}), the bound in that work can be written as
\beq\label{earlierTV}
\TV(\pi,\lap)\leq C\ethree{\s}
\eeq where $C$ is an \emph{explicit} absolute constant and $\s=\s_0$, also an absolute constant. The constant $C$ stems from absorbing an exponentially negligible term into the main term in the upper bound. 

Finally, let us state the TV bound of~\cite{bp}. This bound takes a different form because it is given explicitly in terms of $\ell$ and $\prior$, where $\pi\propto e^{-n\ell}\prior$. A straightforward comparison between the bound in the present work, that of~\cite{katsBVM}, and that of~\cite{bp} is simplest if $\Theta$ is a bounded domain. In this case we can formally choose a uniform prior on $\Theta$ in the work of~\cite{bp} and let $\ell=v$, in which case $\pi\propto e^{-nv}$ on $\Theta$. This enables us to write the bounds of~\cite{bp} in terms of $v$. This trick does not actually limit the bound of~\cite{bp} to a uniform prior, since we are now free to ``redistribute'' $v$ as $\ell + n^{-1}v_0$ for some other $\ell$ and $v_0$. With some further processing, we show in~\cite{katsBVM} that the bound of~\cite{bp} can be written as
\beq\label{theirTV}
\TV(\pi,\lap)\leq C \ethree{\s},\qquad \s\sim\sqrt{\log n},
\eeq for an explicit absolute constant $C$. Taking $\s\sim\sqrt{\log n}$ is needed to ensure that a second term is exponentially small and can be absorbed into the main term. The main assumptions needed in~\cite{bp} are that $v\in C^3(\Theta)$, has a unique global minimizer, and there is some $\kappa$ such that $v(x)-v(\mhat)>\kappa>0$ for all $\|x-\mhat\|_{H_v}\geq \s\sqrt{d/n}$. As we show in~\cite{bp}, this weaker requirement on the growth of $v$ at infinity is what leads to the larger radius $\s\sim\sqrt{\log n}$.

The main advantage of the bounds~\eqref{earlierTV},~\eqref{theirTV} over the bound~\eqref{ourTV} in the present work is that~\eqref{earlierTV} and~\eqref{theirTV}  are fully explicit, require weaker regularity of $v$, and only involve the operator norm of the third derivative, rather than the operator norms of both the third and fourth derivatives. But this comes at the cost of coarser bounds, as we now show.

Let us compare each of the three terms in the righthand side of~\eqref{ourTV} to the righthand side of~\eqref{earlierTV}. We compare~\eqref{ourTV} to~\eqref{earlierTV} rather than to~\eqref{theirTV} because~\eqref{ourTV} and~\eqref{earlierTV} both take $\s$ to be an absolute constant. Clearly, $\ethree{0}^2\ll \ethree{\s}$. Furthermore, $\efour{\s}^2=c_4(\s)d^2/n\ll c_3(\s)d/\sqrt n=\ethree{\s}$ if $c_4(\s)\ll \frac{\sqrt n}{d}c_3(\s)$. Thus as long as the fourth derivative operator norm is not significantly larger than the third derivative operator norm, the third term in~\eqref{ourTV} is also an order of magnitude smaller than righthand side of~\eqref{theirTV}. This holds in both examples considered in Section~\ref{sec:pmf} and~\ref{sec:log}. In the multinomial example in Section~\ref{sec:pmf}, $c_4(\s)d^2/n\ll c_3(\s)d/\sqrt n$ holds whenever $\ethree{\s}=c_3(\s)d/\sqrt n\ll1$, which is the necessary requirement for~\eqref{earlierTV} to be a useful bound. In the logistic regression example in Section~\ref{sec:log}, we will see that $c_3(\s)$ and $c_4(\s)$ are both constants as long as $d^2\leq n$. Therefore $c_4(\s)d^2/n\ll c_3(\s)d/\sqrt n$ holds in this setting if $d^2\ll n$.

Finally, let us compare $\tilep$ to $\ethree{\s}$. By~\eqref{tulip-ep}, we have $\tilep\leq \ethree{0}\leq \ethree{\s}$. Thus, $\tilep$ is always less than or equal to $\ethree{\s}$. But in fact, $\tilep$ can actually be an \emph{order of magnitude} smaller than $\ethree{\s}$. Indeed, we show in Section~\ref{sec:c3:pmf} on the multinomial model that $\ethree{\s}/\tilep$ is at least $\sqrt d$, and can be as large as $d$. 

The reason our upper bounds are tighter than those of~\cite{bp,katsBVM} is that we have derived a leading order decomposition of the TV distance, and estimated both terms $\LTV$ and $\RTV$ separately.  Beyond tighter upper bounds, the additional advantage of this decomposition is that it allows us both to prove lower bounds (by studying the leading order term) and to derive the higher-order accuracy LA. Neither of these things is possible using the approach of~\cite{bp,katsBVM}. This is because these two works use a log Sobolev inequality at the very beginning to bound the TV distance from above.

Additionally, we mention that our covariance error bound is also tighter than that of~\cite{bp}, who obtains a bound scaling as $d\sqrt d/\sqrt n$, up to model dependent terms. This is suboptimal both in its $d$ and $n$ dependence. Another work bounding the mean and covariance is~\cite{huggins2018practical}, via the Wasserstein-1 and Wasserstein-2 distance, respectively. The bounds appear suboptimal by a factor of $\sqrt d$; however, the authors have not made the dimension dependence of their bounds explicit.

We have primarily compared our bounds to those of~\cite{bp} and~\cite{katsBVM} because these are the only other work which obtains bounds in terms of $d/\sqrt n$. But it is also worthwhile to compare our bounds to those of~\cite{spok23}, which are in terms of an effective dimension $d_{\text{eff}}$ (called $p_G$ in that work). The work~\cite{spok23} obtains that $\TV(\pi,\lap)\leq c(v)(d_{\text{eff}}/\sqrt n)^3$. Here, $d_{\text{eff}}\leq d$, so in some cases, $(d_{\text{eff}}/\sqrt n)^3$ may actually be a tighter bound than $(d/\sqrt n)^2$. The size of $d_{\text{eff}}$ depends on the strength of prior regularization: as the strong convexity parameter of the negative log prior grows (holding all else equal), the effective dimension decreases. However, if the convexity parameter of the negative log prior remains constant as $n\to\infty$ (which is the case in the typical setting in which the prior is chosen independently of sample size), then $d_{\text{eff}}$ tends to $d$ as $n\to\infty$. For example, the work~\cite{bp} has compared~\cite{spok23}'s $d_{\text{eff}}$ to $d$ in the setting of logistic regression with a standard Gaussian prior. The authors show that $d_{\text{eff}}=(1-o_p(1))d$, where $o_p(1)$ denotes convergence to zero in probability as $n\to\infty$.

\subsubsection{Computability}\label{subsub:compute}
Obtaining an upper bound on the Laplace TV error metric which is computable in polynomial time is an important and difficult problem. Let us revisit the works discussed in the introduction from the perspective of computability. The TV bounds of~\cite{spok23, helin2022non, dehaene2019deterministic} are not computable in polynomial time because they involve a tensor operator norm. This is the solution to a $d$-dimensional, constrained non-convex optimization problem. Computing this quantity exactly, without resorting to loose upper bounds, typically requires exponential-in-$d$ complexity. The bound of~\cite{fischer2022normal} on the Laplace TV error for exponential families does not involve an operator norm, but at the cost of having very suboptimal dimension dependence; it is stated that at best, $d^6\ll n$ is required for the bound to be small. This is in contrast to the $d^3\ll n$ condition of ~\cite{spok23, helin2022non, dehaene2019deterministic}. Finally,~\cite{bp} and~\cite{katsBVM} obtain a TV bound with the tight, $d^2\ll n$ dimension dependence. However, the bounds in these two works again require computing the operator norm $\|\nabla^3v(\mhat)\|_{H_v}$. 

In this work, we take an important step forward towards resolving the problem of computability. Namely, we provide an upper bound $\LTV\leq\tilep$ on the leading order term of the TV error, which is computable in polynomial time. To see this, recall from Definition~\ref{def:tilep} and the subsequent discussion that $\tilep$ is an explicit, polynomial function of the entries of $\nabla^3V(\mhat)$ and $\nabla^2V(\mhat)^{-1/2}$. Therefore, computing $\tilep$ indeed requires only polynomial complexity. 

Of course, our leading order term $\LTV$ is not leading order in the classical sense, so $\tilep$ is not automatically an accurate upper bound on $\TV(\pi,\lap)$. Obtaining a true upper bound on the TV distance using $\tilep$ would require a more refined analysis of $\RTV$. 
Namely, one could try to show $\RTV\leq C\tilep$ for some known or computable $C$, which leads to the overall bound $\TV(\pi,\lap)\leq (C+1)\tilep$. Note that $\RTV\leq C\tilep$ is likely to be a coarse bound, since we expect $\RTV$ to be much smaller than $\LTV$ in most cases. Only obtaining a coarse bound on $\RTV$ should require less computational or theoretical complexity than obtaining a more precise estimate of $\RTV$.

This additional analysis of $\RTV$ is beyond the scope of the present work.

\subsection{Comparison to Gaussian VI}\label{VI}
Gaussian variational inference (VI) offers another Gaussian approximation to posteriors $\pi$. It is defined as
\beq
\mathcal N(\hat x_{\mathrm{VI}},\hat\Sigma_{\mathrm{VI}})= \argmin_{p\in\mathcal P_{\mathrm{Gauss}}}\KL{p}{\pi},\eeq where $\mathcal P_{\mathrm{Gauss}}$ is the family of nondegenerate Gaussian distributions on $\R^d$. For a measure $\pi\propto e^{-nv}$ on $\R^d$,~\cite{KRVI} bounds the mean, covariance, and TV error of Gaussian VI in terms of $d$ and $n$. Here, we report the error bounds, for simplicity ignoring contributions from model dependent factors.

The authors find that the TV error is bounded by $d/\sqrt n$, the mean error is bounded by $(d/\sqrt n)^3$, and the covariance error is bounded by $(d/\sqrt n)^2$. Thus the VI and Laplace covariance and TV errors are of the same order, but the VI mean approximation is significantly more accurate than the Laplace mean approximation error. However, Corollary~\ref{thm:Vmean} shows that correcting the Laplace mean estimate from $\mhat$ to $\mhat + \Lhatm$ brings the accuracy up to the same order as that of VI. Therefore, when the third order derivative tensor $\nabla^3V(\mhat)$ can be computed, the skew-corrected LA is a viable alternative to Gaussian VI.

\subsection{Further details and intuition}\label{subsec:further}
\subsubsection{Verifying the conditions on $G$, and the role of $\Var_{\lap}(G)$}\label{subsub:gcond}
 The quantity $|G(x)-\int Gd\lap|$ from~\eqref{gcond} may be difficult to bound directly. As an alternative, suppose the following holds: 
\beq\label{g-alt}|G(x)-G(\mhat)|\leq  \cf{G}\e\l(\cgro\sqrt d\|x-\mhat\|_{H_V}/2\r)\quad\forall x\in\R^d.\eeq Then~\eqref{gcond} is satisfied with $\af{G} =\cf{G} +|\int Gd\lap-G(\mhat)|$. 

The condition~\eqref{gcond} is most natural to interpret in the case that $G(x)=F(H_V^{1/2}(x-\mhat))$. In this case,~\eqref{gcond} is equivalent to $F$ satisfying $|F(x)-\int Fd\gamma|\leq \af{F}\e(\cgro\sqrt d\|x\|/2)$ for all $x\in\R^d$, where $\af{F}=\af{G}$ and $\gamma$ is the standard normal distribution. Even more simply, the alternative~\eqref{g-alt} is equivalent to
$$|F(x)-F(0)|\leq \cf{F}\e\l(\cgro\sqrt d\|x\|/2\r)\quad\forall x\in\R^d.$$ Essentially, therefore, the theorem applies whenever $F$ is log linear in $\|x\|$, and the linear growth coefficient can be large as $\cgro\sqrt d/2$. This is equivalent to $G$ being log linear in $\|x-\mhat\|_{H_V}$, with the same linear growth coefficient.\\
Since $\Delta_G=\sqrt{\Var_{\lap}(G)}\Delta_g$ (where $g$ is the standardized function from Theorem~\ref{thm:Vgen}), the scale of $G$ enters into the error bounds via the factor $\sqrt{\Var_{\lap}(G)}$, the standard deviation of $G$ under $\lap$. Let us then discuss the ``typical'' size of $\Var_{\lap}(G)$. This is an important factor to take into account, because for certain functions $G$, the multiplicative factor $\sqrt{\Var_{\lap}(G)}$ entering the bound on $\Delta_G$ can change the order of magnitude of the error. For simplicity, we concentrate only on the $n$ dependence in this discussion, thinking of $d$ as fixed.

Since $\lap$ concentrates in a $\mathcal O(1/\sqrt n)$ neighborhood of $\mhat$, the standard deviation of $G$ depends on how rapidly $G$ changes throughout this neighborhood. In particular, suppose $G$ is fixed, i.e. not changing with $n$. If $G$ is smooth, then we expect $\sqrt{\Var_{\lap}(G)}$ to be small. Indeed, let $Z\sim\mathcal N(0, I_d)$, so that $\mhat + H_v^{-1/2}Z/\sqrt n\sim\lap$. We have $$G(\mhat + H_v^{-1/2}Z/\sqrt n)\approx G(\mhat)+\nabla G(\mhat)^\T H_v^{-1/2}Z/\sqrt n$$ for probable values of $Z$, that is, $\|Z\|=\mathcal O(1)$. Thus the standard deviation of the random variable $G(\mhat + H_v^{-1/2}Z/\sqrt n)$  is on the order of $\mathcal O(n^{-1/2})$. As a result, the bound on the leading order term $L(G)=\sqrt{\Var_{\lap}(G)}L(g)$ of $\Delta_G(\lap)$, scales not as $\tilep\sim n^{-1/2}$ but as $n^{-1/2}\times n^{-1/2}=n^{-1}$. The overall bound on $\Delta_G(\lap)$ thus also scales as $n^{-1}$. 
This is in line with the classical results in Laplace integral asymptotics, which are reviewed in~\cite[Section 2]{schillings2020convergence}. See (4) of that work in particular. 

On the other hand, consider $G$ of the form $G(x)=G(H_V^{1/2}(x-\mhat))$, where now $F$ is fixed with $n$. Such functions $G$ ``match'' the scale of $\lap$, so that $\Var_{\lap}(G)=\Var_{\gamma}(F)$, where $\gamma$ is standard normal. Since neither $\gamma$ nor $F$ depends on $n$, we have $\Var_\gamma(F)=\mathcal O(1)$. Therefore, the bound on $\Delta_G$ does not contain an extra factor of $n^{-1/2}$. For example, in Corollary~\ref{thm:Vmean}, we consider the \emph{scaled} quantity $H_V^{1/2}(\mpi-\mhat)$, which is $\Delta_G(\lap)$ for the function $G(x)=F(H_V^{1/2}(x-\mhat))$, where $F(x)=x$. Correspondingly, we see that the leading order term in the bound~\eqref{unc-mean-L} on $\|H_V^{1/2}(\mpi-\mhat)\|$ depends on $\tilep$ alone.

\subsubsection{Choice of $\s$}\label{subsub:s}
The optimal choice of $\s$ in our bounds from Section~\ref{sec:V:results} depends on a number of parameters. We will give a concrete choice of $\s$ in one representative parameter regime, and a strategy for choosing $\s$ in more general situations.

First, suppose $d\gg\log n$, and that Assumption~\ref{assume:c0} can be satisfied with $\cgro$ and $\s_0$ both given by absolute constants (which is essentially automatic if $v$ is convex; recall Remark~\ref{vconvA2}). We can then take the smallest allowable value for $\s$, i.e. $\s=\sst$, which is itself an absolute constant (recall the definition~\eqref{Esdef} of $\sst$ is in terms of $\cgro$ and $\s_0$). It follows that $\taus{\s}=\taus{\sst}\leq e^{-Cd}$, which is dominated by any power of $d/\sqrt n$, up to an absolute constant. Also, since we have assumed in Theorem~\ref{thm:Vgen} that $\epsilon_3,\efour{\s}\les 1$ for our chosen $\s$, it follows that $\Es{\s}=\Es{\sst}=\e((\epsilon_3^2+\efour{\sst}){\sst}^4)\leq C$. We can therefore absorb $\Es{\s}$ into the $\les$ sign in all of our bounds. These simplifications lead to the below bounds, where $\s=\sst$ is an absolute constant, and $g$ satisfies $\Var_{\lap}(g)=1$:
\beqs
|R(g)|&=|\Delta_g(\corlap)| \les \del{\s} + (\af{g}\vee1)e^{-Cd},\\
|\RTV|&\leq\TV(\pi,\corlap)\les \del{\s},\\
\|R_{\mpi}\|_{H_V}&=\|\mpi- \mhat - \Lhatm\|_{H_V} \\
&\les (\epsilon_3^2+\efour{\s}^2)(\epsilon_3+\efour{\s}^2)+d^{-1/2}\efive{\s}^3.
\eeqs

Note that in the second and third inequality, we discarded $\taus{\s}\leq e^{-Cd}$. This is because we have assumed that $\epsilon_3\gtrsim d/\sqrt n$; recall Remark~\ref{rk:cgrowth}. Thus $\taus{\s}\ll (d/\sqrt n)^k\les\epsilon_3^k$ for any constant power $k$, so $\taus{\s}$ is negligible compared to the other terms.\\

When $d$ is \emph{not} large, we must choose $\s$ sufficiently large to ensure $\taus{\s}$ is negligible compared to the other terms. Essentially, the goal is to choose $\s$ as large as possible while a) ensuring the exponent $(\del{\s})\s^4$ of $\Es{\s}$ remains bounded by a constant, and b) ideally, ensuring that $\efour{\s},\efive{\s}$ remain within a constant factor of $\efour{0},\efive{0}$. If $\efour{\s},\efive{\s}$ do not grow too quickly with $\s$ then a) is the more stringent condition, since it involves $\s$ explicitly. We thus find the largest $\s$ such that $(\del{\s})\s^4\leq1$. This strategy is the one we use in the examples in Section~\ref{sec:pmf} and~\ref{sec:log}. The resulting $\taus{\s}$ goes to zero exponentially fast with $n$, so the bounds do not rely on $d$ being large. 

We cannot offer a general purpose recipe for how to choose $\s$, as this depends on the growth of $\efour{\s},\efive{\s}$ with $\s$.

\subsubsection{Intuition behind the skew correction}\label{discuss:intuit}
Recall that $\pi\propto e^{-V}$, where $V(x)=nv(x)$. For simplicity, assume $v$ is independent of $n$, and consider $d$ to be fixed. The rationale behind the Laplace approximation $\pi\approx\lap$ is that when $n$ is large, $\pi$ concentrates in a $\mathcal O(1/\sqrt n)$-ball around its mode $\mhat$, and in this region, $V$ is close to its quadratic Taylor expansion about $\mhat$. To derive our corrected density $\corlap=\lap(1+\Lhat)$, we Taylor expand $V$ to \emph{third} order, to get that
 \beqsn
 V(x)&\approx V(\mhat) + \frac12(x-\mhat)^\T \nabla^2V(\mhat)(x-\mhat) + \frac{1}{3!}\la\nabla^3V(\mhat), (x-\mhat)^{\otimes 3}\ra\\
 & =\mathrm{const.} -\log\lap(x) - \Lhat (x).
 \eeqsn
 Thus $e^{-V}\approx \mathrm{const.}\lap e^{\Lhat }$. This suggests the approximation $\pi\approx \lap e^{\Lhat }/\mathcal Z$ for some $\mathcal Z$. But we cannot do this, since $\lap e^{\Lhat }$ is not normalizable. Indeed, $\Lhat $ is cubic, and can therefore take positive values that overwhelm the negative quadratic $\log\lap$. So instead, we simply make the approximation $e^{\Lhat }\approx 1+\Lhat$. Heuristically, this is valid because we have $\Lhat(x)\sim 1/\sqrt n$ for probable values of $x$ under $\lap$. Indeed, $\nabla^3V(\mhat)\sim n$ and $x-\mhat\sim 1/\sqrt n$ in the region of concentration of $\lap$, so $\la\nabla^3V(\mhat), (x-\mhat)^{\otimes 3}\ra\sim 1/\sqrt n$. As a result we get that $\pi\approx \corlap=\lap(1+\Lhat )$.

\section{Application to a multinomial model}\label{sec:pmf}
\renewcommand{\ground}{\theta^*}

In this section, we give an illuminating example that clearly demonstrates the improved understanding of the uncorrected LA which can be gleaned from our leading order decompositions. We consider a Dirichlet prior on the parameters of a multinomial distribution, which leads to a Dirichlet posterior. This is a sufficiently simple distribution for which certain calculations can be made explicitly. For example, there are simple, closed-form expressions for both the mean and mode of the Dirichlet distribution. Moreover, the quantities $\tilep, \epsilon_3,\efour{\s},\efive{\s}$ can also be calculated exactly or up to constants.

As such, this setting offers us a rare opportunity to explicitly study our leading order error decompositions; in particular, we study the decomposition of the TV distance from Corollary~\ref{corr:corTV}, and the decomposition of the mean error from Corollary~\ref{thm:Vmean}. Our purpose is to understand how tight our upper bounds are on the leading and remainder terms, relative to their true values and relative to upper bounds from the prior works~\cite{bp,katsBVM}. For the TV distance, we supplement this analysis with a new lower bound, which cleanly demonstrates the features of the model which affect how close the posterior is to Gaussian.

In Section~\ref{pmf:model}, we describe the model in more detail. In Section~\ref{sec:c3:pmf}, we compute or bound the key small quantities from our bounds in Section~\ref{sec:V:results}: $\tilep,\epsilon_3,\efour{\s},\efive{\s}$. In Section~\ref{sec:TV:pmf}, we study the TV error, and in Section~\ref{sec:mean:pmf} we study the mean error.
%

\subsection{The model: multinomial data, Dirichlet posterior}\label{pmf:model}
We consider a model in which we observe samples $X_1, \dots, X_n\in \{0,1,\dots,d\}$ drawn from a probability mass function (pmf) $\theta=(\theta_0,\dots,\theta_d)$ on $d+1$ states, i.e. the probability of state $j$ is $\theta_j$. The data is equivalent to a single sample $N\sim\mathrm{Multi}(n,\theta_0,\dots,\theta_d)\in \R^{d+1}$, where
$$N_j=\sum_{i=1}^n\ind\{X_i=j\},\quad j=0,\dots,d.$$ 
It is well known~\cite[Section 3.4]{bda3} that the Dirichlet prior $\theta\sim\Dir(\alpha_0,\dots,\alpha_d)$ is conjugate for a multinomial sample, with the posterior given by $\theta=(\theta_0,\dots,\theta_d)\sim\Dir(\alpha_0+N_0,\dots, \alpha_d+N_d).$ For simplicity of analysis, we consider a flat prior, corresponding to $\alpha_0=\dots=\alpha_d=1$. However, our results easily generalize to other Dirichlet priors. The Dirichlet distribution is degenerate since $\theta_0+\dots+\theta_d=1$ with probability 1, so instead we study the marginal distribution $\pi^d$ of $\theta^d=(\theta_1,\dots,\theta_d)$ under 
$$(\theta_0,\dots,\theta_d)\sim\pi=\Dir(1+N_0,\dots, 1+N_d).$$ The $d$-dimensional marginal $\pi^d$ is supported on the open set
\beqs\label{Thetapmf}
\Theta^d = \{\theta^d=(\theta_1,\dots,\theta_d)\mid 0<\theta_1+\dots+\theta_d<1, 0<\theta_i<1\,\forall i\}\subset\R^d.\eeqs We show in Appendix~\ref{subsub:vbarv} that if $N_j>0$ for all $j=0,1,\dots,d$, then $\pi^d\propto \e(-nv^d)$ for a strictly convex function $v^d$ on $\Theta^d$. Moreover, the unique global minimizer of $v^d$ is $\nb^d=(\nb_1,\dots,\nb_d)$, where
$$\nb=(\nb_0,\nb_1,\dots,\nb_d),\qquad \nb_j = \frac{N_j}{n}\in(0,1),\quad j=0,\dots,d.$$
We compute in Appendix~\ref{subsub:vbarv} that the LA to $\pi^d$ is $\lap^d = \mathcal N\l(\nb^d, n^{-1}\l[\mathrm{diag}\l(\nb^d\r)-\nb^d(\nb^d)^\T\r]\r).$ Moreover, we show that $\lap^d$ is the marginal of coordinates $(\theta_1,\dots,\theta_d)$ under the degenerate Gaussian distribution
 $$(\theta_0,\dots,\theta_d)\sim\lap=\mathcal N\l(\nb, n^{-1}\l[\mathrm{diag}(\nb)-\nb\nb^\T\r]\r).$$



\subsection{Verifying assumptions and computing coefficients}\label{sec:c3:pmf}
In the next section, we will apply the results from Section~\ref{sec:V:results} to the density $\pi^d\propto e^{-nv^d}$, defined as the marginal of coordinates $(\theta_1,\dots,\theta_d)$ under $(\theta_0,\dots,\theta_d)\sim\Dir(N_0+1,\dots, N_d+1)$. Here, we first verify the assumptions and compute the key small quantities appearing in our bounds: $\tilep$ from Definition~\ref{def:tilep} and $\epsilon_3,\efour{\s},\efive{\s}$ from Definition~\ref{epsdef}.

We have already stated that $v^d$ is strictly convex, with a unique global minimizer $\nb^d$. Hence Assumption~\ref{assume:1} is satisfied, and Assumption~\ref{assume:c0} is also satisfied with $\s_0=4$ and $\cgro=1$ provided $\epsilon_3/\sqrt d$ and $\efour{4}/\sqrt d$ are small enough absolute constants (recall Remark~\ref{vconvA2}). We now compute $\epsilon_3,\efour{\s},\efive{\s},\tilep$. 

\begin{lemma}[Bounds on $\epsilon_3,\epsilon_4,\epsilon_5$]\label{lma:eps3:pmf}Let $$\Nmin=\min_{j=0,1,\dots,d}\nb_j.$$ For $\epsilon_3$, the following equality holds.
\beq\label{eq:eps3:pmf}
\epsilon_3 =2\frac{1-2\Nmin}{\sqrt{1-\Nmin}}\frac{d}{\sqrt{n\Nmin}}\eeq 
If $\Nmin\leq1/3$, then we have the following upper and lower bounds on $\ek{k}{\s}$ for $k=3,4,5$:
\beqs\label{eq:eps4:pmf}
c\frac{d}{\sqrt{n\Nmin}}\leq \ek{k}{\s}\leq C\l(\sqrt{\frac{\s^2d}{n\Nmin}} +1\r)^{\frac{k}{k-2}}\frac{d}{\sqrt{n\Nmin}},
\eeqs
where $c$ and $C$ are absolute constants. 
\end{lemma}
\begin{remark}[Orders of magnitude]\label{rk:order:pmf}
Equations~\eqref{eq:eps3:pmf} and~\eqref{eq:eps4:pmf} show that 
$$\ethree{\s},\,\efour{\s},\,\efive{\s}\gtrsim d/\sqrt{n\Nmin}.$$ In particular, note that $\Nmin\leq1/(d+1)$, since $(d+1)\Nmin\leq \sum_{j=0}^d\nb_j=1$. Thus $d/\sqrt{n\Nmin}\geq d\sqrt d/\sqrt n$, so that $d^3\ll n$ is the minimal necessary requirement to ensure $\epsilon_3,\efour{\s},\efive{\s}$ are small. 
\end{remark}
Using the expression for $\epsilon_3$ and the upper bound on $\efour{\s}$, we can now finish checking Assumption~\ref{assume:c0}. The assumption is satisfied with $\s_0=4$ and $\cgro=1$ provided $\sqrt{d/(n\Nmin)}$ is smaller than a small enough absolute constant. \\

Next, we study the term $\tilep$, which bounds the leading order terms $L(g)$, $\LTV$ and $\|\Lhatm\|_{H_V}$ (as stated in Theorem~\ref{thm:Vgen}, Corollary~\ref{corr:corTV} and Corollary~\ref{thm:Vmean}, respectively). 
To do so, we first give a definition.
\begin{defn}[Chi-squared divergence and uniform pmf]
The chi-squared divergence between pmfs $q=(q_0,\dots, q_d)$ and $p=(p_0,\dots,p_d)$ is given by
$$\chi^2(q||p) = \sum_{j=0}^d\frac{(q_j-p_j)^2}{p_j}=\bigg[\sum_{j=0}^d\frac{q_j^2}{p_j}\bigg]-1.$$ Also, we let $\Unif_{d+1}$ denote the uniform distribution on $d+1$ states. It then holds,
\beq\label{chi-form}\chi^2(\Unif_{d+1}||\nb)  = \frac{1}{(d+1)^2}\bigg[\sum_{j=0}^d\nb_j^{-1}\bigg] - 1.\eeq
\end{defn}
Recall from~\eqref{tulip-ep} of Section~\ref{sec:V:epsilon} that $\tilep\leq \epsilon_3$ always. We now show that in fact, $\tilep\ll\epsilon_3$ in this multinomial/Dirichlet setting. 
\begin{lemma}[Value of $\tilep$]\label{lma:tildeps-pmf}
It holds
\beqs\label{eq:tildeps-pmf}
\tilep^2= \frac53\chi^2(\Unif_{d+1}||\nb)\frac{(d+1)^2}{n}+\frac{2(d^2-d)}{3n}
\eeqs Therefore, we have in particular the bounds
\beq\label{eq:tildeps-pmf-2}
\tilep\leq  \sqrt{\frac{2}{n}\sum_{j=0}^d\nb_j^{-1}} \leq 2\sqrt{\frac{d}{n\Nmin}}.\eeq
\end{lemma}
Comparing the second upper bound~\eqref{eq:tildeps-pmf-2} on $\tilep$ to the exact value~\eqref{eq:eps3:pmf} of $\epsilon_3$, we see that $$\tilep\les \epsilon_3/\sqrt d.$$ Using the first, tighter upper bound on $\tilep$ from~\eqref{eq:tildeps-pmf-2}, it is straightforward to show that $\epsilon_3/\tilep$ could be as large as $d$. This occurs if for example $\nb_0=\Nmin=1/d^2$ and $\nb_1=\dots=\nb_d=(1-\nb_0)/d$. 
\subsection{TV Analysis}\label{sec:TV:pmf}We now substitute the bounds we have just derived on $\tilep,\epsilon_3,\efour{\s}$ into the righthand side of our TV bound from Corollary~\ref{corr:corTV}. We also choose the radius $\s$ to be $\s=(n\Nmin/d^2)^{1/4}$.
\begin{prop}[TV upper bound via Corollary~\ref{corr:corTV}]\label{prop:TV:pmf}Suppose $\Nmin>0$, and $d^2/n\Nmin$ is smaller than a certain small absolute constant. Then 
\begin{align}
0\leq \LTV \leq\tilep\leq  \bigg(\frac{2}{n}\sum_{j=0}^d\nb_j^{-1}\bigg)^{\frac12} &\leq 2\sqrt{\frac{d}{n\Nmin}},\label{LTV-bd-pmf}\\
|\RTV|=\l|\TV(\pi^d,\lap^d)-\LTV \r|&\les \frac{d^2}{n\Nmin} +\tau,\label{RTV-bd-pmf}
\end{align} where $\tau=d\e(-C\sqrt{d}(n\Nmin)^{1/4})$. 
\end{prop}
We now state our overall bound on the TV error, using the coarsest bound on $\LTV$ for simplicity. Up to the exponentially negligible term $\tau$, we obtain
\beq\label{ourTV-pmf}\TV(\pi^d,\lap^d) \les \sqrt{\frac{d}{n\Nmin}}+\frac{d^2}{n\Nmin}.\eeq In contrast, in the unlikely best case scenario that $\ethree{\s}=\ethree{0}$, the bounds~\eqref{earlierTV},~\eqref{theirTV} of~\cite{katsBVM, bp} take the form
\beq\label{theirTV-pmf}\TV(\pi^d,\lap^d) \les \sqrt{\frac{d^2}{n\Nmin}}.\eeq \emph{Both terms in our bound~\eqref{ourTV-pmf} are an order of magnitude smaller than the bound~\eqref{theirTV-pmf}.}

Despite the improvement over~\cite{katsBVM,bp}, our bound~\eqref{ourTV-pmf} still leaves something to be desired. Indeed, note that the upper bound on $\LTV$ is small if $d/n\Nmin$ is small, whereas the upper bound on the remainder term $\l|\RTV \r|$ is small if $d^2/n\Nmin$ is small, a stronger requirement. When bounding the remainder, we do not have the kind of fine-grained control that we have on the leading order term. We therefore surmise that the bound on the remainder is too coarse. This hypothesis is bolstered by our analysis of the mean in Section~\ref{sec:mean:pmf} below, in which  we calculate the remainder term explicitly and compare its actual size to that of our theoretical upper bound. We find that the theoretical upper bound is an order of magnitude coarser than the true value of the remainder.\\

We next prove a \emph{lower} bound on the TV distance.
\begin{theorem}[Explicit TV lower bound]\label{prop:pmf:lb}
Let $\pi^d$ be the marginal distribution of $(\theta_1,\dots,\theta_d)$, when $(\theta_0,\theta_1,\dots,\theta_d)\sim\pi=\Dir(N_0+1,N_1+1,\dots,N_d+1)$, and let $\lap^d$ be the LA to $\pi^d$, which is given by the marginal distribution of $(X_1,\dots,X_d)$ when $(X_0,X_1,\dots,X_d)\sim\lap=\mathcal N(\nb,n^{-1}\l(\mathrm{diag}(\nb)-\nb\nb^\T\r))$. \\

\noindent If $\Nmin \geq 6/\sqrt{n+d}$ and $\TV(\Unif_{d+1},\nb)\geq 6/(d+1)$, then
\beq\label{lb:TV:pmf}
\TV(\pi^d,\lap^d)\geq \frac19\TV(\Unif_{d+1},\nb)\frac{d}{\sqrt{n}} 
\eeq where $\TV(\Unif_{d+1},\nb)=\frac12\sum_{j=0}^d\l|\frac{1}{d+1}-\nb_j\r|$.
\end{theorem}
This lower bound nicely mirrors the form of our upper bounds: it is given by a ``model''-dependent factor multiplied by the ``universal'' factor $d/\sqrt n$. The model-dependent term captures how far the pmf of empirical frequencies $\nb$ is from uniform. We see that the farther $\nb$ is from uniform, the worse the Gaussian approximation. This reflects the fact that as $\nb=(N_j/n)_{j=0}^d$ becomes farther from uniform, the Dirichlet distribution $\Dir(N_0+1,\dots, N_d+1)$ becomes more skewed; see for example the density plots in Figure 1 of~\cite{dirichlet}. Moreover, in this regime where $\nb$ is not too close to uniform, the presence of the universal factor $d/\sqrt n$ in the lower bound proves that $d^2\ll n$ is \emph{necessary} for accuracy of the LA.

Note that if there are absolute constants $0<C_1,C_2<1$ such that $\Nmin\geq C_1/d$ and $\TV(\nb, \Unif_{d+1})\geq C_2$, then the lower bound~\eqref{lb:TV:pmf} on the TV distance is \emph{of the same order as the upper bound~\eqref{LTV-bd-pmf} on $\LTV$}, namely, $d/\sqrt n$. The lower bound does not match the overall TV upper bound, due to the coarse bound on $R$. However, as discussed above, we suspect that there is a gap between the true value of $\RTV$ and its upper bound. Investigating how to prove tighter bounds on $\RTV$ is beyond the scope of this work.
\begin{remark}[Condition on $d$ and $n$]
Recall that $\Nmin\leq 1/(d+1)$ always. Therefore, the condition $\Nmin\geq 6/\sqrt{n+d}$ implies also the condition $d^2\leq Cn$.
\end{remark}
\begin{remark}[Proof technique]
We have previously discussed that the decomposition $\TV=\LTV+\RTV$ provides a technique to prove a lower bound: namely, we can subtract the upper bound on $\RTV$ from a lower bound on $\LTV$. This essentially reduces the problem of lower bounding $\TV$ to the problem of lower bounding $\LTV$. However, we do not take this approach in the proof of Theorem~\ref{prop:pmf:lb}. This is because it allows us to avoid imposing the more stringent condition $d^2 \ll n\Nmin$ required for our bound on $\RTV$ to be small (recall~\eqref{RTV-bd-pmf}). Instead, we derive our lower bound by considering the difference in the probability of half spaces under $\pi$ and under $\lap$. Specifically, we use 
$$\TV(\pi^d,\lap^d)\geq\sup_\I  \l|\PP_{\theta\sim\pi}\l(\sum_{i\in\I} \theta_i\leq\sum_{i\in\I}\nb_i\r)-1/2\r|,$$ and then further lower bound each of the probabilities in the supremum. Note that $1/2$ is the probability under the Gaussian $\lap$ that $\sum_{i\in\I} \theta_i\leq\sum_{i\in\I}\nb_i$, since $\theta-\nb$ has the same distribution as $\nb-\theta$ under $\lap$. The intuition for choosing to look at these half spaces is precisely that, unlike $\lap$, the measure $\pi$ is \emph{not} necessarily symmetric about $\nb$. Heuristically, if $\pi$ is very skewed, then the probability of some half space should be very far from $1/2$.\\
\end{remark}

\subsection{Analysis of the Laplace mean error}\label{sec:mean:pmf}
The mean of the Dirichlet distribution is known explicitly. For our posterior $\pi=\Dir(N_0+1,\dots,N_d+1)$, the mean $\bar\theta$ is given by
\beq\label{pmf:barth}\bar\theta = \l(\frac{N_i+1}{n+d+1}\r)_{i=0}^d=\l(\frac{\nb_i+1/n}{1+(d+1)/n}\r)_{i=0}^d,\eeq whereas the mode is 
\beq\label{pmf:hatth}\hat\theta = \l(\nb_i\r)_{i=0}^d = \nb.\eeq Correspondingly, the mean and mode of the $d$-dimensional marginal $\pi^d$ is $\bar\theta^d$ and $\hat\theta^d$, given by simply removing the zeroth coordinate of $\bar\theta$ and $\hat\theta$, respectively. In Appendix~\ref{app:skew:pmf}, we compute that the skew correction to the mode, which we name $\Lhatth^d$ to match the notation of this section, is given by removing the zeroth coordinate of the following vector:
\beq
\Lhatth = \frac1n\ind-\frac{d+1}{n}\nb.
\eeq
A short calculation using~\eqref{pmf:barth} and~\eqref{pmf:hatth} shows that the true Laplace mean approximation error $\bar\theta-\hat\theta$ is very nearly $\Lhatth$ itself. Namely, we have
$$\bar\theta-\hat\theta = \frac{1}{1+(d+1)/n}\Lhatth = \Lhatth - \frac{d+1}{n+d+1}\Lhatth.$$
Furthermore, we can compute $\|\Lhatth^d\|_{H_V}$ explicitly, and therefore also $\|\bar\theta^d-\hat\theta^d-\Lhatth^d\|_{H_V}$. In Appendix~\ref{app:skew:pmf}, we obtain
\begin{align}
\|\Lhatth^d\|_{H_V} &=  \sqrt{\chi^2(\mathrm{Unif}_{d+1}||\nb)}\frac{d+1}{\sqrt n}\label{Lhat-mean-pmf}\\
\|\bar\theta^d-\hat\theta^d-\Lhatth^d\|_{H_V} &=\sqrt{\chi^2(\mathrm{Unif}||\nb)}\frac{(d+1)^2}{(n+d+1)\sqrt n} \les d^{-1}\sqrt{\chi^2(\mathrm{Unif}||\nb)}\l(\frac{d}{\sqrt{n}}\r)^3\label{Lhat-mean-pmf-R}
\end{align}
We see that if $\nb=\Unif_{d+1}$, then the mean equals the mode, and in line with this, the skew correction $\Lhatth^d$ also vanishes. This is in contrast with the TV bound. Indeed, recall that $\LTV\leq\tilep$. From the formula~\eqref{eq:tildeps-pmf} for $\tilep$ we see that if $\nb=\Unif_{d+1}$ and $d\geq2$, then the $\chi^2$ term in parentheses vanishes but the constant remains, so that we still have $\LTV\leq \tilep \asymp d/\sqrt n$. This is reasonable, since a Dirichlet distribution is never equal to a Gaussian distribution, so we should not expect the TV distance itself to vanish.

We now contrast the precise quantities~\eqref{Lhat-mean-pmf} and~\eqref{Lhat-mean-pmf-R} with the upper bounds on these quantities that we get from Corollary~\ref{thm:Vmean}.
\begin{prop}[Mean upper bound via Corollary~\ref{thm:Vmean}]\label{prop:mean:pmf}Suppose $\nb_0,\dots,\nb_d>0$. Then
\begin{align}
\|\Lhatth^d\|_{H_V} \les \l(\sqrt{\chi^2(\Unif_{d+1}||\nb)}+1\r)\frac{d}{\sqrt n}\label{Lmean-bd-pmf}\end{align}
If additionally, $d^2/n\Nmin$ is smaller than a certain small absolute constant, then
\begin{align}
\|\bar\theta^d-\hat\theta^d-\Lhatth^d\|_{H_V}&\les\l(\frac{d}{\sqrt{n\Nmin}}\r)^3 +\tau,\label{Rmean-bd-pmf}
\end{align} where $\tau=d\e(-C\sqrt{d}(n\Nmin)^{1/4})$. 
\end{prop}
Note that whenever $\chi^2(\Unif_{d+1}||\nb)\gtrsim1$, the upper bound~\eqref{Lmean-bd-pmf} on $\|\Lhatth^d\|_{H_V}$ is greater than the true value~\eqref{Lhat-mean-pmf} of $\|\Lhatth^d\|_{H_V}$ only by a constant factor. 

In contrast, the upper bound~\eqref{Rmean-bd-pmf} on $\|\bar\theta^d-\hat\theta^d-\Lhatth^d\|_{H_V}$ is a drastic overestimate of the true value of this remainder norm. To see this, we start with the true value and take two upper bounds:
\beqs
d^{-1}\sqrt{\chi^2(\mathrm{Unif}||\nb)}\l(\frac{d}{\sqrt{n}}\r)^3 &\les d^{-2}\sqrt{\sum_{j=0}^d\nb_j^{-1}}\l(\frac{d}{\sqrt{n}}\r)^3\\
&\leq d^{-3/2}\Nmin^{-1/2}\l(\frac{d}{\sqrt{n}}\r)^3.\eeqs The quantity in the second line is smaller than the upper bound $(d/\sqrt{n\Nmin})^3$ from~\eqref{Rmean-bd-pmf} by a factor of $d^{3/2}\Nmin^{-1}$, which is \emph{at least} $d^{2.5}$!

\section{Application to logistic regression posterior}\label{sec:log}
\renewcommand{\ground}{\beta}
In this section, we specialize our results to the setting of a posterior in a logistic regression model with Gaussian design, and a flat or Gaussian prior on the coefficient vector. In Section~\ref{sec:log:set}, we describe the model and provide the explicit formulas for the skew-corrected LA $\corlap$, and the mean correction~\eqref{Lm-def} in particular. 

Then in Section~\ref{sec:log:whp}, we obtain high-probability versions of the upper bounds from Section~\ref{sec:V:results} on both the corrected and uncorrected LA errors. 

In Section~\ref{sec:num}, we numerically demonstrate how incorporating the skew correction improves the approximation to $\pi$ with respect to estimating the mean and set probabilities. We also provide numerical lower bounds on the leading order terms $\LTV$ and $\Lhatm$ in the decomposition of the  TV and mean errors. Finally, we discuss the use of our polynomial time computable upper bound on $\LTV$ as a proxy for $\TV(\pi,\lap)$.

Omitted proofs from this section can be found in Appendix~\ref{app:sec:log}.
\subsection{Setting and formulas for leading order terms}\label{sec:log:set}
In logistic regression with Gaussian design, we are given i.i.d. feature vectors $X_i\in\R^d$ drawn according to a Gaussian distribution, and labels $Y_i\in\{0,1\}$ corresponding to $X_i$, $i=1,\dots,n$. The distribution of a label $Y$ conditioned on its corresponding feature vector $X$ is modeled as $Y\vert X\sim\mathrm{Bernoulli}(\sigma(b^\T X))$, i.e.
\beq\label{pYXb}p(Y\mid X,b) =\sigma(b^\T X)^{Y}(1-\sigma(b^\T X))^{1-Y} = \e(YX^\T b - \psi(X^\T b))\eeq where $\sigma$ is the sigmoid $\sigma(t)=(1+e^{-t})^{-1}$,  $\psi(t)=\log(1+e^t)$ is the antiderivative of $\sigma$, and $b\in\R^d$ is the unknown coefficient vector, the parameter of interest. We use $b$ in place of $x$ in this section. We consider a Gaussian prior $\mathcal N(0,\Sigma_0)$ on $b$, and we include the case of a flat prior by allowing $\Sigma_0^{-1}=0$. Then the posterior $\pi$ is given by
\beq\label{pi-def}
\pi(b\mid (X_i,Y_i)_{i=1}^n)\propto \e\l(\sum_{i=1}^nY_iX_i^\T b - \sum_{i=1}^n\psi(X_i^\T b) - \frac12b^\T \Sigma_0^{-1}b\r),\quad b\in\R^d.\eeq We can write $\pi\propto e^{-V}$, where
\beq\label{vndef}V(b)=nv(b)= -\sum_{i=1}^nY_iX_i^\T b + \sum_{i=1}^n\psi(X_i^\T b) + \frac12b^\T \Sigma_0^{-1}b.\eeq We assume the model is well-specified with ground truth parameter $b=\ground$, and we take $\mathcal N(0,M)$ for the design distribution. Thus the joint feature-label distribution is given by
\beq\label{joint}X_i\iid\mathcal N(0, M),\quad Y_i\vert X_i\sim\mathrm{Bernoulli}(\sigma(\ground^\T X_i)),\quad i=1,\dots,n.\eeq
Now, in Proposition~\ref{prop:log} below, we will show that $V$ has a unique global minimizer with high probability. Note that this is not guaranteed when the prior is flat, since then $V$ is strictly but not strongly convex. When a unique global minimizer exists, we denote it $\bhat$:
\beq\label{bhatdef}\bhat=\arg\min_{b\in\R^d}V(b).\eeq 
\begin{lemma}[Skew correction and upper bound formulas]\label{lma:skew:log}
Assuming $\bhat$ exists, we have the following formulas for $H_V$, the skew correction function $\Lhat$, and the mean skew correction $\delta\hat b$:
\begin{align}
H_V &= \nabla^2V(\hat b)=\sum_{i=1}^n\psi(\hat b^\T X_i)X_iX_i^\T  + \Sigma_0^{-1},\label{logHV}\\
\Lhat(b) &= -\frac16\sum_{i=1}^n\psi'(\hat b^\T X_i)(b^\T X_i-\hat b^\T X_i)^3,\label{logLhat}\\
\Lhatb  &= -\frac12\sum_{i=1}^n\psi'(\hat b^\T X_i)(X_i^\T H_V^{-1}X_i)H_V^{-1}X_i.\label{logLhatm}
\end{align}
Hence, the skew-adjusted Laplace approximation to a general observable expectation and to the posterior mean are given by
\beqs\label{LapGLM}
\E_{b\sim\pi}[g(b)]&\approx\E_{b\sim\lap}\l[g(b)\l(1-\frac16\sum_{i=1}^n\psi'(x_i^\T \bhat)(X_i^\T b-X_i^\T \bhat)^3\r)\r],\\
\bar b&\approx \bhat -\frac12\sum_{i=1}^n\psi'(\bhat^\T x_i)(X_i^\T H_V^{-1}X_i)H_V^{-1}X_i,
\eeqs where $\lap=\mathcal N(\bhat, H_V^{-1})$ and $\bpi=\int bd\pi(b)$ is the posterior mean. Moreover, we have the following upper bound on $\LTV$:
\beqs\label{tildec3form}
\LTV^2\leq\tilep^2 = \sum_{\ell,m=1}^n\sigma(X_\ell^T\bhat)&\sigma(X_m^T\bhat)\bigg[\frac16\l(X_\ell^TH_V^{-1}X_{m}\r)^3\\
&+\frac14\l(X_\ell^TH_V^{-1}X_{m}\r)\l(X_\ell^TH_V^{-1}X_{\ell}\r)\l(X_{m}^TH_V^{-1}X_{m}\r)\bigg]
\eeqs
\end{lemma}
\noindent The proof of~\eqref{logHV}-\eqref{logLhatm} follows by a short calculation using~\eqref{Lhatdef} and~\eqref{Lm-def}, and the fact that $\nabla^3V(b)=\sum_{i=1}^n\psi'(b^\T X_i)X_i^{\otimes 3}$. The proof of~\eqref{tildec3form} is in Appendix~\ref{app:sec:log}. We comment on the practical use of the upper bound~\eqref{tildec3form} in Section~\ref{dimsubsub}.
%
%
%
\begin{remark}
The skew-adjustment formulas~\eqref{LapGLM} actually apply to any generalized linear models with observed data $(X_i, Y_i)\in\R^d\times \R$, $i=1,\dots,n$ modeled as $Y\vert X\sim p(\cdot\mid X^\T b)$, where $p$ is an exponential family in canonical form: $p(y\mid\theta)=h(y)\e(y\theta-\psi(\theta))$ for some base density $h$ and normalizing constant $e^{\psi}$. The base density does not affect the posterior distribution. \end{remark}

\subsection{Verification of assumptions and high-probability error bounds}\label{sec:log:whp}
In this section, we use the results of Section~\ref{sec:V:results} to rigorously prove bounds on the error of the original and corrected LA, for posteriors arising from logistic regression. Following our work~\cite{katsBVM} on the Bernstein-von Mises theorem, we formulate our bounds through a statistical lens. This means that, instead of proving bounds in which the righthand side is a function of a fixed set of samples $(X_i, Y_i)_{i=1}^n$, we prove bounds which hold uniformly over all sets of samples falling in an event of high probability under the joint distribution of $(X_i, Y_i)_{i=1}^n$.

As discussed in the introduction, proving high-probability bounds allows us to get a handle on the dimension dependence of the model-dependent terms $c_3,c_4(\s),c_5(\s)$ appearing in $\epsilon_3,\efour{\s},\efive{\s}$. Although high-probability bounds are necessarily coarser than bounds tailored to a given realization of the data, it is very difficult to bound the coefficients $c_3,c_4(\s),c_5(\s)$ for given data, and unclear how these coefficients scale with $d$. Indeed, note that
$$c_k(\s) = \sup_{b\in\U(s)}\|\nabla^{k}v(b)\|_{H_{v}}=\sup_{b\in\U(s)}\sup_{u\neq0}\frac{\frac1n\sum_{i=1}^n\psi^{(k)}(X_i^{\T}b)(X_i^{\T}u)^3}{\l(\frac1n\sum_{i=1}^n\psi\big(X_i^{\T}\bhat\big)(X_i^{\T}u)^2 +u^\T(n\Sigma_0)^{-1}u\r)^{3/2}}$$ for $k=3,4,5$.

The following proposition proves that the assumptions on the random function $v$ are satisfied uniformly over a high probability event, and also bounds $c_3,c_4(\s),c_5(\s)$ uniformly over this event. 
\begin{prop}\label{prop:log}Let $(X_i, Y_i)$, $i=1,\dots, n$ be i.i.d., with $X_i\sim\mathcal N(0, M)$ and $Y_i\vert X_i\sim \mathrm{Bernoulli}(\sigma(X_i^\T \ground))$. Let $V$ be as in~\eqref{vndef}, i.e. the negative log posterior of the logistic regression coefficient vector given a Gaussian or flat prior $\pi_0=\mathcal N(0,\Sigma_0)$. There exist constants $C(\|\ground\|_{\M})>0$ such that if 
$$ \|\Sigma_{0}^{-1}\|_{\M} \leq C(\|\ground\|_{\M})n,\quad (\sqrt d\vee\log(2n))\sqrt{d/n}\leq C(\|\ground\|_{\M})\wedge1,$$ then the intersection of the following events holds with probability at least \\$1 -C\e(-C(\|\ground\|_{\M})\max(\sqrt d,\log n))$:
\begin{enumerate}
\item There exists a unique global minimizer $\hat b$ of $v$, and $H_{v}=\nabla^{2}v(\hat b)\succeq C(\|\ground\|_{\M})\M$.
\item We have the bounds
\beq\label{Hvnablak}\sup_{b\in\R^{d}}\|\nabla^{k}v(b)\|_{H_{v}} \leq C(\|\ground\|_{\M})\l(\frac{d^{k/2}}{n}\vee1\r),\qquad\forall k=3,4,5.\eeq
\item The lower bound~\eqref{assume:c0:eq} holds with $\s_0=4$ and $\cgro=1$.
\end{enumerate}
The probability is with respect to the joint feature-label distribution~\eqref{joint}.
\end{prop}
On the high probability event from this proposition, we see that Assumptions~\ref{assume:1} and~\ref{assume:c0} are satisfied. Thus Theorems~\ref{thm:Vgen},~\ref{thm:odd} and Corollaries~\ref{corr:corTV},~\ref{thm:Vmean},~\ref{thm:Vcov} can all be applied. Furthermore,~\eqref{Hvnablak} shows that if $\|\ground\|_M\leq C$ and $d^2/n\leq C$, then w.h.p. we have $c_3\leq C$, $\sup_{\s\geq0}c_4(\s)\leq C$, and $\sup_{\s\geq0}c_5(\s)\leq C\sqrt d$, for an absolute constant $C$. Hence
\beq\label{c345}\epsilon_3\leq C\epsilon,\qquad\sup_{\s\geq0} \efour{\s}^2\leq C\epsilon^2, \qquad \sup_{\s\geq0}\efive{\s}^3\leq C\sqrt d\epsilon^3\eeq w.h.p., where $\epsilon=d/\sqrt n$. Therefore, logistic regression is an example in which the small quantities $\epsilon_3,\efour{\s},\efive{\s}$ are not all bounded by $d/\sqrt n$ and moreover, they grow with $d$ at different rates. But interestingly enough, $\efive{\s}^3$ only ever appears in our bounds with a factor $d^{-1/2}$ in front; see~\eqref{godd-bd} and~\eqref{Rc5}. As a result, the high probability bounds we now prove \emph{depend purely on powers of $d/\sqrt n$}. 

The following corollary is an application of Theorems~\ref{thm:Vgen},~\ref{thm:odd} and the subsequent corollaries with a particular choice of $\s$; see Appendix~\ref{app:sec:log} for more details.
\begin{corollary}[High probability Laplace error bounds for logistic regression]\label{corr:logreg}
Consider the setting of Proposition~\ref{prop:log}. Let 
\beq\label{eps-tau}\epsilon=d/\sqrt n,\qquad\tau=d\e(-Cn^{1/4}\sqrt d).\eeq Suppose $\|\ground\|_M\leq C$. If $\epsilon$ and $n^{-1}\|\Sigma_0^{-1}\|_\M$ are smaller than sufficiently small absolute constants, then the following bounds hold on an event of probability at least $1 - C\e(-C\max(\sqrt d,\log n))$:
\beqs\label{bds-gam-log-1}
 \bigg|\int  gd\pi-\int gd\lap\bigg|&\lesssim \l(\af{g}\vee1\r)\tau+\begin{cases}\epsilon\vspace{5pt}\\ 
\epsilon^2,\quad\text{g even about $\bhat$},\end{cases}\\
 \bigg|\int  gd\pi-\int gd\corlap\bigg|&\lesssim \l(\af{g}\vee1\r)\tau+\begin{cases}\epsilon^2\vspace{5pt}\\ 
\epsilon^3 ,\quad\text{g odd about $\bhat$}.\end{cases}
\eeqs
for any ``standardized'' $g$ as in Theorem~\ref{thm:Vgen}. Furthermore, we have the following TV bounds:
\beqs
|\LTV|\les\epsilon,\qquad \l|\TV(\pi,\lap)-\LTV\r|\leq\TV\l(\pi,\; \corlap\r)\lesssim \epsilon^2+ \tau,\eeqs
the following mean error bounds:
\beqs
\|\Lhatb\|_{H_V} \les\epsilon,\qquad  \|\bpi - \hat b -\Lhatb \|_{H_V}&\lesssim \epsilon^3+\tau,\eeqs
and the following covariance error bound:
\beq
 \|H_V^{1/2}\big(\Sigpi -  H_V^{-1}\big)H_V^{1/2}\|\lesssim \epsilon^2+\tau.
\eeq
In each bound, the absolute constant suppressed by the $\lesssim$ symbol is deterministic (independent of the data realization).
 \end{corollary} 
These high probability bounds reveal that $d\ll\sqrt n$ is sufficient for the Laplace approximation errors to be small, for a typical realization of the data. Furthermore, we note that our bound on the TV distance between $\pi$ and $\corlap$ can be construed as a skewed Bernstein-von Mises Theorem (BvM) in the logistic regression setting. To state the result asymptotically, we have shown that if $n\to\infty$ and $d^2/n\to0$, then $\TV\l(\pi,\; \corlap\r)$ converges to zero at rate $d^2/n$ with probability tending to 1 under the ground truth data distribution. This is analogous to the result of~\cite{durante2023skewed}, which proves a skewed BvM at rate $n^{-1}$ in the low-dimensional regime, using their skew-corrected approximation $\hat p_{\mathrm{GSN}}$.
 \begin{remark}[Strength of prior regularization]\label{rk:prior-strength}A large inverse covariance $\Sigma_0^{-1}$ corresponds to strong prior regularization. In Corollary~\ref{corr:logreg}, we have only required that $\|\Sigma_0^{-1}\|_\M\leq\delta n$ for a sufficiently small but absolute constant $\delta$. Thus the contribution $b^T\Sigma_0^{-1}b$ to the negative log posterior $V$ from~\eqref{vndef} is comparable to the contribution from the negative log likelihood --- i.e. both can be order $\mathcal O(n)$. In other words, Corollary~\ref{corr:logreg} does not require the prior to wash away in the large $n$ limit.
 
Still, it may seem surprising that we require $\|\Sigma_0^{-1}\|_M$ to be sufficiently small, since in Remark~\ref{rk:deriv:ell}, we noted that the stronger a Gaussian prior, the better the Laplace approximation. The reason for this requirement is that it allows us to show that the MAP $\hat b$ is close to the MLE, which is in turn close to the ground truth parameter. Thus $\hat b$ is bounded with high probability, which is useful for showing that the Laplace approximation bounds hold uniformly over a high probability event. Hence this is purely related to our proof technique. For a fixed realization of the data, it remains true that stronger prior regularization is more favorable for the LA.
 \end{remark}
\begin{remark}Recall that $H_v\succeq CM$ with high probability by point 1 of Proposition~\ref{prop:log}. This allows us to replace the $H_V$-weighted norms in the mean and covariance bounds by $M$-weighted norms. For example, if $M=I_d$, we get
\beqsn
\sqrt n\|\bar b- \hat b\|&\lesssim \epsilon+ \tau,\\
\sqrt n \|\bar b - (\hat b +\Lhatb )\|&\lesssim \epsilon^3+\tau,\\
n\|\Var_\pi(b) -  H_V^{-1}\|&\lesssim \epsilon^2+\tau.
\eeqsn 
\end{remark}

\subsection{Numerics: skew correction and lower bounds}\label{sec:num}
In this section, we demonstrate the improvement in accuracy due to the skew correction in dimension $d=2$. We then show for a range of $(d,n)$ pairs that the leading order terms $\LTV$ and $\|\Lhatb\|_{H_V}$ of the TV and mean errors are bounded from below by $d/\sqrt n$. In all of the simulations in this section, we take a flat prior ($\Sigma_0^{-1}=0$), ground truth $\ground=e_1=(1,0,\dots,0)$, and covariance $M=I_d$ for the Gaussian feature distribution, i.e. we draw the samples $X_i$ according to a standard normal distribution.

\subsubsection{Skew correction in $d=2$}
\begin{figure}
\centering
\vspace{-5pt}
\includegraphics[width=0.49\textwidth]{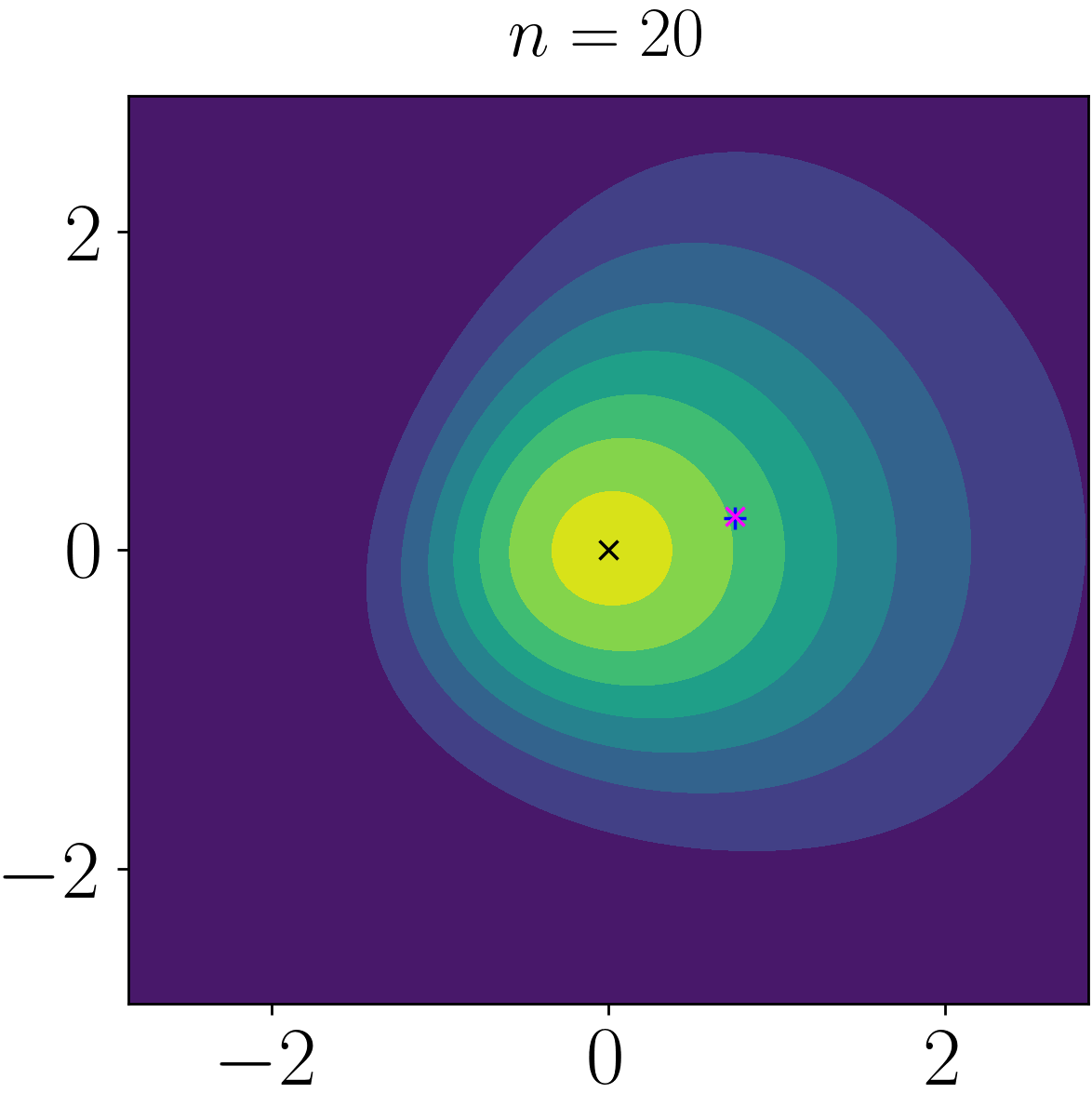}
\includegraphics[width=0.49\textwidth]{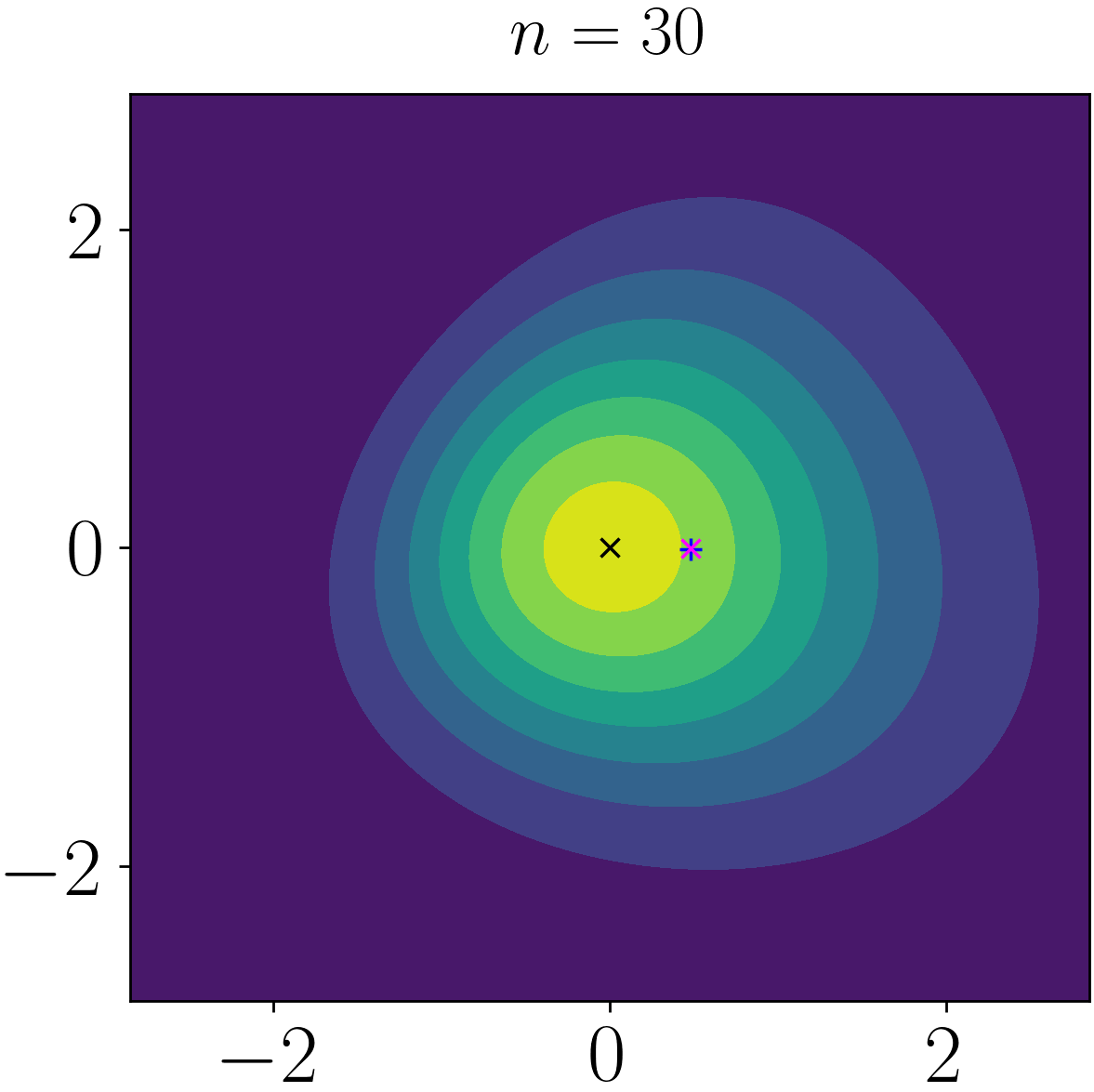}
\includegraphics[width=0.49\textwidth]{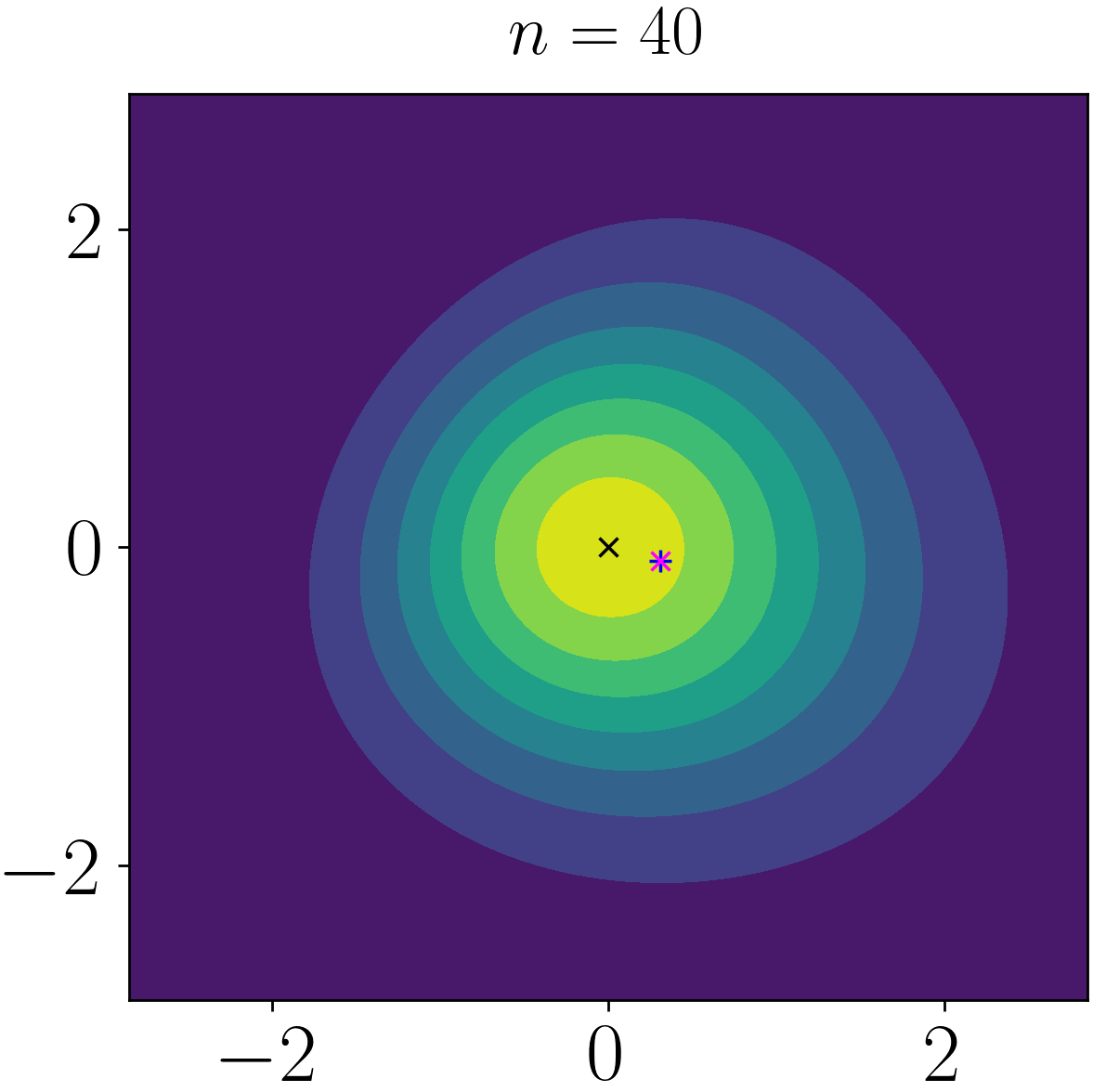}
\hspace{-7pt}
 \includegraphics[width=0.51\textwidth]{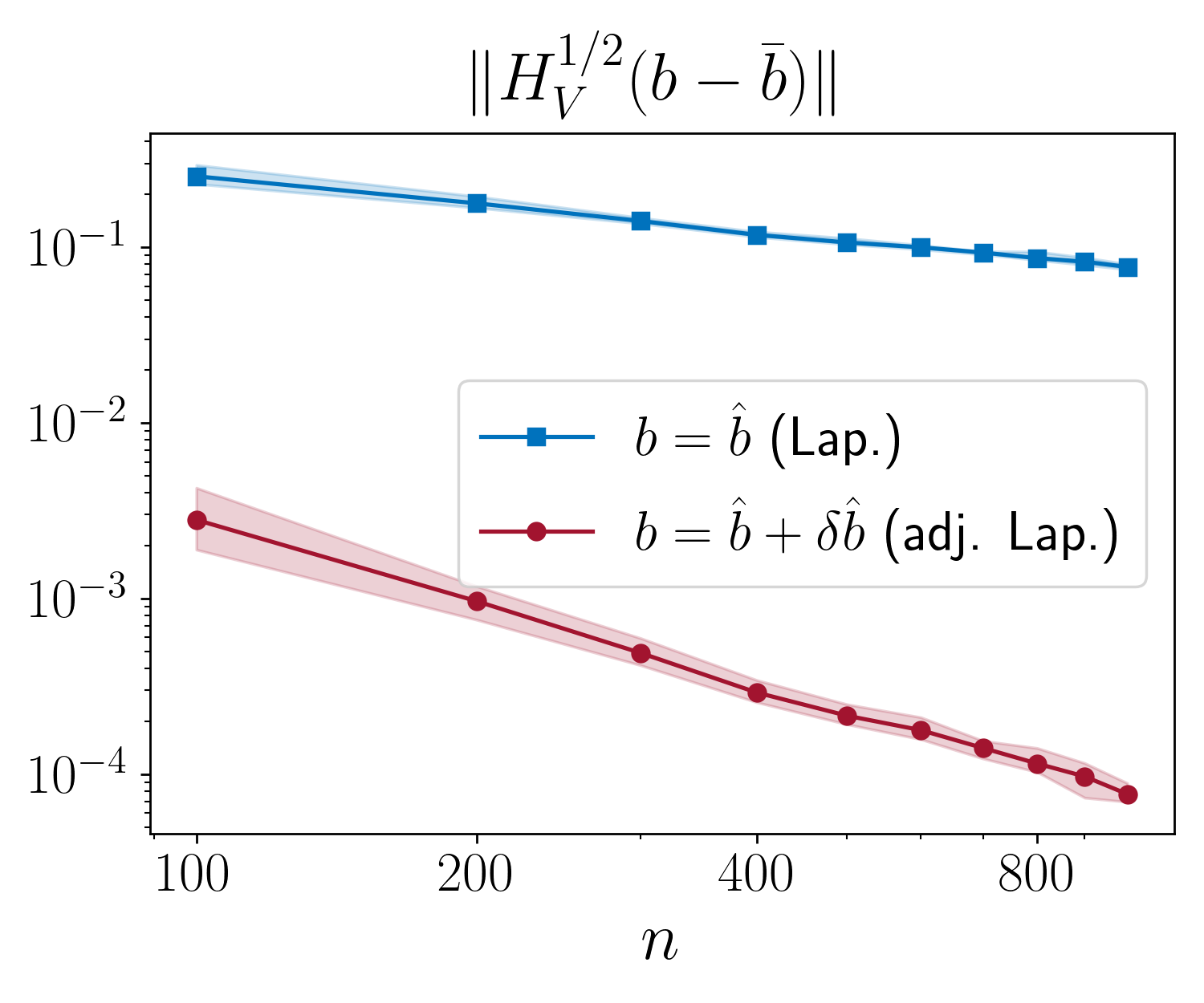}
\caption{Top and bottom left: contour plot of the density $\rho:=T_{\#}\pi$ for $n=20,30,40$. Here, $\pi\propto e^{-nv}$ is the logistic regression posterior from~\eqref{pi-def}, and the rescaling map $T(b)=H_V^{1/2}(b-\bhat )$ shifts the mode to zero, marked by a black x. The blue + marks the location of the (rescaled) true mean, and the magenta x marks the location of the (rescaled) skew adjusted Laplace approximation to the mean. On the $O(1)$ scale depicted here, the true mean is indistinguishable from the skew adjusted Laplace approximation to the mean. Bottom right: uncorrected and corrected Laplace approximation error to the mean, as a function of $n$. The slopes of the best-fit lines to these two curves are $-0.51$ and $-1.52$, respectively.}
 \label{fig:scaled-pdf}
 \end{figure}
The top and bottom left plots in Figure~\ref{fig:scaled-pdf} display the contours of the densities $\rho:=T_{\#}\pi$ for $n=20,30,40$, for a particular realization of the Gaussian features $X_i$. Here, $T$ is the rescaling map $T(b)=H_V^{1/2}(b-\bhat)$. Note that the pushforward by $T$ shifts the mode from $\bhat$ to zero, marked by a black X in the figure. As $n$ increases, we expect the approximation $\pi\approx\mathcal N(\bhat, H_V^{-1})$ to improve, or equivalently, the approximation $\rho\approx\mathcal N(0, I_d)$ to improve. And indeed, we see from the top and bottom left plots that the contours of $\rho$ become closer to perfect circles about zero for larger $n$. Nevertheless, $\rho$ is still noticeably skewed for all three values of $n$. 

This is evident from the contours themselves, and also from the fact that the mean, marked by a blue +, is shifted away from zero. Thus the mode is not a very accurate approximation to the mean. On the other hand, the skew-adjusted Laplace approximation to the mean, marked by a magenta X, is extremely accurate. In fact, on the scale depicted in these contour plots, the true mean is indistinguishable from the skew-adjusted Laplace estimate. The bottom right plot in Figure~\ref{fig:scaled-pdf} quantifies the corrected and uncorrected mean approximation error as a function of $n$, on a log-log scale. The solid blue and red curves are the average approximation errors based on ten posteriors, corresponding to ten different draws of $n$ samples. The slopes of the best-fit lines to these two curves are $-0.51$ and $-1.52$, respectively, as predicted by Corollary~\ref{thm:Vmean}. The shaded regions represent the spread of the error for the middle five of ten posteriors. The true mean is computed using two-dimensional quadrature. 
 
\begin{figure}
\centering
\vspace{-5pt}	
\includegraphics[width=0.47\textwidth]{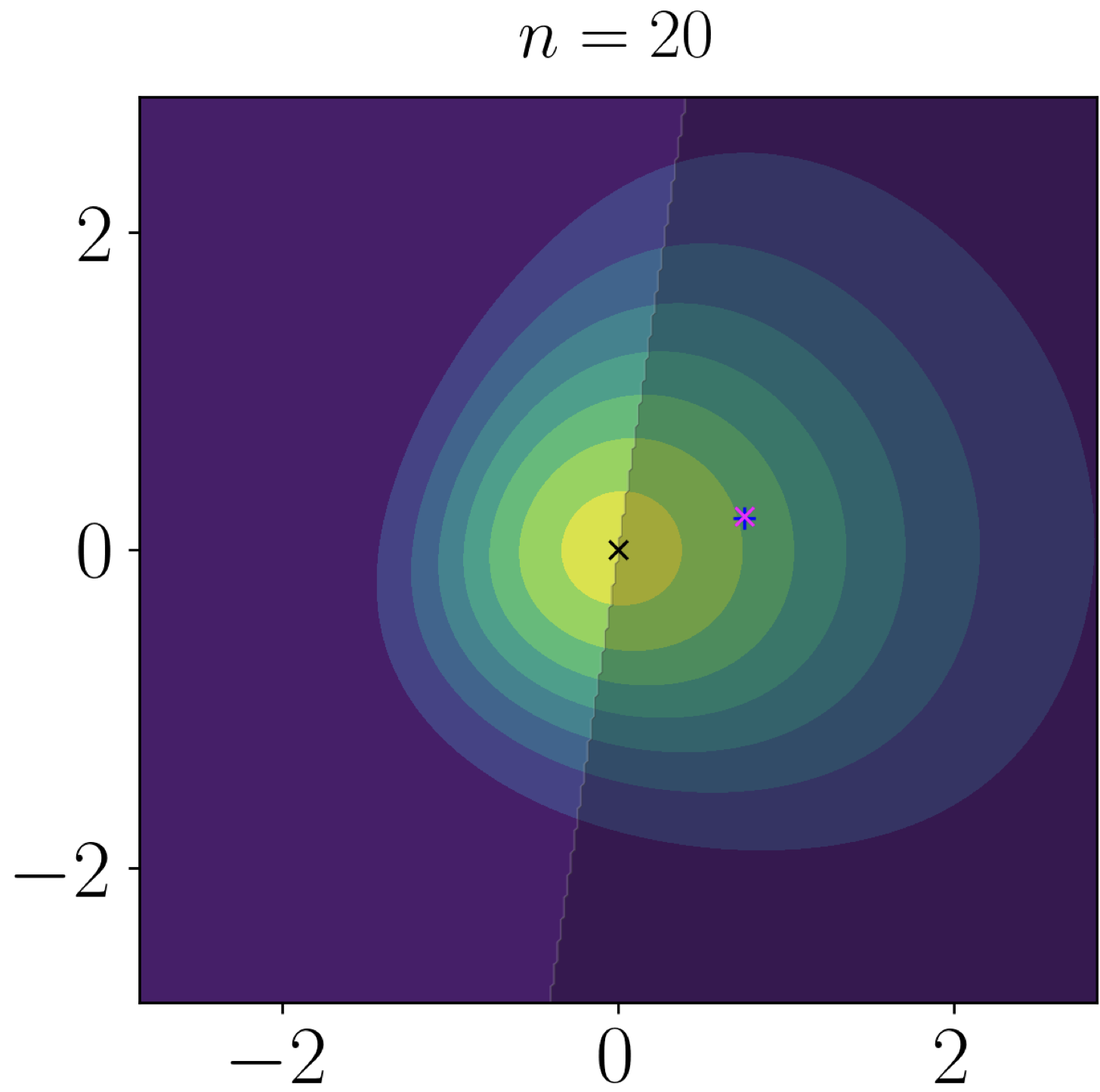}
\includegraphics[width=0.52\textwidth]{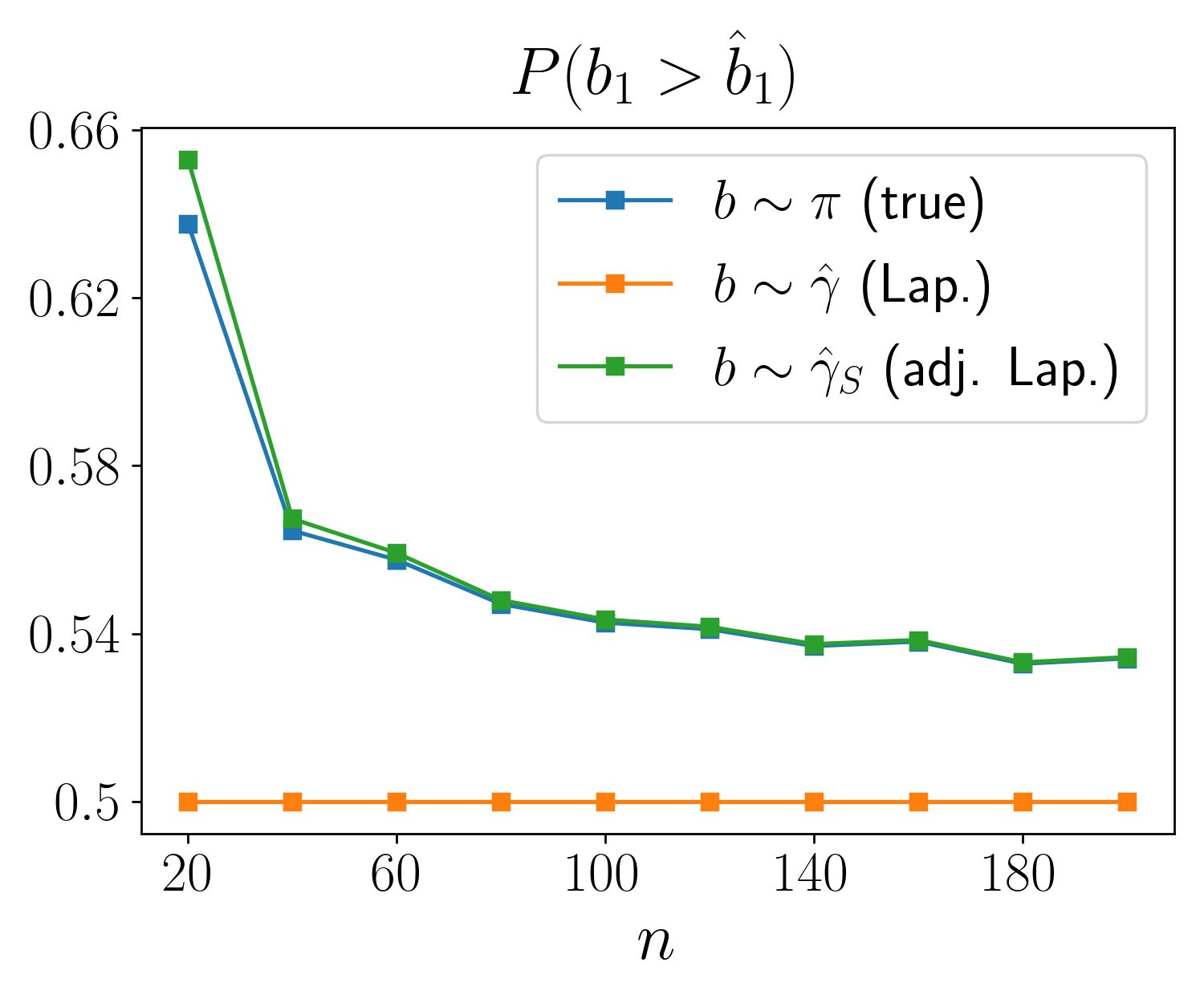}
\includegraphics[width=0.59\textwidth]{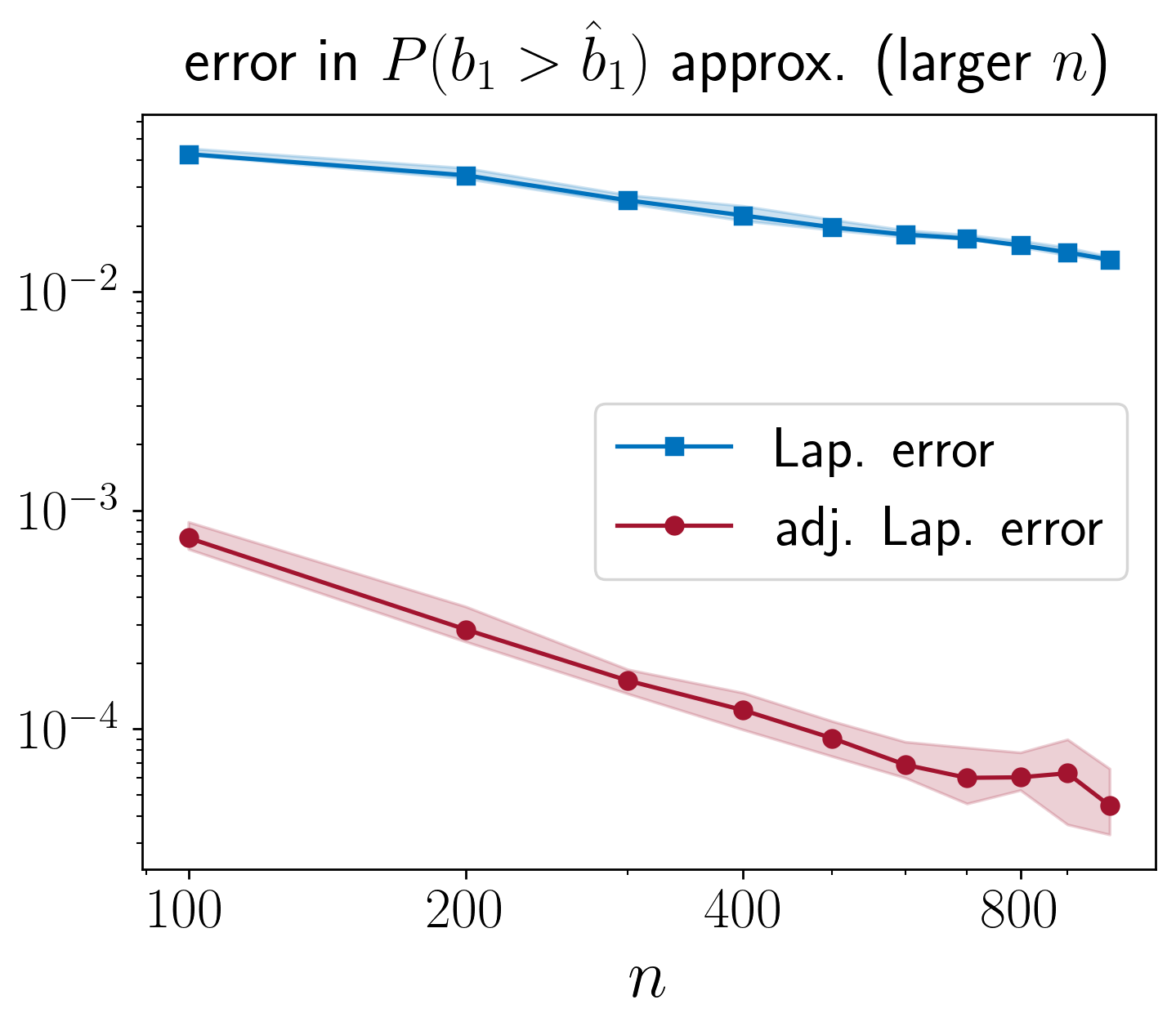}
\caption{Top left: contour plot of the rescaled logistic regression posterior density $\rho=T_{\#}\pi$ with $n=20$ samples. The darker region is the set $A:=\{T((b_1,b_2))\mid b_1\geq\bhat_1\}$, where $b=(b_1,b_2)$ is the two-dimensional coefficient parameter. It is clear from the plot that due to the skewness, $\rho$ assigns more mass to $A$ than to $A^c$, and therefore $\pi(b_1\geq\bhat_1)=\rho(A)>1/2$. Top right: probability of the event $\{b_1\geq\bhat_1\}$ as a function of $n$, under the posterior $\pi$ (blue), the Laplace distribution $\lap$ (orange), and the adjusted Laplace ``distribution'' $\corlap=\lap(1+\Lhat)$ (green). Note that $\lap(b_1\geq\bhat_1)=1/2$, since the event is symmetric about the mode and $\lap$ has no skew, while $\pi(b_1\geq\bhat_1)>1/2$ as expected. We see that $\corlap(b_1\geq\bhat_1)$ is much closer to the true probability, showing that incorporating the term $\Lhat$ into the Laplace approximation accounts for skew. Bottom: the errors $|\pi(b_1\geq\bhat_1)-\lap(b_1\geq\bhat_1)|$ (blue) and $|\pi(b_1\geq\bhat_1)-\corlap(b_1\geq\bhat_1)|$ (red), as a function of $n$, on a log-log scale. The slopes of the best-fit lines are -0.49 and -1.18, respectively, approximately matching the predictions from Corollary~\ref{corr:logreg}.}
 \label{fig:errs}
 \end{figure}
 
Figure~\ref{fig:errs} demonstrates the benefit of using the skew adjustment to improve the approximation to a posterior probability. The top left plot is the contour plot of $\rho=T_{\#}\pi$ for $n=20$, the same as the top left contour plot from Figure~\ref{fig:scaled-pdf}. The shaded region is the set $A:=\{T((b_1,b_2))\mid b_1\geq\bhat_1\}$. Due to the skew of $\rho$, it is clear that $\pi(b_1\geq\bhat_1)=\rho(A)$ is significantly greater than 1/2, which is the Laplace approximation to this set probability. This is confirmed in the top right plot, which depicts the true probability $\pi(b_1\geq\bhat_1)$ as a function of $n$ (blue), the Laplace approximation $\lap(b_1\geq\bhat_1)=1/2$ (orange) and the adjusted Laplace approximation $\corlap(b_1\geq\bhat_1)$ (green), as a function of $n$, for relatively small values of $n$. The true probability is computed using quadrature, and the adjusted Laplace probability is computed by drawing $10^6$ Gaussian samples. This plot demonstrates that for small values of $n$, the skew-adjustment has a very significant effect in improving the accuracy. 

Finally, the bottom plot in the figure depicts the errors $|\pi(b_1\geq\bhat_1)-\lap(b_1\geq\hat b_1)|$ (blue) and $|\pi(b_1\geq\bhat_1)-\corlap(b_1\geq\bhat_1)|$ (red), as a function of $n$, on a log-log scale. The slopes of the best-fit lines to the solid blue and red curves (which represent the average error based on ten posteriors) are -0.49 and -1.18, respectively. The shaded regions represent the spread of the error for the middle five of ten posteriors.

\subsubsection{Lower bounds}\label{dimsubsub}
Thanks to Corollary~\ref{corr:logreg}, we know that the remainder terms $|\TV(\pi,\lap)-\LTV|$ and $\|\bpi - \hat b -\Lhatb\|_{H_V}$ for the TV and mean errors are bounded above by $d^2/n +\tau$ with high probability, where $\tau$ is exponentially small when either $d$ or $n$ is large. We now combine this theoretical result with a numerical lower bound on the leading order terms $L:=\LTV$ and $L:=\|\Lhatb\|_{H_V}$. Namely, we show that both are bounded from below by $d/\sqrt n$.

We take a sequence of increasing dimensions $d$, and for each $d$ we consider $n=2d^2$ and $n=d^{2.5}$. For each $(d,n)$ pair, we compute $L$ for 20 $n$-sample posteriors with standard Gaussian design. The solid blue and red curves in Figure~\ref{fig:Lmean} represent the average of the 20 posteriors, for $n=2d^2$ and $n=d^{2.5}$, respectively. The shaded regions represent the spread of the middle 16 of 20 posteriors. From the blue curve ($n=2d^2$) we see that $L$ levels out to a constant for large $d$, suggesting $L$ is a function of $d/\sqrt n$. To prove this quantity indeed scales as $d/\sqrt n$ (and not e.g. some higher power of $d/\sqrt n$), we check that the slope of the solid red curve ($n=d^{2.5}$) is approximately $-0.25$, which would confirm $L\sim d^{-0.25} =  d/\sqrt{d^{2.5}}=d/\sqrt n$. And indeed, the best fit line has slope $-0.28$ for TV, and slope $-0.3$ for the mean. Therefore, the plots confirms that both leading order terms $L$ are lower bounded as $d/\sqrt n$.
\begin{figure}
 \centering
  \includegraphics[width=0.47\textwidth]{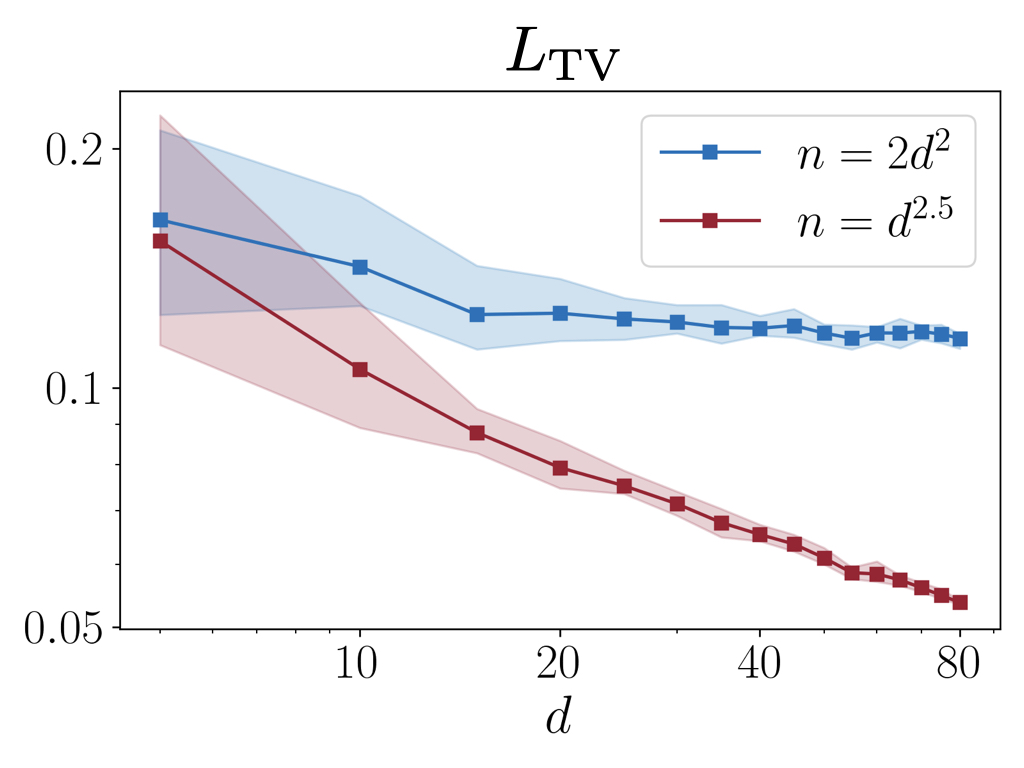}
 \includegraphics[width=0.49\textwidth]{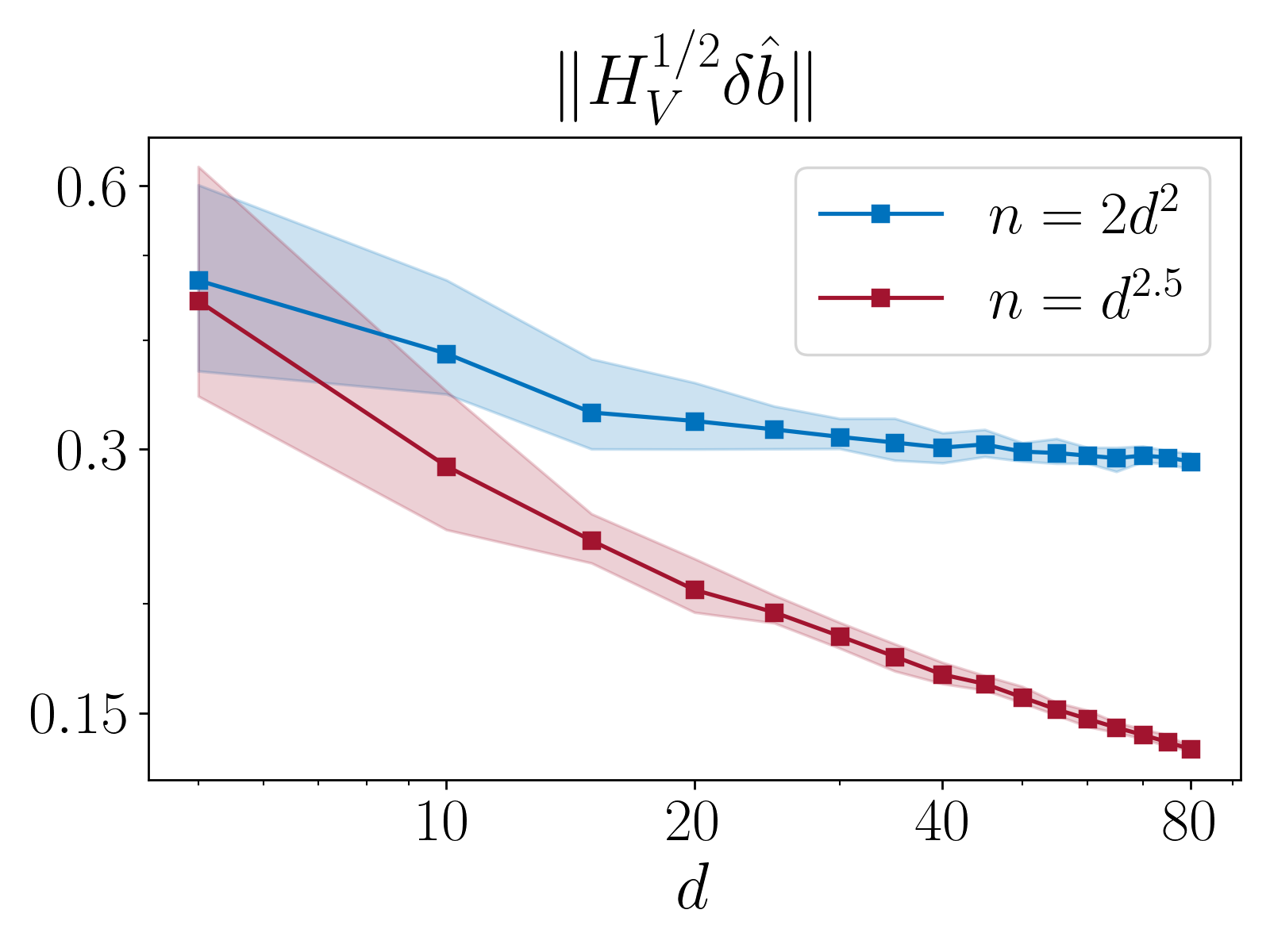}
 \caption{For a sequence of $d$ values, we compute the leading order terms $L:=\LTV$ (left) and $L:=\|\Lhatb\|_{H_V}$ (right) of the TV and mean errors, respectively. The measure $\pi\propto e^{-V}$ is the logistic regression posterior defined in~\eqref{pi-def}. The solid blue and red curves represent the average of the quantity $L$ over 20 posteriors, for $n=2d^2$ and $n=d^{2.5}$, respectively. The shaded regions represent the spread of $L$ for the middle 16 of 20 posteriors. When $d^2/n$ is constant (i.e. the $n=2d^2$ regime), we see that $L$ levels out to a constant for large $d$. When $n=d^{2.5}$, the value of $L$ scales approximately as $d^{-0.28}$ for TV and as $d^{-0.3}$ for the mean, based on the slope of the best fit line to the solid red curve.}
 \label{fig:Lmean}
 \end{figure}
 
We can combine our numerical confirmation that $\LTV\gtrsim d/\sqrt n$ and $\|\Lhatb\|_{H_V}\gtrsim d/\sqrt n$ with our rigorous high probability bounds on the remainder terms, i.e. $|\RTV|\les (d/\sqrt n)^2$ and $\|\bpi-\hat b-\Lhatb\|_{H_V}\les (d/\sqrt n)^3$. Together, these results give a heuristic argument that 
$$\TV(\pi,\lap)\gtrsim d/\sqrt n, \qquad \|\bpi-\hat b\|_{H_V}\gtrsim d/\sqrt n$$ with high probability.

This is a compelling demonstration of how the decomposition $\TV(\pi,\lap)=\LTV+\RTV$ can be used (and similarly for $\bpi = \hat b + \Lhatb + R_{\bpi}$). Without this decomposition, verifying that $\TV(\pi,\lap)\gtrsim d/\sqrt n$ numerically would require one to literally compute the TV distance between the two measures for an increasing sequence of dimensions $d$. This is precisely the kind of expensive numerical calculation we are trying to avoid. Instead, we can combine a rigorous bound on $\RTV$ with the numerical calculation of $\LTV$. This is \emph{much cheaper} than the numerical calculation of a full TV distance. Indeed, $\LTV$ is simply a Gaussian expectation, which can be feasibly calculated in moderate dimensions (e.g. up to $d=80$, as in Figure~\ref{fig:Lmean}).

\subsubsection{Computable,``fixed-sample'' upper bound via~\eqref{tildec3form}}
The high probability upper and lower bounds we have stated so far convey information on the level of orders of magnitude, but they do not allow us to precisely quantify the error for a given dataset of feature vectors and labels. We call such a bound, which is tailored to observed data, a ``fixed sample'' bound.  

In particular, \emph{computable} fixed sample bounds are desired. As discussed in the introduction and Section~\ref{subsub:compute}, upper bounds from prior works are either highly sub-optimal or else involve operator norms. Computing operator norms accurately requires solving a non-convex optimization problem, which requires exponential computational complexity in general.

Here, we have provided the upper bound~\eqref{tildec3form} on $\LTV$, which is computable in polynomial time. Moreover, our results imply that $\LTV$ is in fact the leading order term on the TV distance in the traditional sense. Therefore, the bound~\eqref{tildec3form} on $\LTV$ is a good approximation to an upper bound on the overall TV distance. To be more precise, we have shown that $\LTV\gtrsim d/\sqrt n$ and $\RTV\les (d/\sqrt n)^2$, which implies that $\RTV\leq \LTV$ if $d/\sqrt n$ is sufficiently small. Therefore, the upper bound~\eqref{tildec3form} is safely useable once the constants in these inequalities have been quantified. We leave this study to future work.

\section{Proof of Theorems~\ref{thm:Vgen} and~\ref{thm:odd}}\label{sec:overview}
In this section, we prove~\eqref{g-bd-2} and~\eqref{godd-bd}; that is, we bound the remainder $R(g):=\int gd\pi-\int g(1+\Lhat)d\lap$ in the case of general $g$ and odd $g$ (about $\mhat$). To do so, we first make a scale-removing and simplifying coordinate transformation in Section~\ref{sec:affine}. Then in Section~\ref{subsec:gendec}, we give a general decomposition for the remainder $R(g)$ into a sum of two error terms --- a local integral and a tail integral. This decomposition does not use anything specific about $g,\pi, \Lhat$, and $\lap$, except that $\int \Lhat d\lap=0$. In Section~\ref{subsec:apply}, we apply the decomposition to our specific setting. We also state a few key bounds on the resulting terms, which allow us to finish the proof of~\eqref{g-bd-2} and~\eqref{godd-bd}. In Section~\ref{sec:prop:r2k}, we outline the proofs of these key bounds.

In this section, we let $\gamma$ denote the standard normal distribution, and if $f\in L^p(\gamma)$, then we use the notation
$$\|f\|_p = \l(\int |f|^pd\gamma\r)^{1/p}.$$

See Appendix~\ref{app:sec:overview} for the omitted proofs from this section.

\subsection{Scale-Removing Coordinate Transformation}\label{sec:affine}
We start with a scale-removing coordinate transformation which transforms $\pi\to\rho$ and $\lap\to\gamma$. Here, $\gamma=\mathcal N(0,I_d)$ is standard normal, and $\rho$ is a small perturbation thereof. The benefit of this transformation is that whereas $\pi$ and $\lap$ are``moving targets'' --- both changing with $n$ --- the distribution $\gamma$ is fixed, and $\rho$ is converging to it. Showing $\pi$ is close to $\lap$ is equivalent to showing $\rho$ is close to $\gamma$, and the latter approximation analysis is conceptually clearer. Let $T(x) = H_V^{1/2}(x-\mhat )$, and $\rho=T_{\#}\pi\propto e^{-W}$, where \beq\label{W-def}W(x) = V\l(\mhat  + H_V^{-1/2}x\r)=nv\l(\mhat +H_v^{-1/2}x/\sqrt n\r).\eeq The function $W$ is defined on $\Omega:=T(\Theta)$ and minimized at zero, with Hessian $\nabla^2W(0)=I_d$. Therefore, its Laplace approximation is $\gamma$, the standard normal distribution. This is also precisely the pushforward under $T$ of the Laplace approximation $\lap$ to $\pi$, i.e. $\gamma=T_{\#}\lap$. 

Now, recalling the definition~\eqref{Lhatdef} of $\Lhat$, note that
\beqs\label{g-f}
\int gd\pi &-\int g(1+\Lhat)d\lap = \int fd\rho-\int f(1+p_3)d\gamma,\qquad f(x):=g(\mhat +H_V^{-1/2}x)\eeqs where \beq\label{p3def}p_3(x)=-\frac{1}{3!}\la\nabla^3W(0), x^{\otimes 3}\ra\eeq is the negative of the third order term in the Taylor expansion of $W$. Recall from~\eqref{gnorm} that $g$ has been standardized so that $\int gd\lap=0$ and $\int g^2d\lap=1$. Correspondingly, $\int fd\gamma=0$ and $\int f^2d\gamma=1$. 

Abusing notation, we redefine $\U(\s)=T(\U(s))$, which is given by
\beq\label{Unew}\U(\s)=\{x\in\R^d\; : \; \|x\|\leq\s\sqrt d\}.\eeq
The following lemma lists properties of $W$. The proof is immediate by the definition of $W$.
\begin{lemma}\label{ass:VtoW} Let $W$ be given by~\eqref{W-def}, and suppose $v$ satisfies Assumptions~\ref{assume:1}, \ref{assume:c0} with constants $\cgro,\s_0$. Then $W\in C^4$ and has unique global minimizer $x=0$, with $\nabla^2W(0)=I_d$. Furthermore, we have
\begin{align}
\|\nabla^3W(0)\|&=\frac{c_3}{\sqrt n},\label{c3W}\\
\sup_{\|x\|\leq \s\sqrt d}\|\nabla^4W(x)\|&=\frac{c_4(\s)}{n},\label{c4Wloc}\\
\sup_{\|x\|\leq \s\sqrt d}\|\nabla^5W(x)\|&=\frac{c_5(\s)}{n\sqrt n},\label{c5Wloc}\\
 W(x) - W(0)&\geq \cgro\sqrt d\|x\|\quad\forall x\in\Omega\setminus\U(\s_0).\label{Wglobd}
\end{align} The equation~\eqref{c5Wloc} holds whenever we make the stricter assumption that $v\in C^5$.
 \end{lemma}
The fact that $W$ is minimized at zero with $\nabla^2W(0)=I_d$, and the fact that higher order derivatives of $W$ are small at zero, implies $W(x)\approx W(0)+\|x\|^2/2$, and hence $\rho\propto e^{-W}$ is close to $\gamma\propto e^{-\|x\|^2/2}$.  We can therefore write $d\rho\propto e^rd\gamma$, for a ``small'' function $r$. This function will be specified in Section~\ref{subsec:apply} below.

  \subsection{Generic Decomposition}\label{subsec:gendec}We now temporarily step away from the specifics of Section~\ref{sec:affine} and consider a general setting in which $d\rho\propto e^rd\gamma$ for $r$ small. We are interested in understanding the error of approximating $\int fd\rho$ using $\int f(1+h)d\gamma$. We will bound the error $\int fd\rho - \int f(1+h)d\gamma$ in terms of various norms of $f$, $h$, and $r$.
  
\begin{defn}
For functions $f,h\in L^2(\gamma)$ such that $\int fd\gamma=\int hd\gamma=0$,  we define
\begin{equation}
\begin{split}
\bar f_{h} = f - \int f(1+h)d\gamma =  f - \int  fhd\gamma.
\end{split}
\end{equation}
\end{defn}
\begin{lemma}\label{lma:pretty}
Let $\rho$ be a probability measure on an open subset $\Omega\subset\R^d$, and $\gamma$ be a probability measure on $\R^d$. Suppose $d\rho\propto Qd\gamma$, where $Q(x)=e^{r(x)}$ when $x\in\Omega$, and $Q(x)=0$ when $x\in\Omega^c$. Let $f,h$ be such that $\int fd\gamma=\int hd\gamma=0$. Then
\beq\label{pretty}\int fd\rho - \int f(1+h)d\gamma = \frac{\int (Q-1-h)\bar f_{h}d\gamma}{\int Qd\gamma}=\frac{\mathrm{Loc}+\mathrm{Tail}}{\int Qd\gamma},\eeq 
where
\beq\label{loc-tail-def}
\mathrm{Loc} =\int_{\U} (e^r-1-h)\bar f_{h}d\gamma,\qquad\mathrm{Tail}=\int_{\Omega\setminus\U}e^r\bar f_hd\gamma -\int_{\U^c}(1+h)\bar f_hd\gamma,
\eeq
and $\U\subset\Omega\subset\R^d$ is an open set. Moreover if $\int_{\U}rd\gamma=0$, then $\int Qd\gamma\geq\gamma(\U)$.
\end{lemma}
Next, we will further bound Loc and Tail. To bound Loc, we decompose it as
\beqs\label{eq:prelim}
\mathrm{Loc}=& \int_{\U} (e^r-1-r)\bar f_hd\gamma + \int_{\U}(r-h) \bar f_hd\gamma\eeqs We bound these two integrals using Cauchy-Schwarz and H\"{o}lder, and the fact that $e^r-1-r$ is the 2nd order Taylor remainder of the exponential function. To bound $\mathrm{Tail}$, we use that $|\bar f_h(x)|\leq | f(x)|+\| f\|_2\|h\|_2=|f(x)|+\|h\|_2$ in the first integral. (Here we are assuming $\|f\|_2=1$.) For the second integral, we apply H\"{o}lder's inequality to the functions $1+h$, $\bar f_h$, and $\ind_{\U^c}$ with powers 4,2,4, respectively.

Overall, we obtain the following preliminary bounds on Loc and Tail.
\begin{lemma}[Preliminary bounds]\label{lma:prelim}Suppose $\int fd\gamma=0$ and $\int f^2d\gamma=1$. Then 
\begin{align}
\l|\mathrm{Loc}\r|&\les(1+\|h\|_2)\|e^r\ind_{\U}\|_{4}\|r\ind_{\U}\|_{8}^{2} + (1+\|h\|_2)\|(r-h)\ind_{\U}\|_2,\label{preliloc}\\
\l|\mathrm{Tail}\r|&\les \int_{\Omega\setminus\U}(| f|+\|h\|_2)e^rd\gamma + (1+\|h\|_4)^2\gamma(\U^c)^{\frac14}.\label{prelitail}
\end{align} 
\end{lemma}
\subsection{Application of generic decomposition}\label{subsec:apply}
We combine Lemma~\ref{lma:pretty} and~\ref{lma:prelim} with $\rho\propto e^{-W}$ on $\Omega$, and $\gamma$ given by the standard normal distribution. We will now specify the function $r$ such that $d\rho\propto e^rd\gamma$ on $\Omega$. To do so, we introduce a few important quantities.  
\begin{defn}[Remainder terms]\label{def-r}
Let
$$
p_3(x)=-\frac{1}{3!}\la\nabla^3W(0), x^{\otimes 3}\ra,\qquad p_4(x)=-\frac{1}{4!}\la\nabla^4W(0), x^{\otimes 4}\ra,$$
and define
$$
r_3(x)= -\big(W(x)-W(0) -\tfrac12\|x\|^2\big),\qquad r_4(x)= r_3(x) - p_3(x),\qquad r_5(x)=r_4(x)-p_4(x).
$$
%
\end{defn}
Note that $r_3 = p_3+r_4 = p_3+p_4+r_5$, and  $W(x)=W(0)+\|x\|^2/2 -r_3(x)$. In other words, $r_k$ is the $k$th order Taylor remainder of $-W$ about zero, for $k=3,4,5$. We also use the convention that for any continuous function $u$, we let
$$\hat u = u-\int_{\U}ud\gamma,$$ where $\U=\U(\s)$ is as in~\eqref{Unew}, for an $\s$ to be specified. Since $\U$ is symmetric about zero and $p_3$ is odd, we have $\hat p_3=p_3$, and hence $\hat r_3 = p_3 + \hat r_4 = p_3+\hat p_4 + \hat r_5.$ 

By Taylor's remainder theorem we have $r_k(x)=\frac{1}{k!}\la\nabla^kW(tx), x^{\otimes k}\ra$ for some $t=t(x)\in[0,1]$ and $k=3,4,5$. Therefore, the bounds~\eqref{c3W},~\eqref{c4Wloc},~\eqref{c5Wloc} show that the following scalings hold pointwise in $x$:
\beq\label{n-informal} r_3\sim p_3\sim n^{-1/2}, \qquad  r_4\sim  p_4\sim n^{-1}, \qquad  r_5\sim n^{-3/2}.\eeq

Now, the ``small'' function $r$ we choose to satisfy $d\rho\propto e^{r}d\gamma$ is $r:=\hat r_3$, which satisfies $\int_\U \hat r_3d\gamma=0$ by construction. Recall from~\eqref{g-f} and the following discussion that we are interested in studying $\int fd\rho - \int f(1+p_3)d\gamma$ for a ``standardized'' function $f$ (i.e. $\int fd\gamma=0$ and $\int f^2d\gamma=1$). Therefore, in the language of Section~\ref{subsec:gendec}, we take $h=p_3$. 

We now apply Lemma~\ref{lma:prelim}. Note that $r-h=\hat r_3-p_3=\hat r_4$, and that $\|\hat r_k\ind_\U\|_m\leq 2\|r_k\ind_\U\|_m$ for any $k=3,4,5$ and $m\geq1$. 
\begin{corollary} Suppose $\int fd\gamma=0$ and $\int f^2d\gamma=1$. Then
\begin{align}
\l|\mathrm{Loc}\r|&\les(1+\|p_3\|_2)\|e^{\hat r_3}\ind_{\U}\|_{4}\|r_3\ind_{\U}\|_{8}^{2} + (1+\|p_3\|_2)\|r_4\ind_{\U}\|_2,\label{appl:loc}\\
\l|\mathrm{Tail}\r|&\les \int_{\Omega\setminus\U}(| f|+\|p_3\|_2)e^{\hat r_3}d\gamma + (1+\|p_3\|_4)^2\gamma(\U^c)^{\frac14}.\label{appl:tail}
\end{align}
\end{corollary}
Now, in Lemmas~\ref{lma:r345} and~\ref{lma:exp} in the appendix, we prove the following key bounds:
\begin{equation}\label{eq:key}
\begin{gathered}
\|p_3\|_{k} \les_k\tilep\leq\epsilon_3,\qquad\|r_4\ind_{\U}\|_k \les_k\efour{\s}^{2},\\
\|e^{\hat r_3}\ind_{\U}\|_4 \leq \e\l(\l(\efour{\s}^2/12+ \epsilon_3^2/2\r)\s^4\r)\leq \Es{\s},
\end{gathered}
\end{equation}recalling the definition of $\Es{\s}$ from~\eqref{Esdef}. Recall from Definition~\ref{epsdef} that in terms of $n$, we have the scalings $\epsilon_3\sim n^{-1/2}$ and $\efour{\s}^2\sim n^{-1}$. Thus the order of magnitude of the first two bounds in~\eqref{eq:key} aligns with the order of magnitude of $p_3, r_4$ from the informal argument leading to~\eqref{n-informal}. See Section~\ref{sec:prop:r2k} for a discussion of the proof of~\eqref{eq:key}. These bounds then also imply that
\beq\label{eq:key:r3}
\|r_3\ind_{\U}\|_k \les_k \epsilon_3 + \efour{\s}^2.
\eeq Using~\eqref{eq:key} and~\eqref{eq:key:r3} in~\eqref{appl:loc}, we obtain the following overall bound on Loc:
\beq\label{key:loc}
|\mathrm{Loc}| \les \Es{\s}(\epsilon_3 + \efour{\s}^2)^2 + \efour{\s}^2 \les \Es{\s}(\epsilon_3^2 + \efour{\s}^2).
\eeq\\

Furthermore, we have the following bound on the tail integral, using that $f=(F-\int Fd\gamma)/\sqrt{\Var_{\gamma}(F)}$ for $F(x)=G(\mhat+H_V^{-1/2}x)$, and that $F$ satisfies the following bound as a consequence of $G$ satisfying~\eqref{gcond}:
\beq\label{fbd}
|F(y)-\medint\int Fd\gamma|\leq \af{F}e^{\cgro\sqrt d\|y\|/2}\quad\,\forall y\in\R^d,\quad \af{F}=\af{G}.
\eeq Here, $\af{G}$ is as in~\eqref{gcond}.
\begin{lemma}\label{lma:tail}
Let $f$ be normalized from $F$ as above, where $F$ satisfies~\eqref{fbd}. Also, suppose $\s\geq\sst=\max(\s_0, \frac{8}{\cgro}\log\frac{2e}{\cgro})$, and that $\tilep,\efour{\s}\leq C$. Then
\beq\label{eq:lma:tail}
|\mathrm{Tail}|\les \l| \int_{\Omega\setminus\U}(| f|+\|p_3\|_2)e^{\hat r_3}d\gamma + (1+\|p_3\|_4)^2\gamma(\U^c)^{\frac14}\r|\les \l(\af{f}\vee1\r)\taus{\s},
\eeq where $\af{f}=\af{F}/\sqrt{\Var_\gamma(F)}$ and $\taus{\s}=de^{-\cgro\s d/8}$.
\end{lemma}
In the proof, we use a standard Gaussian tail inequality to bound $\gamma(\U^c)$. For the first integral, we use that $e^{\hat r_3}\gamma(x)\propto e^{W(0)-W(x)}$ when $x\in\Omega\setminus\U$. We then use the lower bound~\eqref{Wglobd} on the growth of $W$ in $\Omega\setminus\U$, and the upper bound~\eqref{fbd} on the growth of $F$. Together, these facts allow us to bound the first integral by another integral of the form $(2\pi)^{-d/2}\int_{\|x\|\geq \s\sqrt d}e^{-a\sqrt d\|x\|}dx$. This can be bounded using standard gamma integral tail bounds. \\

Finally, note that $\gamma(\U)=\PP(\|Z\|\geq\s\sqrt{d})\geq1/2$. This follows by a standard Gaussian tail bound, since $\s\geq(8/\cgro)\log(2e/\cgro)\geq8$. We use this lower bound on $\gamma(\U)$ together with the upper bounds~\eqref{eq:lma:tail} and~\eqref{key:loc} on Tail and Loc in the decomposition~\eqref{pretty}. This concludes the proof of the bound~\eqref{g-bd-2} in Theorem~\ref{thm:Vgen}.

\paragraph*{Proof of Theorem~\ref{thm:odd}: odd $f$}
We now briefly discuss how to prove Theorem~\ref{thm:odd}; see Appendix~\ref{app:sec:overview} for the full proof. First note that if $g$ is odd about $\mhat$, then $f$ is odd about zero. Now, the main change is that we decompose Loc not as~\eqref{eq:prelim}, but rather as
\beqs\label{eq:prelim:odd}
\mathrm{Loc} = &\int_\U (e^{r}-1-r-r^2/2)fd\gamma + \int_\U (r+r^2/2-h)fd\gamma\\
&+\int fhd\gamma\l( \int_\U (e^{r}-1-r)d\gamma + \int_\U (r-h)d\gamma\r).
\eeqs
We again take $r=\hat r_3$ and $h=p_3$. Let us give some intuition for the size of the terms in this decomposition. For simplicity, we only discuss the $n$ dependence here. We have $$h=p_3\sim n^{-1/2},\qquad r=\hat r_3\sim p_3\sim n^{-1/2},\qquad r-h=\hat r_3 - p_3 = \hat r_4\sim n^{-1}.$$ Informally, this leads to the following orders of magnitude:
\beqs
\int_\U (e^{r}-1-r-r^2/2)fd\gamma\sim r^3&\sim n^{-3/2},\\
\int fhd\gamma\int_\U (e^{r}-1-r)d\gamma \sim hr^2&\sim n^{-3/2},\\
\int fhd\gamma\int_\U (r-h)d\gamma \sim h(r-h) &\sim n^{-3/2}.\eeqs
Thus these three terms match the order of accuracy stated in~\eqref{godd-bd}. The remaining question is about the size of $\int_\U (r+r^2/2-h)fd\gamma$. It turns out that due to $f$ being odd, the leading order contribution to this integral vanishes, and the remainder is of order $n^{-3/2}$ as well.

\subsection{Proof of Key Bounds~\eqref{eq:key}}\label{sec:prop:r2k}
Recall from~\eqref{p3def} that $p_3(x) = \la -\tfrac{1}{3!}\nabla^3W(0), x^{\otimes3}\ra$. We start by discussing the first bound in~\eqref{eq:key}, on the quantity
 \beq\label{p32k}\|p_3\|_{k}=\E\l[\l|\lla \tfrac{1}{3!}\nabla^3W(0), Z^{\otimes 3}\rra\r|^{k}\r]^{1/k},\qquad Z\sim\mathcal N(0, I_d).\eeq %
 \begin{lemma}\label{lma:prop:r2k}
It holds
 \begin{align}
\|p_3\|_2^2 &= \frac16\|\nabla^3W(0)\|_F^2 + \frac14\|\la \nabla^3W(0), I_d\ra\|= \tilep^2,\label{eqp31}\\
 \|p_3\|_{k} &\les_k\|p_3\|_2=\tilep,\quad k\geq2.\label{eqp32}\end{align}
 \end{lemma}
 \begin{proof}
For brevity, let $T:=\tfrac{1}{3!}\nabla^3W(0)$. Now, note that \beq\label{Ssymm}p_3(x)=\la T, x^{\otimes 3}\ra=\la T, \H_3(x)\ra +3\la T, I_d\ra^\T x,\eeq where $\H_3(x)$ is the tensor of third order Hermite polynomials of $x$. See Appendix~\ref{app:hermite-primer} for a brief primer on Hermite polynomials and in particular the definition of $\H_3$, from which~\eqref{Ssymm} follows. Next, a simple calculation in Lemma~\ref{lma:TH32} in the appendix shows that
\beq\label{laTH3ra}
\|\la T, \H_3\ra\|_2^2=\E\l[\la T, \H_3(Z)\ra^{2}\r]= 6\|T\|_F^2.\eeq Now by orthogonality of Hermite polynomials, we have that $\la T, \H_3(x)\ra$ is orthogonal to $\la T, I_d\ra^\T x$ with respect to $\gamma$, i.e. $\E[\la T, \H_3(Z)\ra\la T, I_d\ra^\T Z]=0$. Therefore, the representation~\eqref{Ssymm} of $p_3$ leads to
\beq\label{Ep32}\|p_3\|_2^2= \|\la T, \H_3\ra\|_2^2 + 9\E\l[\l(\la T, I_d\ra^\T Z\r)^2\r] = 6\|T\|_F^2 +9\|\la T, I_d\ra\|^2.\eeq Recalling the definition of $T$, we see that~\eqref{Ep32} is precisely the desired equality~\eqref{eqp31}. Next, we again use~\eqref{Ssymm} to bound $\|p_3\|_{k}$ by
 \beqs\label{Ebigla}
\|p_3\|_{k} &\leq \|\la T, \H_3\ra\|_{k} + 3\E\l[\l|\la T, I_d\ra^\T Z\r|^{k}\r]^{1/k}\les_k \|\la T, \H_3\ra\|_{k}+\|\la T, I_d\ra\|.
\eeqs
Finally, the crucial observation is that by hypercontractivity of the Ornstein-Uhlenbeck semigroup, for which the Hermite polynomials are eigenfunctions, we can relate the $k$ norm of $\la T, H_3\ra$ to the $2$ norm:
\beq\label{SH3}\|\la T, \H_3\ra\|_{k} \les_k \|\la T, \H_3\ra\|_2=\sqrt{6}\|T\|_F.\eeq See Lemma~\ref{hypercon} in the appendix for the proof of the inequality; the equality is by~\eqref{laTH3ra}. Substituting~\eqref{SH3} into~\eqref{Ebigla} and using~\eqref{Ep32} concludes the proof of~\eqref{eqp32}.
 \end{proof}
 \begin{remark}\label{rk:hermite}
 The use of Hermite polynomials is inspired by our earlier work~\cite{KRVI} on the approximation accuracy of Gaussian VI. There, Hermite polynomials played a much more conceptually important role. In the present work, they are purely a technical tool which simplifies calculations.
\end{remark}



The second bound in~\eqref{eq:key} states $\|\rfour\ind_{\U}\|_{k}\les_k\efour{\s}^2$. Proving this bound is straightforward. We use Taylor's remainder theorem to write $\rfour(x)=\la\nabla^4W(tx), x^{\otimes4}\ra/24$ for some $t=t(x)\in[0,1]$. We then use the bound~\eqref{c4Wloc} to conclude $|\rfour(x)|\leq c_4(\s)\|x\|^4/(24n)$ for $x\in\U$. From here we get $\|\rfour\ind_{\U}\|_{k}\lesssim_k c_4(\s)d^2/n=\efour{\s}^2$ by straightforward Gaussian calculations. See Lemma~\ref{lma:r345} in the appendix for the full proof. \\

Finally, the third bound in~\eqref{eq:key} states $\|e^{\hat r_3}\ind_{\U}\|_4 \leq \e((\efour{\s}^2/12+ \epsilon_3^2/2)\s^4)$. We first write $\hat r_3 = p_3+\hat r_4$ and bound $e^{\hat r_4}$ with a sup norm. This is straightforward, so here we will only discuss how to bound $\|e^{p_3}\ind_{\U}\|_4$.  Now, in absolute value, $p_3(x)$ gets as high as $\|\nabla^3W(0)\|d\sqrt d/\sqrt n$ over $x\in\U$. The large values of $|p_3|$ seem to present an obstacle to bounding an integral of $e^{p_3}$ over $\U$. However, note that  $p_3$ is $L$-Lipschitz on $\U$, with a Lipschitz constant 
$$L=\sup_{x\in\U}\|\nabla p_3(x)\|\leq \|\nabla^3W(0)\|(\s\sqrt d)^2 = (c_3d/\sqrt n)\s^2=\epsilon_3\s^2,$$ in this region. Therefore, we can apply Herbst's argument on the exponential integrability of Lipschitz functions with respect to the Gaussian distribution (or any measure satisfying a log Sobolev inequality). This leads to the bound $\l\|e^{p_3}\mathbbm{1}_{\U}\r\|_4\les \e(CL^2) = \e(C\epsilon_3^2\s^4)$. See Lemma~\ref{lma:exp} in the appendix for the full proof. Note that unlike the bound~\eqref{eqp32} on $\|p_3\|_k$, which involves only $\tilep$, the bound on $\|e^{p_3}\ind_\U\|_4$ involves the larger quantity $\epsilon_3$. This is because our technique to bound $\|e^{p_3}\ind_\U\|_4$ requires taking a \emph{supremum} over all $x\in\U$, resulting in the appearance of the operator norm $\|\nabla^3W(0)\|$. 
\\

%
%

\section*{Acknowledgments}
Thank you to Philippe Rigollet for useful discussions, and to Rishabh Dudeja for suggesting to use hypercontractivity, which significantly simplified the proof of a key lemma.


%
\appendix

\section{Proofs from Section~\ref{sec:V}}\label{app:main}
First, we give the formula for $\tilep$ in terms of $V$. Recall that we have defined $\tilep$ via Definition~\ref{def:tilep}. It also holds, 
\begin{align}\label{def:app:tilep}
 \tilep^2 = \frac1{6}\|\nabla^3V(\mhat)\|^2_{F;H_V^{-1}} + \frac1{4}\|H_V^{-1/2}\lla \nabla^3V(\mhat), H_V^{-1}\rra\|^2
 \end{align} Here, 
$$
\|T\|^2_{F;A} = \sum_{i,j,k=1}^d\la T, (A^{1/2})_i\otimes (A^{1/2})_j\otimes (A^{1/2})_k\ra^2
$$ for a symmetric tensor $T$ and a symmetric matrix $A\succ0$, with $(A^{1/2})_i$ the $i$th column of $A^{1/2}$. 

\begin{proof}[Proof of claim in Remark~\ref{vconvA2}]Convexity of $v$ implies that $$c(r):=\inf_{\|x-\mhat\|_{H_v}=r\sqrt{d/n}}\;\frac{v(x)-v(\mhat)}{\sqrt{d/n}\|x-\mhat\|_{H_v}}$$ is increasing, and hence
$$\inf_{\|x-\mhat\|_{H_v}\geq \s_0\sqrt{d/n}}\;\frac{v(x)-v(\mhat)}{\sqrt{d/n}\|x-\mhat\|_{H_v}}=c(\s_0).$$ We now bound $c(\s_0)$ from below. We use that if $\|x-\mhat\|_{H_v}=\s_0\sqrt{d/n}$ then
\beqs
v(x)-v(\mhat) = &\frac12\|x-\mhat\|_{H_v}^2 +\frac{1}{3!}\la\nabla^3v(\mhat), (x-\mhat)^{\otimes 3}\ra + \frac{1}{4!}\la\nabla^4v(\xi), (x-\mhat)^{\otimes 4}\ra\\
 \geq &\frac12\|x-\mhat\|_{H_v}^2 -\frac{c_3}{6}\|x-\mhat\|_{H_v}^3 - \frac{c_4(\s_0)}{24}\|x-\mhat\|_{H_v}^4\\
 = &\|x-\mhat\|_{H_v}\l(\frac {\s_0}2\sqrt{d/n}\r)\l(1 - \frac{c_3}{3}\s_0\sqrt{d/n} - \frac{c_4(\s_0)}{12}(\s_0\sqrt{d/n})^2\r).
\eeqs
If $\frac{c_3}{3}\s_0\sqrt{d/n} + \frac{c_4(\s_0)}{12}(\s_0\sqrt{d/n})^2\leq1/2$ then~\eqref{assume:c0:eq} is satisfied with this $\s_0$ and $\cgro=\s_0/4$. Substituting $\s_0=4$ into this inequality and rearranging terms gives the stated condition $c_3\sqrt{d/n}+c_4(4)d/n\leq 3/8$.
\end{proof}
\begin{proof}[Proof of bound in Example~\ref{corr:mgf}]
Let $G(x)=\e(u^\T H_V^{1/2}(x-\mhat))$. Then $\int Gd\lap = \E[e^{u^\T Z}] = e^{\|u\|^2/2}$, and $\Var_{\lap}(g)=\Var(e^{u^\T Z}) = e^{2\|u\|^2} - e^{\|u\|^2}$ by standard calculations. Also, we have
\beqs
\l|G(x)-\int Gd\lap\r|& =\l|\e(u^\T H_V^{1/2}(x-\mhat))-\e(\|u\|^2/2)\r| \\
&\leq \max\l(e^{\|u\|^2/2}, \e(u^\T H_V^{1/2}(x-\mhat))\r)\leq e^{\|u\|^2/2}e^{\|u\|\|x-\mhat\|_{H_V}}
\eeqs
and hence~\eqref{gcond} is satisfied with $\af{G}=e^{\|u\|^2/2}$, as long as $\|u\|\leq \cgro\sqrt d/2$. Now, let $g=(G-\int Gd\lap)/\sqrt{\Var_{\lap}(G)}$. It follows by~\eqref{g-bd-2} that
\beqs\label{Gdpilap}
\l|M(u)-\int Gd\corlap\r|&=\l|\int Gd\pi-\int Gd\corlap\r|=\sqrt{\Var_{\lap}(G)}\l|\int gd\pi - \int gd\corlap\r|\\
&\leq \sqrt{\Var_{\lap}(G)}\Es{\s}\l(\del{\s}\r) +  \af{G}\taus{\s}\\
&= e^{\|u\|^2/2}\sqrt{e^{\|u\|^2}-1}\Es{\s}\l(\del{\s}\r) +  e^{\|u\|^2/2}\taus{\s}.
\eeqs
We show below that
\beq\label{intGlap}\int Gd\corlap = e^{\|u\|^2/2}\l(1-\lla\nabla^3W(0), \; \tfrac16u^{\otimes3}+\tfrac12I_d\otimes u\rra\r).\eeq Substituting this formula into~\eqref{Gdpilap} and dividing by $e^{\|u\|^2/2}$ gives the stated bound. Now, to show~\eqref{intGlap}, recall that $\int Gd\lap = e^{\|u\|^2/2}$. Hence we need only show that $\int G\Lhat d\lap = -e^{\|u\|^2/2}\la\nabla^3W(0), \; \tfrac16u^{\otimes3}+\tfrac12I_d\otimes u\ra$. Indeed, using a change of variables we have
\beqs
\int G\Lhat d\lap =-\frac{1}{3!} \E\l[e^{u^\T Z}\la\nabla^3W(0), Z^{\otimes3}\ra\r] = -\frac{1}{3!}e^{\|u\|^2/2} \E\l[\la\nabla^3W(0), (Z+u)^{\otimes3}\ra\r],
\eeqs and a straightforward Gaussian calculation concludes the proof.
\end{proof}

\begin{proof}[Proof of Corollary~\ref{corr:corTV}]We have $\TV(\pi,\corlap)=\sup_{A}\Delta_{\ind_A}(\corlap)$, where the supremum is over all Borel sets $A\in\R^d$. Note that $G=\ind_A$ satisfies $|G(x)-\int Gd\lap| \leq 1$ for all $x$, so~\eqref{gcond} is satisfied with $\af{G}=1$. Furthermore, $\sqrt{\Var_{\lap}(G)}\leq1$ as well. Thus we can apply the bound~\eqref{g-bd-2} on $\Delta_g(\corlap)$, where $g$ is the standardization of $G=\ind_A$. Multiplying both sides by $\sqrt{\Var_{\lap}(G)}\leq1$ does not change the righthand side. Taking the supremum over all Borel sets $A$ then gives $\TV(\pi,\corlap)\les \Es{\s}(\del{\s}) + \taus{\s}$, proving the righthand inequality in~\eqref{TVcorlap}. 

Next, it is straightforward to show that $$\sup_AL(\ind_A)=\sup_A\int\ind_A\Lhat d\lap = \int\max(\Lhat,0) d\lap = \frac12\int|\Lhat|d\lap=\LTV,$$ using that $\int\Lhat d\lap=0$. Therefore,
\beqs\label{ARTV}
|\RTV|&=|\TV(\pi,\lap)-\LTV|= \l|\sup_A\Delta_{\ind_A}(\lap)- \sup_AL(\ind_A)\r| = \l|\sup_A\{L(\ind_A)+R(\ind_A)\}- \sup_AL(\ind_A)\r| \\
&\leq\sup_A|R(\ind_A)|=\sup_{A}\Delta_{\ind_A}(\corlap)=\TV(\pi,\corlap),\eeqs proving the first inequality in~\eqref{TVcorlap}.


The bound on $\LTV$ is immediate from Lemma~\ref{lma:Lbd}; namely, $\LTV=\|\Lhat\|_{L^1(\hat\gamma)}/2\leq\|\Lhat\|_{L^2(\hat\gamma)}/2=\tilep/2$. Finally, if $A$ is an even set about $\mhat$, then $L(\ind_A)=0$, and hence $|\pi(A)-\lap(A)|=|\Delta_{\ind_A}(\lap)|=|R(\ind_A)|\leq \TV(\pi,\corlap)$ by~\eqref{ARTV}.
\end{proof}
\begin{proof}[Proof of Corollary~\ref{thm:Vmean}]
%
%

We begin by showing that $\Lhatm$ defined in~\eqref{Lm-def} is precisely$\int x\Lhat(x)\lap(x)dx$. Indeed, let $Z\sim\mathcal N(0, I_d)$. Then
\beqs
\int  x \Lhat(x) d\lap(x)&= -\frac16\E[(\mhat+H_V^{-1/2}Z)\la\nabla^3V(\mhat),(H_V^{-1/2}Z)^{\otimes 3}\ra]\\
&=-\frac16H_V^{-1/2}\E[Z\la\nabla^3V(\mhat),(H_V^{-1/2}Z)^{\otimes 3}\ra]\\
&= -\frac12H_V^{-1}\E[\la\nabla^3V(\mhat), (H_V^{-1/2}Z)^{\otimes 2}\ra]\\
&=-\frac12H_V^{-1}\la\nabla^3V(\mhat), H_V^{-1}\ra = \Lhatm,\eeqs using Gaussian integration by parts (i.e. $\E[Zf(Z)]=\E[\nabla f(Z)]$) to get the third line. We conclude that $\mhat+\Lhatm$ is the first moment of $\corlap$, and of course $\mhat$ is the first moment of $\lap$. Therefore, letting $g_u(x)=u^\T H_V^{1/2}(x-\mhat)$, we have
\beqsn
\|H_V^{1/2}(\mpi - \mhat - \Lhatm)\|& = \sup_{\|u\|=1}\int G_ud\pi - \int G_ud\corlap.
\eeqsn
Now, note that $G=G_u$ is already in standardized form, since $\int Gd\lap=0$ and $\int G^2d\lap =\Var_{\lap}(G_u)=1$ for all $\|u\|=1$. Furthermore,
$|G(x)|\leq \|x-\mhat\|_{H_v}.$ We can therefore take
$$\af{G}=\sup_{x\in\R^d}\|x-\mhat\|_{H_v}e^{-\cgro\sqrt d\|x-\mhat\|_{H_v}/2}=\sup_{t\geq 0}te^{-\cgro\sqrt dt/2} = e^{-1}\frac{2}{\cgro\sqrt d} \les 1$$ to satisfy~\eqref{gcond}. Here we used the assumption $\cgro\gtrsim1/\sqrt d$.
An application of Theorems~\ref{thm:Vgen} and~\ref{thm:odd} concludes the proof of~\eqref{cor-mean} and~\eqref{Rc5}. Finally, by the same reasoning as above,
$$
\|H_V^{1/2}\Lhatm\|= \sup_{\|u\|=1}\int g_u\Lhat d\lap \leq\|\Lhatm\|_{L^2(\lap)}=\tilep$$ by Cauchy-Schwarz.
\end{proof}
 \begin{proof}[Proof of Corollary~\ref{thm:Vcov}]
Let $G_u(x)=(u^\T H_V^{1/2}(x-\mhat))^2$, $\|u\|=1$. Note that $G_u$ is even about $\mhat$. Then
\beqs
\|H_V^{1/2}&(\Sigpi -H_V^{-1})H_V^{1/2}\|= \sup_{\|u\|=1}\l|\Var_{X\sim\pi}(u^\T H_V^{1/2}(X-\mhat))-1\r|\\
&\leq\sup_{\|u\|=1}\l|\E_{X\sim\pi}\l[\l(u^\T H_V^{1/2}(X-\mhat)\r)^2\r]-1\r| +  \|H_V^{1/2}(\mpi - \mhat )\|^2\\
&=\sup_{\|u\|=1}\l|\int  G_ud\pi - \int  G_ud\lap\r| +  \|H_V^{1/2}(\mpi - \mhat )\|^2\\
&=\sup_{\|u\|=1}|R(G_u)|+  \|H_V^{1/2}(\mpi - \mhat )\|^2.
\eeqs To get the last line we used that $\int  G_ud\pi - \int  G_ud\lap=\Delta_{G_u}(\lap)=L(G_u)+R(G_u)=R(G_u)$, since $G_u$ is even.

For the bound on $\|H_V^{1/2}(\mpi - \mhat )\|^2$, we use Corollary~\ref{thm:Vmean}. For the bound on $\int G_ud\pi - \int G_ud\lap$, we use that $\int G_ud\lap=1$, and that $|G_u(x)-\int G_ud\lap|\leq 1+\|x-\mhat\|_{H_V}^2$, so that we can take
$$\af{G_u}=\sup_{t\geq0}(1+t^2)e^{-\cgro\sqrt dt/2} \leq 1 + \frac{4}{(\cgro\sqrt d/2)^2}e^{-2}\les1$$ to satisfy~\eqref{gcond}. Finally, note that $\Var_{\lap}(G_u)=\Var((u^\T Z)^2)= 2$ for all $\|u\|=1$. The bound on $\int G_ud\pi - \int G_ud\lap$ now follows by~\eqref{g-bd-2} from Theorem~\ref{thm:Vgen} (upon multiplying both sides by $\sqrt{\Var(G_u)}=\sqrt 2$). Combining this bound with the one on $\|H_V^{1/2}(\mpi - \mhat )\|^2$ gives
\beqs
\|H_V^{1/2}(\Sigpi -H_V^{-1})H_V^{1/2}\| \les &\Es{\s}\l(\del{\s}\r) + \taus{\s} + \l(\tilep+\Es{\s}\l(\del{\s}\r) + \taus{\s}\r)^2\\
\les &\Es{\s}^2\l(\del{\s}\r) +  \taus{\s}.
\eeqs 
\end{proof}
 \begin{proof}[Proof of Lemma~\ref{lma:log-reg-TV}]We first check that Assumptions~\ref{assume:1} and~\ref{assume:c0} are satisfied and bound $\epsilon_3$ and $\efour{\s}$. Note that $\nabla^2v(x)=\E[\psi''(Z^\T x)ZZ^\T]$, so $u^\T\nabla^2v(x)u =\E[\psi''(Z^\T x)(u^\T Z)^2]$. Since $\psi''(Z^\T x)(u^\T Z)^2$ is nonnegative, we infer that if $u^\T\nabla^2v(x)u=0$ then $\psi''(Z^\T x)(u^\T Z)^2=0$ almost surely. But $\psi''(Z^\T x)>0$, so then $u^\T Z=0$ almost surely, which happens only if $u=0$. Hence $\nabla^2v(x)\succ0$, so $v$ is strictly convex. Moreover, we have $\nabla v(e_1)=0$. Thus $e_1$ is a strict local minimizer and hence a global minimizer, by convexity.

Next, we compute that for $k\geq2$, 
\beq\label{nablak-log-inf}\nabla^kv(x)=\E[\psi^{(k)}(Z^\T x)Z^{\otimes k}],\eeq so in particular,
\beq\label{HV-log-inf}H_{v}=\nabla^2v(e_1)=\E[\psi''(Z_1)ZZ^\T ]= \mathrm{diag}(a_{2,2}, a_{2,0}\dots, a_{2,0}).\eeq  From~\eqref{nablak-log-inf} and~\eqref{HV-log-inf} we get
\beqs
\|\nabla^kv(x)\|_{H_v} &= \sup_{u\neq0}\frac{\E[\psi^{(k)}(Z^\T x)(Z^\T u)^{k}]}{(u^\T H_vu)^{k/2}} \leq \|\psi^{(k)}\|_\infty \E[|Z_1|^k]\frac{\|u\|^k}{\|u\|^k\lambda_{\min}(H_v)^{k/2}} \\
&=  \|\psi^{(k)}\|_\infty \E[|Z_1|^k]\min(a_{2,2},a_{2,0})^{-k/2} = C_k,
\eeqs since all of the quantities in the last line depend only on $\psi$ and $k$, which depends neither on $d$ nor on $n$. Thus we have shown that $c_3,c_4(\s)\leq C$ for all $\s\geq 0$. Hence $\epsilon_3,\efour{\s}\leq C\epsilon$ for all $\s$, where $\epsilon=d/\sqrt n$. Using this and convexity of $v$, Remark~\ref{vconvA2} implies that Assumption~\ref{assume:c0} is satisfied with $\s_0=4$ and $\cgro=1$ whenever $\sqrt{d/n}$ is small enough.

We can now apply Corollaries~\ref{corr:corTV} and~\ref{thm:Vmean} to get upper bounds on the remainder terms $|\TV-\LTV|$ and $\|H_V^{1/2}(\mpi-\mhat-\Lhatm)\|$. We choose $\s=\epsilon^{-1/2}$, which satisfies the required condition $\s\geq\sst=\max(4,8\log(2e))$ for all $\epsilon$ small enough (we have used $\s_0=4$ and $\cgro=1$). Now, since $\epsilon_3\leq C\epsilon$ and $\sup_{\s\geq0}\efour{\s}\leq C\epsilon$ by the above calculations, we conclude that $\Es{\s}=\e((\del{\s})\s^4)\leq C$. Moreover, $\taus{\s}=d\e(-\cgro\s d/8) =d\e(-\epsilon^{-1/2}d/8) = d\e(-n^{1/4}\sqrt d/8)$. Substituting these values of $\taus{\s}$ and $\Es{\s}$ into~\eqref{TVcorlap} and~\eqref{cor-mean} concludes the proof.
\end{proof}
\begin{proof}[Proof of Lemma~\ref{lma:leading-lb-log}]
Let us compute $\LTV$ and $H_V^{1/2}\Lhatb$. Taking $x=e_1$ in~\eqref{nablak-log-inf} gives $\nabla^3V(e_1)=n\E[\psi'''(Z_1)Z^{\otimes 3}]$, and recall from~\eqref{HV-log-inf} that $H_{V}= n\mathrm{diag}(a_{2,2}, a_{2,0}\dots, a_{2,0})$. Now, for a fixed vector $x\in\R^d$, we compute
\beqs
\la\nabla^3V(e_1), x^{\otimes3}\ra &= n\E[\psi'''(Z_1)(x^\T Z)^{3}]\\
&= n\E[\psi'''( Z_1)(x_1Z_1)^3] + 3nx_1\E[\psi'''( Z_1)Z_1]\E[(x_{2:d}^\T Z_{2:d})^2]\\
&= n\l(a_{3,3}x_1^3 + 3a_{3,1}x_1\|x_{2:d}\|^2\r).\eeqs Hence
\beq\label{b-prep}|\la\nabla^3V(e_1), x^{\otimes3}\ra|\geq n\l(3|a_{3,1}||x_1|\|x_{2:d}\|^2 - |a_{3,3}||x_1|^3\r).\eeq Now substitute $$x=H_{V}^{-1/2}Z = n^{-1/2}\l(a_{2,2}^{-1/2}Z_1, a_{2,0}^{-1/2}Z_2,\dots,  a_{2,0}^{-1/2}Z_d\r)$$ into~\eqref{b-prep} and take expectations on both sides:
\beqs
\LTV&=\frac{1}{12}\E|\la\nabla^3V(e_1), (H_{V}^{-1/2}Z)^{\otimes3}\ra|\\
&\geq \frac14\frac{(d-1)}{\sqrt n}\frac{|a_{3,1}|}{a_{2,2}^{1/2}a_{2,0}}\E[|Z_1||Z_2|^2] -\frac{1}{12}\frac{|a_{3,3}|}{a_{2,2}^{3/2}}\frac{\E[|Z_1|^3]}{\sqrt n}\\
&\geq \frac{1}{a_{2,2}^{1/2}\sqrt n}\l(\frac{|a_{3,1}|}{8a_{2,0}}(d-1) -\frac{|a_{3,3}|}{4a_{2,2}}\r), 
\eeqs as desired. Here we used that $\E|Z_1|\geq1/2$ and $\E[|Z_1|^3]\leq3$.

Next, we substitute the expressions for $\nabla^3V(e_1)$ and $H_V$ into~\eqref{Lm-def} to get
\beqs
H_V^{1/2}\Lhatm = &-\frac{1}{2}H_{V}^{-1/2}\la\nabla^3V(e_1), H_{V}^{-1}\ra=-\frac{n}{2}\E\l[\psi'''( Z_1)Z^\T H_{V}^{-1}ZH_{V}^{-1/2}Z\r]\\
=&-\frac{1}{2}\E\bigg[\psi'''( Z_1)\bigg(a_{2,2}^{-1}Z_1^2+\sum_{j=2}^da_{2,0}^{-1}Z_j^2\bigg)H_{V}^{-1/2}Z\bigg]\\
=&-\frac{1 }{2a_{2,2}}\E\l[\psi'''( Z_1)Z_1^2H_{V}^{-1/2}Z\r] \\
&\qquad- \frac{1}{2 a_{2,0}}\sum_{j=2}^d\E\l[\psi'''(Z_1)Z_j^2H_{V}^{-1/2}Z\r]\\
=& -\frac{a_{3,3}}{2 a_{2,2}^{3/2}\sqrt n}e_1 - \frac{(d-1)a_{3,1}}{2  a_{2,0}a_{2,2}^{1/2}\sqrt n}e_1\\
=&-\frac{1}{2\sqrt n a_{2,2}^{1/2}}\l(\frac{a_{3,3}}{a_{2,2}}+(d-1)\frac{a_{3,1}}{a_{2,0}}\r)e_1.
\eeqs Taking the norm of this vector gives the desired formula for $\|H_V^{1/2}\Lhatm\|$. 
 \end{proof}
 \section{Proofs from Section~\ref{sec:pmf}}\label{app:sec:pmf}
 \renewcommand{\ground}{\theta^*}
 Define
 \beqs
 \Theta^d &= \l\{(\theta_1,\dots,\theta_d)\; : \; \sum_{i=1}^d\theta_i<1,\,0<\theta_i<1\,\,\forall i\r\}\subset\R^d,\\
\Theta_{d}&=\l\{(\theta_0,\theta_1,\dots,\theta_d)\; : \; \sum_{i=0}^d\theta_i=1,\,0<\theta_i<1\,\,\forall i\r\}\subset\R^{d+1},\\
\Omega_d &=\l\{(u_0',u_1',\dots,u_d')\; : \; \sum_{i=0}^du_i'=0\r\}\subset\R^{d+1},\\
 \R_+^{d+1} &=\l \{\theta\in\R^{d+1}\; : \; \theta_i>0\;\,\forall i=0,1,\dots,d+1\r\}.
 \eeqs
 \subsection{Preliminary calculations and observations}
 \subsubsection{Conversions between $d$ and $d+1$ dimensions}\label{subsub:dim}
We denote vectors in $\R^d$ with a superscript $d$ and index them starting at 1, i.e. $u^d=(u_1,\dots, u_d)$. We index vectors in $\R^{d+1}$ starting at zero, i.e. $u=(u_0,u_1,\dots, u_d)$. We let $e_0,e_1,\dots,e_d$ denote the standard unit vectors in $\R^{d+1}$.

 Let $M\in\R^{(d+1)\times d}$ be given by 
 $$M=\begin{pmatrix}
 -1 &-1 &\dots & -1\\
1 &0&\dots&0\\
0&1&\dots&0\\
\vdots &\vdots &\vdots&\vdots\\
0&0&\dots&1
 \end{pmatrix}.$$
We define the following transformations from $\R^d\to\R^{d+1}$:
\beqs
u^d&\mapsto u'=Mu^d\\
\theta^d&\mapsto \theta=e_0+M\theta^d
\eeqs
Thus the first transformation is indicated by removing the $d$ and adding the prime, while the second transformation is indicated by removing the $d$ only. Note that although we remove the superscript $d$, we are actually adding a coordinate in both cases. Explicitly, if $u^d=(u_1,\dots, u_d)$ then $u'=(u_0,u_1,\dots, u_d)$, where $u_0= -\sum_{i=1}^du_i$. Similarly if $\theta^d=(\theta_1,\dots,\theta_d)$ then $\theta=(\theta_0,\theta_1,\dots,\theta_d)$, with $\theta_0=1-(\theta_1+\dots+\theta_d)$. Note that if $\theta^d\in\Theta^d$ then $\theta\in\Theta_{d}$. Also, note that $u'\in\Omega_d$ for all $u\in\R^d$. Finally, let $\theta^d, \tau^d\in\R^d$. Then
$$(\theta^{d}-\tau^{d})'=\theta-\tau.$$
Next, define $v^d:\Theta^d\to\R$ and $ v:\R_{+}^{d+1}\to\R$ by 
\beqs\label{def:barv} 
v^d(\theta^d)&=-\nb_0\log\l(1-\sum_{j=1}^d\theta_j\r)-\sum_{j=1}^d\nb_j\log\theta_j,\\
v(\theta)&=-\sum_{j=0}^d\nb_j\log\theta_j.\eeqs Thus $$v^d(\theta^d)= v(\theta)= v(e_0+M\theta^d).$$ Note that $\nabla v^d(\theta^d)=M^\T\nabla v(\theta)$ and therefore $\la\nabla v^d(\theta^d),u^d\ra = \la\nabla v(\theta), u'\ra$. More generally, we have
 \beqs
 \nabla^2v^d(\theta^d) &= M^\T\nabla^2v(\theta)M,\\
 \la\nabla^k v^d(\theta^d), (u^d)^{\otimes k}\ra &= \la\nabla^k v(\theta), {u'}^{\otimes k}\ra,\\
 \la\nabla^3 v^d(\theta^d), A\otimes b^d\ra &=\la\nabla^3 v(\theta), MAM^\T\otimes b'\ra,\\
 \la\nabla^3v^d(\theta^d), A\ra &=M^\T\la\nabla^3 v(\theta), MAM^\T\ra.
 \eeqs

 \subsubsection{Derivatives of $v^d$ and $ v$, and chi-squared divergence}\label{subsub:vbarv}
 We see from~\eqref{def:barv} that the derivatives of $v$ are 
 \beq\label{barkderiv}\nabla^k v(\theta) = (-1)^k(k-1)!\sum_{j=0}^d\frac{\nb_j}{\theta_j^k}e_j^{\otimes k}.\eeq
 Note that $\nabla v(\nb) = \ind$, the all ones vector in $\R^{d+1}$. Therefore $\nabla v^d(\nb^d)=M^\T\ind=0$. Moreover, fix $\theta^d\in\Theta^d$. Then
 $$\la\nabla^2v^d(\theta^d), (u^d)^{\otimes2}\ra = \la\nabla^2 v(\theta), {u'}^{\otimes2}\ra = \sum_{j=0}^d\nb_j({u_j'}/\theta_j)^2,$$ which is strictly positive when $u^d\neq0$, because we have assumed the $N_j$, and hence also the $\nb_j$, are strictly positive. This implies $v^d$ is strictly convex on $\Theta^d$, so that $\nb^d$ is the unique global minimizer of $v^d$ in $\Theta^d$. Let $\nb^{-1} = \{\nb_j^{-1}\}_{j=0}^d$ and similarly, $(\nb^d)^{-1} = \{ \nb_j^{-1}\}_{j=1}^d$. We compute
 \beqs\label{Hvandco}
 H_v&=\nabla^2v^d(\nb^d)=M^\T\nabla^2 v( \nb)M = M^\T\mathrm{diag}(\nb^{-1})M = \mathrm{diag}((\nb^d)^{-1}) + \nb_0^{-1}\ind\ind^\T,\\
 H_v^{-1}&=\mathrm{diag}(\nb^d)-\nb^d(\nb^d)^\T,\qquad MH_v^{-1}M^\T = \mathrm{diag}(\nb)-\nb\nb^\T.
 \eeqs
 The equations in the second line follow by direct calculations using $H_v$ from the first line. These calculations imply that the Laplace approximation to $\pi^d\propto e^{-nv^d}$ is 
 $$\lap^d=\mathcal N(\nb^d, n^{-1}H_v^{-1})=\mathcal N(\nb^d, n^{-1}[\mathrm{diag}(\nb^d)-\nb^d(\nb^d)^\T]).$$ Moreover, since $\theta=e_0+M\theta^d$ is the pmf corresponding to $\theta^d$, we see that 
 \beq\label{th-thd}\theta^d\sim\lap^d\implies\theta\sim \mathcal N(\nb, n^{-1}MH_v^{-1}M^\T) = \mathcal N(\nb, n^{-1}[\mathrm{diag}(\nb)-\nb\nb^\T]).\eeq

 Next, note that \beq\label{unorm}\|u^d\|^2_{H_v} = (u^d)^\T H_vu^d = {u'}^\T\nabla^2 v(\nb)u' = \sum_{j=0}^d\frac{{u_j'}^2}{\nb_j}.\eeq In particular, since $(\theta^d-\nb^d)'=\theta-\nb$, we have
 \beq
 \|\theta^d - \nb^d\|^2_{H_v} = \sum_{j=0}^d\frac{(\theta_j-\nb_j)^2}{\nb_j}=:\chi^2(\theta||\nb).
 \eeq Therefore, the neighborhood $\U(\s)$ from Definition~\ref{Udef} is given by a chi-squared ball around $\nb$:
\beq\label{Uchi}\U(\s)=\{\theta^d\in\R^d\; : \;\chi^2(\theta||\nb)\leq \s^2d/n\}.\eeq
\begin{remark}
We will make frequent use of the bound
\beq\label{max2chi}
\max_{j=0,\dots,d}\l|\frac{\theta_j}{\nb_j}-1\r|^2\leq \chi^2(\theta||\nb)/\Nmin
\eeq which follows directly from the above definition of $\chi^2$ divergence.
\end{remark}
We also take note of the $\chi^2$ divergence between $\nb$ and the uniform distribution $\Unif_{d+1}:=\l(\frac{1}{d+1},\dots,\frac{1}{d+1}\r)$: 
\beqs\label{chiunif}
\chi^2(\Unif||\nb) &= (d+1)^{-2}\l(\sum_{j=0}^d\nb_j^{-1}\r)-1.
\eeqs
%
\subsubsection{Skew corrections}\label{app:skew:pmf}
For the density correction function $\Lhat$, we use the definition~\eqref{Lhatdef} of $\Lhat$, the properties discussed in Section~\ref{subsub:dim}, and the third derivative tensor of $ v$ given by~\eqref{barkderiv} with $k=3$, to get
\beqs\label{lhat-pmf}
\Lhat(\theta) &=-\frac n6\la\nabla^3v^d(\nb^d), (\theta^d-\nb^d)^{\otimes3}\ra = -\frac n6\la\nabla^3 v(\nb), (\theta- \nb)^{\otimes3}\ra \\
&= \frac n3\sum_{j=0}^d\frac{(\theta_j-\nb_j)^3}{\nb_j^2}.
\eeqs
Next, we compute $(\Lhatth)'=M\Lhatth$. Recall the notation $\nb^{-1} = \{ \nb_j^{-1}\}_{j=0}^d$. Using the definition~\eqref{Lm-def} of $\Lhatth$ and the formula for $MH_v^{-1}M^\T$ from~\eqref{Hvandco}, we have
\beqs
n(\Lhatth)' &=-\frac12MH_v^{-1}\la\nabla^3v^d(\nb^d), H_v^{-1}\ra = -\frac12MH_v^{-1}M^\T\la\nabla^3 v(\nb), MH_v^{-1}M^\T\ra\\
&=\l(\mathrm{diag}(\nb)-\nb\nb^\T\r)\lla\sum_{j=0}^d\nb_j^{-2}e_j^{\otimes 3}, \,\mathrm{diag}(\nb)-\nb\nb^\T\rra\\
&=\l(\mathrm{diag}(\nb)-\nb\nb^\T\r)\l(\nb^{-1}-\ind\r)=\ind-(d+1)\nb -\nb+\nb = \ind-(d+1)\nb.
\eeqs
Hence
\beq
(\Lhatth)' =\frac1n\ind-\frac{d+1}{n}\nb.
\eeq In particular, we compute using~\eqref{unorm} that
\beqs
\|\Lhatth\|^2_{H_V}&=n\|\Lhatth\|^2_{H_v}=n\sum_{j=0}^d\frac{\l[(\Lhatth)_j'\r]^2}{\nb_j}=n^{-1}\sum_{j=0}^d\frac{(1-(d+1)\nb_j)^2}{\nb_j}\\
&=n^{-1}\l(\bigg[\sum_{j=0}^d\nb_j^{-1}\bigg]- (d+1)^2\r)=\chi^2(\Unif_{d+1}||\nb)\frac{(d+1)^2}{n},
\eeqs using~\eqref{chiunif} to get the last equality.

\subsection{Proofs from Section~\ref{sec:c3:pmf}}\label{app:sec:c3:pmf}
\begin{proof}[Proof of Lemma~\ref{lma:eps3:pmf}]
Using~\eqref{barkderiv} and the identities from Sections~\ref{subsub:dim} and~\ref{subsub:vbarv}, we have
\beqs
c_3&=\|\nabla^3v^d(\nb^d)\|_{H_v} = \sup\l\{\la\nabla^3 v(\nb), {u'}^{\otimes3}\ra\; : \;  {u'}^\T \nabla^2 v(\nb) u' = 1\r\} \\
&= \sup\l\{-2\sum_{j=0}^d\frac{{u_j'}^3}{{\nb_j}^2}\; : \; \sum_{j=0}^du_j'=0, \sum_{j=0}^d\frac{{u_j'}^2}{\nb_j}=1\r\}.
\eeqs
Using Lagrange multipliers we get that $u_j'=t_1\nb_j$ for all $j$ in some nonempty index set $I\subset\{0,1,\dots,d\}$, and $u_j'=t_2\nb_j$ for all other $j$. Let $P=P(I)=\sum_{j\in I}\nb_j$. Writing the constraints in terms of $P$, $t_1$, and $t_2$, we get 
\beqs
t_1P+t_2(1-P)&=0,\\
t_1^2P+t_2^2(1-P)&=1.\eeqs 
Solving for $t_1,t_2$ we get $t_1=\mp\sqrt{(1-P)/P}$ and $t_2=\pm\sqrt{P/(1-P)}$. These two possible pairs $(t_1,t_2)$ give rise to a positive objective value and its negative, so we simply take the absolute value of the result. We get
\beqs
M(P):=-2\sum_{j=0}^d\frac{u_j^3}{{\nb_j}^2} &= -2\sum_{j\in I}t_1^3\nb_j - 2\sum_{j\in I^c}t_2^3\nb_j = -2t_1^3P -2t_2^3(1-P)\\
&=2\l|\frac{(1-P)^{3/2}}{\sqrt P} - \frac{P^{3/2}}{\sqrt{1-P}}\r| = 2\frac{|1-2P|}{\sqrt{P(1-P)}}
\eeqs Replacing $P$ by $1-P$ above does not change the answer, so assume without loss of generality that $P<1/2$. Then we can write
$$M(P)=2\frac{1-2P}{\sqrt{P(1-P)}}$$ It remains to choose the index set $I$, which determines $P$. Since $P\mapsto M(P)$ is decreasing on $[0,1]$, we should take $P$ as small as possible. The smallest possible value of $P$ over nonempty index sets $I$ is $P=\Nmin$. Hence
$$c_3=\max_IM(P(I)) = 2\frac{1-2\Nmin}{\sqrt{\Nmin}\sqrt{1-\Nmin}}.$$ The equation~\eqref{eq:eps3:pmf} for $\epsilon_3=c_3d/\sqrt n$ now follows.

A similar Lagrange multiplier argument shows that for $k=4,5$, we have $c_k(0)=\max_IM(P(I))$, where now
$$M(P)= (k-1)!\l(\frac{(1-P)^{\frac k2}}{P^{\frac k2-1}} -\frac{P^{\frac k2}}{(1-P)^{\frac k2-1}}\r) = (k-1)!\frac{(1-P)^{k-1}-P^{k-1}}{(P(1-P))^{\frac k2-1}}$$ This function is also decreasing, so we again choose $P=\Nmin$ to get
$$c_k(\s)\geq c_k(0)\geq (k-1)!\frac{(1-\Nmin)^{k-1}-\Nmin^{k-1}}{\sqrt{\Nmin(1-\Nmin)}^{k-2}}\geq \frac{c}{\sqrt{\Nmin}^{k-2}},$$ where the last inequality uses the assumption $\Nmin\leq 1/3$. Finally, recall that $\efour{\s} = c_4(\s)^{\frac12}(d/\sqrt n)$ and $\efive{\s} =c_5(\s)^{\frac13}(d/\sqrt n)$. Taking the square root and the cube root in the above bound, in the case $k=4$ and $k=5$ respectively, finishes the proof of the lower bound on $\ek{k}{\s}$ in~\eqref{eq:eps4:pmf}. The same lower bound also holds on $\ethree{\s}$ by the equation~\eqref{eq:eps3:pmf}.

Next, we prove the upper bound in~\eqref{eq:eps4:pmf}. Recall that $c_k(\s)=\sup\{\|\nabla^kv^d(\theta^d)\|_{\nabla^2v^d(\nb^d)}\,:\, \theta^d\in\U(\s)\}$. Fix $\theta^d\in\U(s)$. Again using~\eqref{barkderiv} and the identities from Sections~\ref{subsub:dim} and~\ref{subsub:vbarv}, we have
\beqs
\|\nabla^kv^d(\theta^d)\|_{\nabla^2v^d(\nb^d)}&= \sup\l\{\la\nabla^k v(\theta),{u'}^{\otimes k}\ra\;: \; {u'}^\T \nabla^2 v(\nb)u'=1\r\}\\
&=\sup\l\{(-1)^k(k-1)!\sum_{j=0}^d\nb_j\l(\frac{u_j'}{\theta_j}\r)^k\; : \sum_{j=0}^du_j'=0,\; \sum_{j=0}^d\nb_j\l(\frac{u_j'}{\nb_j}\r)^2=1 \r\}\\
&\leq\sup\l\{(k-1)!\sum_{j=0}^d\nb_j\l(\frac{u_j'}{\theta_j}\r)^k\; : \;u'\in\R^{d+1}, \sum_{j=0}^d\nb_j\l(\frac{u_j'}{\nb_j}\r)^2=1 \r\}
\eeqs
To get the third line, we dropped the constraint on the sum of the $u_j'$, and noted that we can always choose the sign of the $u_j'$'s to align with $(-1)^k$, so we can drop $(-1)^k$ from the objective. Using Lagrange multipliers, we see that the maximum is achieved at a $u'$ such that $u_j'=0$ for $j\in I$ and $u_j' = \lambda\theta_j^{\frac{k}{k-2}}/{\nb_j}^{\frac{2}{k-2}}$ for $j\in\{0,\dots,d\}\setminus I$, where $I$ is some non-empty index set. The constraint $\sum_{j=0}^d\nb_j(u_j'/\nb_j)^2=1$ reduces to $\lambda^2S=1$ and the objective $(k-1)!\sum_{j=0}^d\nb_j(u_j'/\theta_j)^k$ reduces to $(k-1)!\lambda^kS$, where $S=\sum_{j\in I}\nb_j(\theta_j/\nb_j)^{\frac{2k}{k-2}}$. Thus $\lambda=1/\sqrt S$, and the objective is given by $(k-1)!\lambda^kS=(k-1)!S^{1-\frac k2}$, which is maximized when $S=\min_{j=0,\dots,d}\nb_j(\theta_j/\nb_j)^{\frac{2k}{k-2}}$. We conclude that
\beq\label{nablatemp}
\|\nabla^kv^d(\theta^d)\|_{\nabla^2v^d(\nb^d)}\leq (k-1)!\max_j\l(\frac{\theta_j}{\nb_j}\r)^{k}\l(\frac{1}{\Nmin}\r)^{\frac k2-1}
\eeq
Next, by~\eqref{Uchi} we have that $\chi^2(\theta||\nb)\leq\s^2d/n$ for all $\theta\in\U(\s)$. Moreover, the bound~\eqref{max2chi} implies we have
\beq\label{nablasup}\sup_{j=0,\dots,d}\l|\frac{\theta_j}{\nb_j}-1\r|^2\leq \chi^2(\theta||\nb)/\Nmin \leq \frac{\s^2d}{n\Nmin}.\eeq
Using~\eqref{nablasup} in~\eqref{nablatemp} gives
$$
c_k(\s)=\sup_{\theta^d\in\U(\s)}\|\nabla^kv^d(\theta^d)\|_{H_v}\leq (k-1)!\l(\sqrt{\frac{\s^2d}{n\Nmin}} +1\r)^k\l(\frac{1}{\sqrt{\Nmin}}\r)^{k-2}.
$$ Recalling that $\ek{k}{\s}=c_k(\s)^{\frac{1}{k-2}}(d/\sqrt n)$ for $k=3,4,5$ concludes the proof of~\eqref{eq:eps4:pmf}. 
\end{proof}
\begin{proof}[Proof of Lemma~\ref{lma:tildeps-pmf}]
Let $\theta^d\in\R^d$ be a random vector with distribution $\lap^d = \mathcal N(\nb^d, n^{-1}(\mathrm{diag}(\nb^d)-\nb^d(\nb^d)^\T))$. Recall from~\eqref{th-thd} that this implies
$$\theta\sim\lap=\mathcal N\l(\nb, n^{-1}(\mathrm{diag}(\nb)-\nb\nb^\T)\r).$$

We now use the formula~\eqref{lhat-pmf} for $\Lhat$ to get
\beqs
\tilep&=\E_{\theta^d\sim\lap^d}[\Lhat(\theta^d)^2] = \frac{n^2}{9}\sum_{j=0}^d\frac{1}{\nb_j^4}\E_{\theta\sim\lap}\l[(\theta_j-\nb_j)^6\r]+\frac{n^2}{9}\sum_{j\neq k}\frac{1}{\nb_j^2\nb_k^2}\E_{\theta\sim\lap}\l[(\theta_j-\nb_j)^3(\theta_k-\nb_k)^3\r]\\
&=\frac{15}{9n}\sum_{j=0}^d\frac{1}{\nb_j^4}(\nb_j(1-\nb)_j)^3 + \frac{1}{9n}\sum_{j\neq k}\frac{1}{\nb_j^2\nb_k^2}\l[6(-\nb_j\nb_k)^3 + 9(-\nb_j\nb_k)\nb_j(1-\nb_j)\nb_k(1-\nb_k)\r]
\eeqs
To get the third line, we used that $\E[X^6]=15\sigma_X^2$ and $\E[X^3Y^3] = 6\sigma_{XY}^3 + 9\sigma_{XY}\sigma_X^2\sigma_Y^2$ for mean zero Gaussians $X,Y$ with variance $\sigma_X^2$, $\sigma_Y^2$, respectively, and covariance $\sigma_{XY}$. Simplifying the last line, and multiplying by $n$ to simplify the expressions, we get
\beqs
n\tilep&=\frac{15}{9}\sum_{j=0}^d\frac{(1-\nb_j)^3}{\nb_j} + \frac{1}{9}\sum_{j\neq k}\l[-6\nb_j\nb_k- 9(1-\nb_j)(1-\nb_k)\r]\\
&=\frac{15}{9}\sum_{j=0}^d\l[\nb_j^{-1} - 3 + 3\nb_j - \nb_j^2\r] +  \sum_{j\neq k}\l[(\nb_j+\nb_k)-1-\frac{15}{9}\nb_j\nb_k\r]\\
&=\frac{15}{9}\sum_{j=0}^d\nb_j^{-1} -\frac{15}{9}\l(\sum_{j=0}^d\nb_j\r)^2 -5(d+1) + 5\sum_{j=0}^d\nb_j - d(d+1) + 2d\sum_{j=0}^d\nb_j\\
&= \frac{5}{3}\sum_{j=0}^d\nb_j^{-1}-\frac53 -5(d+1) + 5-d(d+1)+2d=\frac{5}{3}\sum_{j=0}^d\nb_j^{-1}-4d-d^2.
\eeqs
Next, write $\sum_{j=0}^d\nb_j^{-1}=(d+1)^2\chi^2(\Unif_{d+1}||\nb) +(d+1)^2$, using~\eqref{chiunif}. Substituting this above, we get
\beqs
n\tilep &= \frac53(d+1)^2\chi^2(\Unif_{d+1}||\nb) +\frac53(d+1)^2 -4d-d^2 \\
&=\frac53(d+1)^2\chi^2(\Unif_{d+1}||\nb) +\frac23d^2-\frac23d.
\eeqs
\end{proof}
\subsection{Proofs from Section~\ref{sec:TV:pmf}}\label{app:sec:TV:pmf}
\begin{proof}[Proof of Proposition~\ref{prop:TV:pmf}]
The bound~\eqref{LTV-bd-pmf} follows immediately from the observation $\LTV\leq \|\Lhat\|_{L^2(\lap)}=\tilep$, and the bound on $\tilep$ from Lemma~\ref{lma:tildeps-pmf}. To prove~\eqref{RTV-bd-pmf}, we use Corollary~\ref{corr:corTV}.
As noted following Remark~\ref{rk:order:pmf}, Assumption~\ref{assume:c0} is satisfied with $\s_0=4$ and $\cgro=1$ provided $d/n\Nmin$ is smaller than an absolute constant, which holds since we have assumed the stricter condition that $\epsilon^2:=d^2/n\Nmin$ is small. Next, Lemma~\ref{lma:eps3:pmf} yields the following bound on $\del{\s}$:
\beq\label{e234}
\del{\s} \les\frac{d^2}{n\Nmin}\l(1 + \l(\sqrt{\frac{\s^2d}{n\Nmin}} +1\r)^4\r)\les \frac{d^2}{n\Nmin}\l(1 +\frac{\s^4d^2}{n\Nmin}\r),
\eeq
where we have used that $1/(n\Nmin)^2\les 1/(n\Nmin)$, since $d^2/(n\Nmin)$ and hence also $1/(n\Nmin)$ is bounded by a small constant. We now take $\s^4=\epsilon^{-2}=n\Nmin/d^2$, which satisfies the requirement $\s=\epsilon^{-1/2}\geq\sst =\min(4, 8\log(2e))$ when $d^2/n\Nmin$ is small. We also show at the end of the proof that $\U(\s)\subset\Theta$. The bound~\eqref{e234} now implies that $\Es{\s}=\e((\del{\s})\s^4)\leq C$, and $\del{\s}\les d^2/(n\Nmin)$. Furthermore,
$$\taus{\s} = d\e\l(-\cgro\s d/8\r) = d\e\l(-C\s d\r) = d\e\l(-C\sqrt d(n\Nmin)^{1/4}\r).$$ Substituting these bounds into~\eqref{TVcorlap} concludes the proof of~\eqref{RTV-bd-pmf}, pending the proof that $\U(\s)=\U(1/\sqrt{\epsilon})\subset\Theta$, which we now show. Fix $\theta\in\U(\s)$. We will show that $\theta_i>0$ for all $i$, which suffices to prove $\theta\in\Theta$. By~\eqref{max2chi} we have
$$\frac{\theta_i}{\bar N_i}\geq 1-\frac{\chi(\theta||\bar N)}{\sqrt{\Nmin}}\geq 1-\frac{\sqrt d}{\sqrt{n\Nmin}\sqrt\epsilon}\geq 1-\sqrt{\epsilon}>0$$ as long as $\epsilon<1$.\end{proof}
To prove Theorem~\ref{prop:pmf:lb}, we start with a preliminary lemma. 
\begin{lemma}\label{lma:deltaI}
Suppose $\Nmin\geq 6/\sqrt{n+d}$. Then for all strict, nonempty subsets $\I$ of $\{0,1,\dots,d\}$, it holds
\beq
\frac{||\I| - d\nb_\I|}{\sqrt{n+d}\sqrt{\nb_\I(1-\nb_\I)}} \leq \frac12,
\eeq where $\nb_\I=\sum_{i\in\I}\nb_i$. 
\end{lemma} The proof is at the end of the section.
\begin{proof}[Proof of Theorem~\ref{prop:pmf:lb}]Let $\theta^d\sim\pi$ and $X^d\sim\lap$, so that $\theta\sim\Dir(N_0+1,N_1+1,\dots, N_d+1)$ and $ X\sim\mathcal N(\nb, n^{-1}(\mathrm{diag}(\nb)-\nb\nb^\T))$. Next, given a set $\I\subset\{0,1,\dots,d\}$, define
$$ X_{\I}=\sum_{i\in \I} X_i,\qquad\theta_{\I}=\sum_{i\in \I}\theta_i,\qquad  \nb_\I= \sum_{i\in\I}\nb_i,\qquad  N_\I = n\nb_\I.$$ Also, let $\I^c=\{0,1,\dots,d\}\setminus I$. Note that 
$$\PP( X_\I\leq\nb_\I)=1/2$$ for all nonempty, strict subsets $\I$ of $\{0,1,\dots,d\}$. Therefore,
\beq\label{TV-lb}\TV(\pi,\lap)\geq\sup_\I  \l|\PP( \theta_\I\leq\nb_\I)-\PP( X_\I\leq\nb_\I)\r|= \sup_\I\l|\PP( \theta_\I\leq\nb_\I)-1/2\r|,\eeq where we take the supremum over a certain subset of all the nonempty, strict subsets $\I$ of $\{0,1,\dots,d\}$.

Now, by the aggregation property of the Dirichlet~\cite{dirichlet}, it holds $(\theta_\I, \theta_{\I^c})\sim\Dir( N_\I+|\I|,  N_{\I^c} + d+1-|\I|)$, which is the same as saying that
$$\theta_\I\sim\mathrm{Beta}(a,b),\qquad a= N_\I+|\I|, \quad b= N_{\I^c} + d+1-|\I|.$$ Therefore,
$$\PP (\theta_\I\leq \nb_\I) = I_{\nb_\I}(a,b) = \frac{B(\nb_\I; a,b)}{B(a,b)},$$ where $I_x(a,b)$ is the regularized incomplete beta function, defined as the ratio of the incomplete beta function and the beta function. Now, since $a,b$ are positive integers, we can use that the incomplete beta function is the complementary cdf of a binomial distribution~\cite[p. 52]{wadsworth1960introduction} to get that
$$\PP (\theta_\I\leq \nb_\I)=I_{\nb_\I}(a,b) = \PP (\mathrm{Bin}(n+d, \nb_\I)\geq a),$$ where $\mathrm{Bin}(n+d, \nb_\I)$ is a binomial distribution with $n+d$ trials and success probability $\nb_\I$. We have
$$
\mathrm{Bin}(n+d, \nb_\I)=\sum_{i=1}^{n+d}B_i,\qquad B_i\iid\mathrm{Bernoulli}(\nb_\I),$$ and hence
\beqs\label{theta-nb}
\PP (\theta_\I\leq \nb_\I)&=\PP (\mathrm{Bin}(n+d, \nb_\I)\geq a) = \PP \l(\sum_{i=1}^{n+d}B_i\geq a\r) \\
&= \PP \l(S_{n+d} \geq \frac{a-(n+d)\nb_\I}{\sqrt{(n+d)\nb_\I(1-\nb_\I)}}\r)\\
&= \PP \l(S_{n+d} \geq \frac{|\I|-d\nb_\I}{\sqrt{(n+d)\nb_\I(1-\nb_\I)}}\r).\eeqs We used the definition of $a$ to get the last equality, and we have defined
$$S_{n+d}:=\sqrt{(n+d)\nb_\I(1-\nb_\I)}^{\,-1}\sum_{i=1}^{n+d}(B_i-\nb_\I).$$ Next, we apply Theorem~\ref{berry}, the Berry-Esseen theorem, with $Y_i = B_i - \nb_\I$. We have $\sigma^2=\nb_\I(1-\nb_\I)$, and we compute
$$\beta = \E\l[\l|B_i-\nb_\I\r|^3\r]=\nb_\I(1-\nb_\I)^3 + (1-\nb_\I)\nb_\I^3=\nb_\I(1-\nb_\I)\l[(1-\nb_\I)^2+\nb_\I^2\r]\leq2\nb_\I(1-\nb_\I)=2\sigma^2.$$

We conclude that
\beq\label{BE}
\sup_{t\in\R}|\PP (S_{n+d}\geq t)-(1-\Phi(t))| \leq \frac12\frac{\beta}{\sigma^3\sqrt{n+d}} \leq\sqrt{(n+d)\nb_\I(1-\nb_\I)}^{\,-1}
\eeq where $\Phi$ is the standard Gaussian CDF. Now, fix $t=\frac{|\I|-d\nb_\I}{\sqrt{(n+d)\nb_\I(1-\nb_\I)}}$. Using~\eqref{BE} and that $\Phi(0)=1/2$, we have
\beqs\label{BE2}
\l|1/2- \PP (S_{n+d} \geq t) \r|&\geq \l| 1/2-(1-\Phi(t))\r| -\l|(1-\Phi(t))- \PP (S_{n+d} \geq t) \r|\\
&\geq  \l|\Phi(0)-\Phi(t)\r|-\sqrt{(n+d)\nb_\I(1-\nb_\I)}^{\,-1}
\eeqs Next, using that $|t|\leq 1/2$ by Lemma~\ref{lma:deltaI}, we have 
$$\l|\Phi(0)-\Phi(t)\r| = \l|\int_0^t\frac{1}{\sqrt{2\pi}}e^{-x^2/2}dx\r| \geq \frac{|t|}{\sqrt{2\pi}}e^{-t^2/2}\geq \frac {|t|}{3}.$$ Substituting this lower bound into~\eqref{BE2} and using the definition of $t$ gives
\beqs\label{BE3}
\l|\frac12-\PP \l(S_{n+d} \geq \frac{|\I|-d\nb_\I}{\sqrt{(n+d)\nb_\I(1-\nb_\I)}}\r)\r| &\geq  \frac{||\I|-d\nb_\I|}{3\sqrt{(n+d)\nb_\I(1-\nb_\I)}}-\frac{1}{\sqrt{(n+d)\nb_\I(1-\nb_\I)}}\\
&= \frac{||\I|-d\nb_\I|-3}{3\sqrt{(n+d)\nb_\I(1-\nb_\I)}}\geq \frac{||\I|-(d+1)\nb_\I|-4}{3\sqrt{(n+d)\nb_\I(1-\nb_\I)}}
\eeqs
Now, recall from~\eqref{theta-nb} that the lefthand side of~\eqref{BE3} is precisely $|\PP (\theta_\I\leq \nb_\I)-\frac12|$. Returning to~\eqref{TV-lb}, we now take the supremum over $\I\in\mathscr{I}$, the set of all nonempty, strict subsets $\I$ of $\{0,1,\dots,d\}$ for which $||\I|-(d+1)\nb_\I|\geq4$. This gives
\beqs\label{TVfinally}
\TV(\pi^d,\lap^d) &\geq \sup_{\I\in\mathscr{I}}\frac{||\I|-(d+1)\nb_\I|-4}{3\sqrt{(n+d)\nb_\I(1-\nb_\I)}}\geq \frac23\sup_{\I\in\mathscr{I}}\frac{||\I|-(d+1)\nb_\I|-4}{\sqrt{n+d}}\\
&=\frac23\frac{d+1}{\sqrt{n+d}}\sup_{\I\in\mathscr{I}}|\Unif(\I)-\nb_\I| -\frac{8}{3\sqrt{n+d}},
\eeqs where $\Unif(\I)=|\I|/(d+1)$, which is the probability of set $\I$ under the uniform pmf. Similarly, note that $\nb_\I$ is the probability of set $\I$ under the pmf $\nb$. We claim that if $\TV(\Unif,\nb)\geq4/(d+1)$, then $\sup_{\I\in\mathscr{I}}|\Unif(\I)-\nb_\I|=\TV(\Unif,\nb)$. Indeed, using standard properties of TV distance and the assumption $\TV(\Unif,\nb)\geq4/(d+1)$, it holds
$$\frac{4}{d+1}\leq \TV(\Unif,\nb)=\sup_{\I\subseteq\{0,1,\dots,d\}}|\Unif(\I)-\nb_\I| = \Unif(\I_0)-\nb_{\I_0},$$ where $\I_0=\{i\in\{0,\dots,d\}\; : \; \Unif(\{i\}) > \nb_i\}$. Therefore, $\I_0$ must lie in $\mathscr{I}$, and therefore the supremum of $|\Unif(\I)-\nb_\I|$ over all $\I\in\mathscr{I}$ is precisely the TV distance. Substituting this into~\eqref{TVfinally} gives
\beq
\TV(\pi^d,\lap^d) \geq \frac23\frac{d+1}{\sqrt{n+d}}\l(\TV(\Unif,\nb) - \frac{4}{d+1}\r)
\eeq
Finally, since we have in fact assumed the stricter inequality $\TV(\Unif,\nb)\geq 6/(d+1)$, it holds $4/(d+1)\leq \tfrac23\TV(\Unif,\nb)$. Substituting this above gives
$\TV(\pi,\lap)\geq \frac29\TV(\Unif_{d+1},\nb)\frac{d+1}{\sqrt{n+d}}$. We now note that $6/\sqrt{n+d}\leq\Nmin\leq 1/(d+1)$ implies $n\geq d$, so $\tfrac29(d+1)/\sqrt{n+d}\geq\tfrac19d/\sqrt n$
\end{proof}
\begin{proof}[Proof of Lemma~\ref{lma:deltaI}] We consider two cases. First, suppose $\nb_\I\leq1/2$. Then $\nb_\I(1-\nb_\I)\geq\nb_\I^2$ and hence
\beqs
\frac{||\I| - d\nb_\I|}{\sqrt{n+d}\sqrt{\nb_\I(1-\nb_\I)}} &\leq \frac{|\I| + d\nb_\I}{\sqrt{n+d}\,\nb_\I} \leq \frac{|\I|}{\sqrt{n+d}\,\nb_\I} + \frac{d}{\sqrt{n+d}}\\
&\leq \frac{1}{\sqrt{n+d}\,\Nmin} + \frac{d}{\sqrt{n+d}} \leq \frac{2}{\sqrt{n+d}\,\Nmin}\leq \frac13.
\eeqs We used that $\nb_\I\geq|\I|\Nmin$ to get the first inequality in the second line, and the fact that $d\leq\Nmin^{-1}$ (since $\Nmin\leq 1/(d+1)$) to get the second inequality in the second line.

Next, suppose $\nb_\I\geq1/2$, so that $1-\nb_\I = \nb_{\I^c}\leq 1/2$, and $\nb_\I(1-\nb_\I)\geq\nb_{\I^c}^2$. Note that
$$\l||\I| - d\nb_\I \r|= \l|d+1-|\I^c| - d\nb_\I\r| =\l|d\nb_{\I^c} + 1-|\I^c|\r|\leq d\nb_{\I^c} + 1+|\I^c|\leq d\nb_{\I^c} + 2|\I^c|,$$ where the last inequality holds since $|\I^c|\geq1$ by assumption. By similar logic as in the first case, we get
\beqs
\frac{||\I| - d\nb_\I|}{\sqrt{n+d}\sqrt{\nb_\I(1-\nb_\I)}} &\leq \frac{2}{\sqrt{n+d}\,\Nmin} + \frac{d}{\sqrt{n+d}} \leq \frac{3}{\sqrt{n+d}\,\Nmin}\leq \frac12.
\eeqs
\end{proof}
\begin{theorem}[Berry-Esseen Theorem,~\cite{il2010refinement}]\label{berry}
Let $Y_1,\dots, Y_n$ be i.i.d. mean zero random variables, with $\E[Y_i^2]=\sigma^2$ and $\E[|Y_i|^3]=\beta<\infty$. Let $S_n = \frac{1}{\sigma\sqrt n}(Y_1+\dots+Y_n)$. Then
\beq
\sup_{t\in\R}\l|\PP(S_n\leq t)-\Phi(t)\r|\leq C\frac{\beta}{\sigma^3\sqrt n},\qquad C\leq 0.4785,
\eeq where $\Phi$ is the cdf of the standard normal distribution.
\end{theorem}
\section{Proofs from Section~\ref{sec:log}}\label{app:sec:log}
\renewcommand{\ground}{\beta}
\begin{proof}[Proof of~\eqref{tildec3form}]
We use the definition of $\tilep$ and the fact that $W(b)=V(\bhat + H_V^{-1/2}b)$ to get
\beqs\label{nv0}
\tilep^2=&\frac16\|\nabla^3W(0)\|_F^2 + \frac14\|\la\nabla^3W(0), I_d\ra\|^2\\
=&\frac16\sum_{i,j,k=1}^d\la\nabla^3W(0), e_i\otimes e_j\otimes e_k\ra^2 + \frac14\sum_{k=1}^d\la\nabla^3W(0), I_d\otimes e_k\ra^2\\
=&\frac16\sum_{i,j,k=1}^d\la\nabla^3V(\bhat), H_V^{-1/2}e_i\otimes H_V^{-1/2}e_j\otimes H_V^{-1/2}e_k\ra^2 \\
&+ \frac14\sum_{k=1}^d\la\nabla^3V(\bhat), H_V^{-1}\otimes H_V^{-1/2}e_k\ra^2\eeqs
Next, recall that $\nabla^3V(\bhat) = \sum_{\ell=1}^na_\ell X_\ell^{\otimes 3}$, where $a_\ell=\psi'''(X_\ell^T\bhat)$. Therefore,
\beqs\label{nv1}
\la\nabla^3V(\bhat), H_V^{-1/2}e_i&\otimes H_V^{-1/2}e_j\otimes H_V^{-1/2}e_k\ra \\
&= \sum_{\ell=1}^na_\ell(X_\ell^TH_V^{-1/2}e_i)(X_\ell^TH_V^{-1/2}e_j)(X_\ell^TH_V^{-1/2}e_k)\\
&=\sum_{\ell=1}^na_\ell (B_\ell)_i(B_\ell)_j(B_\ell)_k,\eeqs where $B_\ell= H_V^{-1/2}X_\ell$. Similarly,
\beq\label{nv2}
\la\nabla^3V(\bhat), H_V^{-1}\otimes H_V^{-1/2}e_k\ra =  \sum_{\ell=1}^na_\ell\|B_\ell\|^2(B_\ell)_k.\eeq
Using~\eqref{nv1},~\eqref{nv2} in the last line of~\eqref{nv0}, we get
\beqs
\tilep^2=&\frac16\sum_{i,j,k=1}^d\l(\sum_{\ell=1}^na_\ell (B_\ell)_i(B_\ell)_j(B_\ell)_k\r)^2 \\
&+ \frac14\sum_{k=1}^d\l(\sum_{\ell=1}^na_\ell \|B_\ell\|^2(B_\ell)_k\r)^2\\
=&\sum_{\ell,m=1}^na_\ell a_m\l[\frac16(B_\ell^TB_m)^3 + \frac14\|B_\ell\|^2\|B_m\|^2B_\ell^TB_m\r]
\eeqs Now, substituting the definition of $a_\ell, a_m, B_\ell, B_m$ we get the expression in~\eqref{tildec3form}.
\end{proof}
Our proof of Proposition~\ref{prop:log} is based on the following lemma.
\begin{lemma}\label{MLEexist}
Let $v$ be convex and $C^3$ on $\R^{d}$, and let $0<\lambda\leq1$. If (1) $\|\M^{-1/2}\nabla v(\ground)\|\leq 2\lambda s$, (2) $\nabla^2v(\ground)\succeq\lambda\M\succ0$, and (3) $\|\nabla^{2}v(b)-\nabla^{2}v(\ground)\|_{\M}\leq\lambda/4$ for all $\|b-\ground\|_{\M}\leq s$, then the function $v$ has a unique global minimizer $\MLE$. Moreover, $\hat b$ satisfies $\|\MLE-\ground\|_\M\leq s$ and $\nabla^{2}v(\hat b)\succeq\frac{3\lambda}{4}\M$.
\begin{proof}
In~\cite[Corollary F.6]{katsBVM}, it is shown that if $v\in C^3$ and (1), (2), and (3) are satisfied, then there exists a point $\MLE$ such that $\nabla v(\MLE)=0$, and $\MLE$ satisfied $\|\MLE-\beta\|_\M\leq s$ and $\nabla^{2}v(\MLE)\succeq\frac{3\lambda}{4}\M\succ0$. But since $v$ is convex, the fact that $\nabla v(\MLE)=0$ and $\nabla^2v(\MLE)\succ0$ implies $\MLE$ is the unique global minimizer.
\end{proof}
\end{lemma}
We will check the conditions of this lemma in the following results. But first, we decompose $v$ into two parts. Recall from~\eqref{vndef} the definition of $v=n^{-1}V$. Note that we can write $v(b)=\ell(b) + n^{-1}\frac12b^\T\Sigma_0^{-1}b$, where
$$\ell(b)=-\frac1n\sum_{i=1}^nY_iX_i^\T b + \frac1n\sum_{i=1}^n\psi(X_i^\T b)$$ is the normalized, negative log likelihood. The derivatives of $\ell$ are given by
\beqs\label{nablalog}
\nabla^k\ell(b)&=\frac1n\sum_{i=1}^n\psi^{(k)}(b^\T X_i)X_i^{\otimes k},\quad \forall k=1,2,\dots.\eeqs

\begin{lemma}[Adaptation of Lemma 7, Chapter 3,~\cite{pragyathesis}]\label{lma:pragya}
Suppose $d<n$. Then  for some $\lambda=\lambda(\|\ground\|_{M})$ the event
$$\{\nabla^2\ell(\ground)\succeq\lambda\M\}=\l\{\frac1n\sum_{i=1}^n\psi''(X_i^{\T} \ground)X_iX_i^{\T} \succeq\lambda\M\r\}$$ has probability at least $1-4e^{-Cn}$ for an absolute constant $C$. 
\end{lemma} The function $\|b\|_\M\mapsto\lambda(\|b\|_\M)$ is nonincreasing, so if $\|\ground\|_{\M}$ is bounded above by a constant then $\lambda=\lambda(\|\ground\|_{\M})$ is bounded away from zero by a constant.
\begin{lemma}[Lemma 6.2 in~\cite{katsBVM} with $r=2\lambda s$]\label{bern-gauss}
Let $X_i\iid \mathcal N(0, \M)$, $i=1,\dots, n$, and $Y_i\mid X_i\sim\mathrm{Bernoulli}(\psi'(X_{i}^\T \ground))$. 
If $d/n\leq 1/8$ and $r\geq 4\sqrt2$, then the event
$$\{\|\M^{-1/2}\nabla\ell(\ground)\|\leq r\sqrt{d/n}\}=\l\{\l\|\frac1n\sum_{i=1}^n(Y_i-\E[Y_i\mid X_i])\M^{-1/2}X_i\r\|\leq r\sqrt{d/n}\r\}$$ has probability at least $1-e^{-n/4}-e^{-dr^{2}/16}$.
\end{lemma} 
\begin{lemma}\label{lma:adam}
Let $X_{i}\iid\mathcal N(0, \M)$, $i=1,\dots, n$, and assume $(\log(2n)\vee1)\sqrt{d/n}\leq 1$. Then \beq\label{lma:adam:eq}\sup_{\|u\|_{M}=1}\frac1n\sum_{i=1}^{n}|X_{i}^{\T}u|^{p} \leq C\l(1\vee \frac{d^{p/2}}{n}\r)\quad\forall p=3,4,5\eeq with probability at least $1-C\e(-C\sqrt m)$, where $m=\max(d,(\log n)^{2})$.
\end{lemma}
\begin{proof}
First, let $Z_{i}=\M^{-1/2}X_{i}\iid\mathcal N(0, I_{d})$, $i=1,\dots, n$ and note that \beq\label{Spdef}\sup_{\|u\|_{\M}=1}\frac1n\sum_{i=1}^{n}|u^{\T}X_{i}|^{p} =\sup_{\|u\|=1}\frac1n\sum_{i=1}^n|u^\T Z_i|^p.\eeq Let $Z_{i}\iid\mathcal N(0, I_{d})$, $i=1,\dots,n$ and let $\bar Z_{i} = (Z_{i}, Y_{i})\in\R^{m}$, where $Y_{i}\iid\mathcal N(0, I_{m-d})$ and the $Y_{i}$ are independent of all the $Z_{i}$. Thus $\bar Z_{i}\iid\mathcal N(0, I_{m})$. Now, define the random variable 
$$A(p, d, Z_{1:n}) = \sup_{u\in S^{d-1}}\l|\frac1n\sum_{i=1}^{n}\l(|Z_{i}^{\T}u|^{p}-\E[|Z_{i}^{\T}u|^{p}]\r)\r|,$$ where $S^{d-1}$ is the unit sphere in $\R^d$. Note that $A(p,d,Z_{1:n})\leq A(p,m,\bar Z_{1:n})$, since $u^{\T}Z_{i}=\bar u^{\T}\bar Z_{i}$, where $\bar u =(u,0)\in S^{m-1}$. Since $m\leq n\leq e^{\sqrt m}$, we can apply Proposition 4.4 of~\cite{adamczak2010quantitative} with $s=t=1$ to get that
\beqs
A(p,d,Z_{1:n})&\leq A(p,m,\bar Z_{1:n})\les \log^{p-1}\l(\frac{2n}{m}\r)\sqrt{\frac mn} + \frac{m^{p/2}}{n} + \frac {m}{2n} \\
&\les\log^{p-1}\l(\frac{2n}{m}\r)\sqrt{\frac mn} + \frac{m^{p/2}}{n}\\
&\les 1 + \frac{d^{p/2}}{n} +\frac{(\log n)^{p}}{n}\les 1+ \frac{d^{p/2}}{n}
\eeqs with probability at least
\beqs
1&-\e(-C\sqrt{m})-\e\l(-C\min\l(m\log^{2p-2}(2n/m), \sqrt{nm}/\log(2n/m)\r)\r)\\
&\geq 1-\e(-C\sqrt{m})-\e\l(-C\min\l(m, \sqrt{n}/\log(2n)\r)\r)\\
&\geq 1-2\e(-C\sqrt m)-\e(-C\sqrt n/\log(2n))\\
&\geq1-3\e(-C\sqrt{m}),\eeqs using that $\log(2n/m)\geq C$ and $\sqrt m\leq\sqrt n/\log(2n)$. Finally, note that $\sup_{u\in S^{d-1}}\frac1n\sum_{i=1}^{n}\E[|Z_{i}^{\T}u|^{p}]= C$, so 
$$\sup_{u\in S^{d-1}}\frac1n\sum_{i=1}^{n}|Z_{i}^{\T}u|^{p}\leq C + A(p,d,Z_{1:n}) \les1+ \frac{d^{p/2}}{n} \les 1\vee\frac{d^{p/2}}{n}$$ with the same probability as above.
\end{proof}
Combining Lemma~\ref{lma:adam} with the bound
\begin{equation}\label{nablaellp}
\|\nabla^{p}\ell(b)\|_{\M} \leq \|\psi^{(p)}\|_{\infty}\sup_{\|u\|_{\M}=1}\frac1n\sum_{i=1}^{n}|u^{\T}X_{i}|^{p},
\end{equation} and noting that $\|\psi^{(p)}\|_{\infty}\leq C$, $p=2,3,4,5$, we conclude that
\beq\label{nablaellp-bd}\sup_{b\in\R^{d}}\|\nabla^{p}\ell(b)\|_{\M} \leq C\l(1\vee \frac{d^{p/2}}{n}\r)\quad\forall p=3,4,5\eeq with probability at least $1-C\e(-C\sqrt m)=1-C\e(-C\max(\sqrt d,\log n))$.

\begin{proof}[Proof of Proposition~\ref{prop:log}]In this proof, the meaning of $C(\|\ground\|_M)$ and $C$ may change from line to line, with $C$ being an absolute constant. Let $\lambda = \lambda(\|\ground\|_{\M})$ be as in Lemma~\ref{lma:pragya}, so that 
\beq\label{nabla2-lb}\nabla^{2}v(\ground)=\nabla^{2}\ell(\ground)+\Sigma_{0}^{-1}\succeq\nabla^{2}\ell(\ground)\succeq\lambda \M\eeq with probability at least $1-4e^{-Cn}$. Next, fix $r\geq 4\sqrt 2$. 
Then Lemma~\ref{bern-gauss} gives that
\beqs\label{nabla-ub}
\|\M^{-1/2}\nabla v(\ground)\|&=\|\M^{-1/2}\nabla \ell(\ground) + n^{-1}\M^{-1/2}\Sigma_{0}^{-1}\ground\|\\
&\leq \|\M^{-1/2}\nabla \ell(\ground)\| +n^{-1}\|\Sigma_{0}^{-1}\|_{\M}\|\ground\|_{\M}\\
& \leq r\sqrt{d/n}+n^{-1}\|\Sigma_{0}^{-1}\|_{\M}\|\ground\|_{\M}\eeqs with probability at least $1-e^{-n/4}-e^{-dr^{2}/16}$. 
Furthermore, note that 
\beqs\label{nabla2}\|\nabla^{2}v(b)-\nabla^{2}v(\ground)\|_{\M}&=\|\nabla^{2}\ell(b)-\nabla^{2}\ell(\ground)\|_{\M}\\
&\leq \sup_{b\in\R^{d}}\|\nabla^{3}\ell(b)\|_{\M}\|b-\ground\|_{\M}\leq C\|b-\ground\|_\M\eeqs with probability at least $1 - C\e(-C\max(\sqrt d,\log n))$, using~\eqref{nablaellp-bd} with $p=3$ and the assumption $d\sqrt d\leq d^2\leq n$. We now use~\eqref{nabla2-lb},~\eqref{nabla-ub},~\eqref{nabla2} to satisfy the assumptions of Lemma~\ref{MLEexist}. Namely, we simply define $s$ to be $(1/2\lambda)$ times the righthand side of~\eqref{nabla-ub}, i.e.
$$s=\frac{1}{2\lambda}\l( r\sqrt{d/n}+n^{-1}\|\Sigma_{0}^{-1}\|_{\M}\|\ground\|_{\M}\r).$$ Thus (1) and (2) from Lemma~\ref{MLEexist} are satisfied. Finally, (3) is satisfied by~\eqref{nabla2} if $s\leq\lambda/4C$, which occurs if
$$ r\sqrt{d/n}+n^{-1}\|\Sigma_{0}^{-1}\|_{\M}\|\ground\|_{\M}\leq C\lambda^2=C(\|\ground\|_{\M}).$$ It suffices that
$$r\sqrt{d/n}\leq C(\|\ground\|_{\M}),\qquad \|\Sigma_0^{-1}\|_\M\leq C(\|\ground\|_{\M})n.$$ Thus we can take $r=C(\|\ground\|_{\M})\sqrt{n/d}$, provided the assumption $\sqrt{d/n}\leq C(\|\ground\|_{\M})/4\sqrt2=C(\|\ground\|_{\M})$ is satisfied. We conclude by Lemma~\ref{MLEexist} and the above statements that if $d/n\leq  C(\|\ground\|_{\M})$ and $\|\Sigma_0^{-1}\|_\M\leq C(\|\ground\|_{\M})n$, then a unique global minimizer $\bhat$ of $v$ exists, and $H_v=\nabla^2v(\bhat)\succeq (3\lambda/4)\M=C(\|\ground\|_{\M})\M$, with probability at least 
\beqs
1&-4e^{-Cn}-e^{-n/4}-e^{-d(C(\|\ground\|_{\M})\sqrt{n/d})^2/16}-Ce^{-C\max(\sqrt d,\log n)}\\
&\geq 1 -Ce^{-C(\|\ground\|_{\M})\max(\sqrt d,\log n)}.
\eeqs
This concludes the proof of point 1. Next, we use $H_v\succeq C(\|\ground\|_{\M})\M$ together with~\eqref{nablaellp-bd} with $p=3,4,5$ to get that
\beqs
\|\nabla^{p}v(b)\|_{H_{v}}& =\|\nabla^{p}\ell(b)\|_{H_{v}} =\sup_{u\neq0}\frac{\la\nabla^{p}\ell(b), u^{\otimes k}\ra}{(u^{\T}H_{v}u)^{p/2}} \leq C(\|\ground\|_{\M})\sup_{u\neq0}\frac{\la\nabla^{p}\ell(b), u^{\otimes k}\ra}{(u^{\T}\M u)^{p/2}} \\
&= C(\|\ground\|_{\M})\|\nabla^{p}\ell(b)\|_{\M}\leq C(\|\ground\|_{\M})\l(\frac{d^{p/2}}{n}\vee1\r).
\eeqs This proves point 3. Finally, we use that $v$ is convex with probability 1, and hence by Remark~\ref{vconvA2}, the bound~\eqref{assume:c0:eq} hold with $\s_0=4$ and $\cgro=1$, provided $c_3\sqrt{d/n})+ c_4(4)d/n\leq3/8$. Since $c_3,\sup_{\s\geq0}c_4(\s)\leq C(\|\ground\|_\M)$, this is satisfied if $\sqrt{d/n} \leq C(\|\ground\|_\M)$.\end{proof}
 \begin{proof}[Proof of Corollary~\ref{corr:logreg}]
 Note that making the assumption $d/\sqrt n$ small enough suffices to satisfy the assumptions $(\sqrt d\vee\log(2n))\sqrt{d/n}\leq C(\|\ground\|_{\M})\wedge1=C$ from Proposition~\ref{prop:log}.

 We now apply Theorem~\ref{thm:Vgen}. Proposition~\ref{prop:log} shows Assumptions~\ref{assume:1},~\ref{assume:c0} hold on an event of probability $1 - C\e(-C\max(\sqrt d,\log n))$, and we can take $\cgro=1,\s_0=4$ uniformly over the event. Furthermore, $\tilep\leq\epsilon_3\les\epsilon$, $\efour{\s}^2\les \epsilon^2$, and $d^{-1/2}\efive{\s}^3\les \epsilon^3$, where the latter two bounds hold uniformly over both the event and $\s\geq0$. Finally, therefore, the required condition $\epsilon_3,\efour{\s}\les1$ is satisfied on the event, regardless of choice of $\s$.

Specifically, we choose $\s=\epsilon^{-1/2}$. Note that since $\s_0=4$ and $\cgro=1$, we have $\sst=\max(\s_0,(8/\cgro)\log(2e/\cgro))=C$, and therefore $\s\geq\sst$ is satisfied if $\epsilon$ is small enough. Moreover, note that $\Es{\s}=\e((\del{\s})\s^4)\leq C$ using $\s=\epsilon^{-1/2}$ and $\epsilon_3\les\epsilon$, $\sup_{\s\geq0}\efour{\s}\les\epsilon$.  Moreover, we get $\taus{\s}=d\e(-\cgro d\s/8) = d\e(-Cd(\sqrt n/d)^{1/2}) = d\e(-Cn^{1/4}\sqrt d)$. 

Substituting these bounds into~\eqref{g-bd-2} gives
\beq
|R(g)|=|\Delta_g(\corlap)|\les \epsilon^2+ \big(\af{g}\vee1\big)\taus{\s}.
\eeq
The only change if $g$ is odd is that $\epsilon^2$ becomes $\epsilon^3$ (using Theorem~\ref{thm:odd}). Finally, we have $L(g)\leq\tilep\leq c_3\epsilon\les\epsilon$. Combining all of these bounds proves~\eqref{bds-gam-log-1}. The proof of the TV, mean and covariance bounds is similar: we simply use Corollaries~\ref{corr:corTV},~\ref{thm:Vmean},~\ref{thm:Vcov} with $\s=\epsilon^{-1/2}$.
 \end{proof}

\section{Proofs from Section~\ref{sec:overview}}\label{app:sec:overview}
\begin{proof}[Proof of Lemma~\ref{lma:pretty}]
First, note that $\int (1+h)\bar f_{h}d\gamma = \int (1+h)fd\gamma - \int (1+h)d\gamma\int f(1+h)d\gamma=0$ by the assumption that $\int hd\gamma=0$. Furthermore, $d\rho\propto Qd\gamma$ implies $\int Qd\gamma\int fd\rho = \int fQd\gamma$. Hence
\beqs\int Qd\gamma\l(\int fd\rho - \int f(1+h)d\gamma\r) &=\int Qfd\gamma-\int Qd\gamma\int f(1+h)d\gamma\\
& = \int Q\l(f-\int f(1+h)d\gamma\r)d\gamma=\int Q\bar f_{h}d\gamma.\eeqs Next, note that $\int (1+h)\bar f_h d\gamma=0$, and hence $\int Q\bar f_{h}d\gamma=\int (Q-1-h)\bar f_{h}d\gamma$. Dividing by $\int Qd\gamma$ on both sides finishes the proof of the first equality in~\eqref{lma:pretty}. To prove the second equality, we have 
\beqs
\int (Q-1-h)\bar f_{h}d\gamma&=\int_{\U} (e^r-1-h)\bar f_{h}d\gamma +\int_{\Omega\setminus\U} (e^r-1-h)\bar f_{h}d\gamma -\int_{\Omega^c}(1+h)\bar f_hd\gamma\\
&=\int_{\U} (e^r-1-h)\bar f_{h}d\gamma +\int_{\Omega\setminus\U}e^r\bar f_{h}d\gamma -\int_{\U^c}(1+h)\bar f_hd\gamma,\eeqs as desired. Finally, if $\int_{\U}rd\gamma=0$, then 
\begin{equation}
\int Qd\gamma\geq \int_{\Omega} e^rd\gamma \geq \int_{\U} e^rd\gamma = \gamma(\U)\int e^rd\gamma\vert_{\U} \geq \gamma(\U)e^{\int rd\gamma\vert_{\U}}=\gamma(\U),
\end{equation}
using Jensen's inequality and the fact that $\int_{\U} rd\gamma=0$. 
\end{proof}
\begin{proof}[Proof of Lemma~\ref{lma:prelim}]
We start with the decomposition~\eqref{eq:prelim}. Now, note that $\|\bar f_h\|_2 \leq \|f\|_2 + \int|fh|d\gamma \leq \|f\|_2 + \|f\|_2\|h\|_2\leq 1+\|h\|_2$, since $\|f\|_2=1$. Using this and Lemma~\ref{lma:exppre} with $a=c=4$, $b=2$, gives 
$$|\mathrm{Loc}| \leq (1+\|h\|_2)\|r\ind_{\U}\|_{8}^{2}\|e^r\ind_{\U}\|_4 + (1+\|h\|_2)\|(r-h)\ind_\U\|_2,$$ the desired bound on Loc. The bound on Tail follows by the simple bounds described above the statement of the lemma.
\end{proof}
\begin{lemma}\label{lma:exppre}
Let $1\leq a,b,c\leq\infty$ such that $1/a+1/b+1/c=1$. Let $r\ind_{\U}\in L^{ak}(\gamma)$, $g\in L^b(\gamma)$, $e^{cr}\in L^1(\gamma)$, and $\int r\ind_{\U}d\gamma=0$. Then
\beqs
\calE_k(g):=\l|\int_\U \l(e^r-\sum_{j=0}^{k-1}\frac{r^j}{j!}\r)gd\gamma\r|\leq \frac{1}{k!}\|r\ind_{\U}\|_{ak}^{k}\|g\|_b\|e^r\ind_{\U}\|_c
\eeqs
\end{lemma}
\begin{proof}
Define $I(t)= \int_{\U}e^{tr}gd\gamma$. Then 
\beqs
\calE_k(g)= \l|I(1)-\sum_{j=0}^{k-1}I^{(j)}(0)/j!\r|\leq\frac{1}{k!}\sup_{t\in[0,1]}|I^{(k)}(t)|\leq\frac{1}{k!}\sup_{t\in[0,1]}\int_{\U}|r|^{k}|g|e^{tr}d\gamma.\eeqs Applying Holder's inequality, with powers $a,b,c$ gives
$$
\calE_k(g)\leq \frac{1}{k!}\|r\ind_{\U}\|_{ak}^{k}\|g\|_b\sup_{t\in[0,c]}\l(\int_{\U}e^{tr}d\gamma\r)^{\frac1c}.$$ Next, let $0<t<c$. Note that
$$\int_{\U}e^{tr}d\gamma=\gamma(\U)\int e^{tr}d\gamma\vert_{\U}\leq \gamma(\U)\l(\int e^{cr}d\gamma\vert_{\U}\r)^{\frac tc}\leq \gamma(\U)\int e^{cr}d\gamma\vert_{\U}=\int_{\U}e^{cr}d\gamma.$$ The first inequality holds by Jensen, and the second inequality holds because $\int e^{cr}d\gamma\vert_{\U}\geq e^{\int crd\gamma\vert_{\U}}=1$, also by Jensen. Thus
$$\sup_{t\in[0,c]}\l(\int_{\U}e^{tr}d\gamma\r)^{\frac1c}=\l(\int_{\U}e^{cr}d\gamma\r)^{\frac1c}= \|e^r\ind_{\U}\|_c.$$
\end{proof}
\begin{proof}[Proof of Lemma~\ref{lma:tail}]The first inequality in~\eqref{eq:lma:tail} is immediate from~\eqref{appl:tail}. Since $\|p_3\|_4\les\tilep\leq C$ (by Lemma~\ref{lma:prop:r2k}), we can first bound the middle expression in~\eqref{eq:lma:tail} by replacing $\|p_3\|_2$ in the first integrand and $(1+\|p_3\|_4)^2$ in the second term by 1. Next, using~\eqref{fbd}, we conclude that $|f(x)|\leq \af{f}e^{\cgro\sqrt d\|x\|/2}$ for all $x$, and hence
\beqs\label{1p32}
\int_{\Omega\setminus\U(\s)}(| f| +1)e^{\hat r_3}d\gamma& \leq (\af{f} \vee 1)\int_{\Omega\setminus\U(\s)}e^{\cgro\sqrt d\|x\|/2+\hat r_3(x)}d\gamma\\
& \leq  (\af{f} \vee1)(2\pi)^{-d/2}e^{\|r_{4}\ind_{\U}\|}\int_{\U(\s)^c} e^{-\cgro\sqrt d\|x\|/2}dx.\eeqs
Here, we used that $\hat r_3 = r_3-\int_{\U}r_3d\gamma = r_3 - \int_{\U}r_4d\gamma$, and that
$$r_3(x)-\|x\|^2/2 = W(0)-W(x)\leq -\cgro\sqrt d\|x\|\qquad \forall x\in\Omega\setminus\U(\s)\subset \Omega\setminus\U(\s_0).$$ We then extended the integral from $\Omega\setminus\U(s)$ to $\U(s)^c$, since $e^{-\cgro\sqrt d\|x\|/2}$ is well-defined on $\R^d$. We now apply Lemma~\ref{aux:gamma}, using the assumption $\s \geq \frac{8}{\cgro}\log\frac{2e}{\cgro}$, to get that
\beq\label{rhotail}
\int_{\U(\s)^c}(| f| + 1)e^{\hat r_3}d\gamma\leq (\af{f}\vee1)e^{\|r_{4}\ind_{\U}\|}de^{-\cgro\s d/8}, 
\eeq
Next, note that $\s \geq \frac{8}{\cgro}\log\frac{2e}{\cgro}$ and $\cgro\leq1$ imply $\s\geq8$. Hence a standard Gaussian tail inequality can be applied to bound $\gamma(\U^c)$. This gives 
\beq\label{gamtail}
 \gamma(\U^c)^{\frac14}\leq  e^{-(\s-1)^2d/8}
\eeq
Finally, note that $\s\geq8\geq(3+\sqrt{13})/2$ and $\cgro\leq1$ imply $(\s-1)^2\geq\s\geq\s\cgro$. Thus we can combine the bounds from~\eqref{rhotail} and~\eqref{gamtail} to get
\beq
\l|\mathrm{Tail}\r|\les  \l| \int_{\Omega\setminus\U}(| f|+1)e^{\hat r_3}d\gamma +\gamma(\U^c)^{\frac14}\r| \les \l(a_f\vee1\r)e^{\|r_{4}\ind_{\U}\|_1}de^{-\cgro\s d/8}.\eeq Finally, we use that $\|r_4\ind_{\U}\|_1\les\efour{\s}^2\les1$, where the first inequality is by~\eqref{E4} of Lemma~\ref{lma:r345} below, and the second inequality is by assumption. Thus we can bound $e^{\|r_{4}\ind_{\U}\|_1}$ by a constant.
\end{proof}
\begin{proof}[Proof of Theorem~\ref{thm:odd}]We use~\eqref{pretty}, and the bound on $|\mathrm{Tail}|$ from Lemma~\ref{lma:tail}. It remains to bound Loc. To do so, we use~\eqref{eq:prelim:odd} with $h=p_3$ and $r=\hat r_3$. We also apply Lemma~\ref{lma:exppre} for the first and third term on the righthand side. Finally, we use that $\|f\|_2=1$ and that $\|(\hat r_3-p_3)\ind_\U\|_2=\|\hat r_4\ind_\U\|_2\leq 2\|r_4\ind_\U\|_2$. This gives
\beqs\label{odd-part1}
\l|\int_\U (e^{\hat r_3}-1-\hat r_3-\hat r_3^2/2)fd\gamma\r| &\les \|e^{\hat r_3}\ind_\U\|_4\|r_3\ind_\U\|_{12}^3,\\
\l|\int fp_3d\gamma\int_\U (e^{\hat r_3}-1-\hat r_3)d\gamma\r| &\les \|e^{\hat r_3}\ind_\U\|_2\|r_3\ind_\U\|_4^2\|p_3\|_2,\\
\l|\int fp_3d\gamma\int_\U (\hat r_3-p_3)d\gamma\r|&\les\|p_3\|_2\|r_4\ind_\U\|_2.
\eeqs
We also need to bound the final term in~\eqref{eq:prelim:odd}, namely $\int_\U (\hat r_3 + \hat r_3^2/2 - p_3)fd\gamma$. Write $\hat r_3=p_3 + \hat r_4=p_3+\hat p_4 +\hat r_5$. Then
$$
 \l(\frac12\hat r_3^2+\hat r_3\r) - p_3 = \frac12p_3^2+p_3\hat r_4 + \hat r_4^2+\hat p_4+\hat r_5.$$ Note that $\int_\U p_3^2fd\gamma=\int_\U\hat p_4fd\gamma=0$, since the integrand is odd in both cases, and $\U$ is a symmetric set. Hence
 \beqs\label{odd-part2}
\l| \int_\U (\hat r_3 + \hat r_3^2/2 - p_3)fd\gamma\r| &= \l|\int_\U (p_3\hat r_4+\hat r_4^2+\hat r_5)fd\gamma\r| \\
 &\les \|p_3\|_4\|r_4\ind_\U\|_4+\|r_4\ind_\U\|_4^2 +\|r_5\ind_\U\|_2.
 \eeqs
 Using~\eqref{odd-part1} and~\eqref{odd-part2} in~\eqref{eq:prelim:odd} with $h=p_3$ and $r=\hat r_3$ gives
\beqs
|\mathrm{Loc}|\les &\|e^{\hat r_3}\ind_\U\|_4\l(\|r_3\ind_\U\|_{12}^3 +\|p_3\|_2\|r_3\ind_\U\|_4^2\r)\\
&+\|p_3\|_4\|r_4\ind_\U\|_4+\|r_4\ind_\U\|_4^2 +\|r_5\ind_\U\|_2\\
\les &\Es{\s}((\epsilon_3+\efour{\s}^2)^3 +\epsilon_3(\epsilon_3+\efour{\s}^2)^2) + \epsilon_3\efour{\s}^2 +\efour{\s}^4+d^{-1/2}\efive{\s}^3\\
\les &\Es{\s}(\epsilon_3^3+\efour{\s}^4) + \epsilon_3\efour{\s}^2+d^{-1/2}\efive{\s}^3\\
\leq &\Es{\s}(\epsilon_3^2+\efour{\s}^2)(\epsilon_3 + \efour{\s}^2)+d^{-1/2}\efive{\s}^3
\eeqs
To get the second inequality, we used Lemma~\ref{lma:r345} below, Lemma~\ref{lma:prop:r2k}, and the fact that $\tilep\leq\epsilon_3$. We now substitute this bound and the bound on Tail from Lemma~\ref{lma:tail} into~\eqref{pretty}. This concludes the proof.
\end{proof}
\subsection{Bounds on $r_3,r_4,r_5$}\label{app:r345}
\begin{lemma}\label{lma:r345}
The following bounds hold:
\begin{align}
\sup_{x\in\U}\|\nabla p_3(x)\|&\leq \epsilon_3\s^2/2,\label{sup3}\\
\sup_{x\in\U}|r_4(x)|&\leq \efour{\s}^2\s^4/24,\label{sup4}\\
\|r_4\ind_{\U}\|_k &\les_k\efour{\s}^2,\label{E4}\\
\|r_5\ind_{\U}\|_k &\les_k d^{-\frac12}\epsilon_{5}^{3}.\label{E5}
\end{align} \end{lemma}
\begin{proof}
First, we have $$\sup_{x\in\U}\|\nabla p_3(x)\| = \frac 12\sup_{x\in\U}\|\la\nabla^3W(0), x^{\otimes 2}\ra\|  \leq \frac 12\|\nabla^3W(0)\|\s^2d=\epsilon_3\s^2/2,$$ proving~\eqref{sup3}.
Next, let $m=4,5$. Then we have the bound 
\beq\label{rm}|r_m(x)|\leq\frac{1}{m!}\sup_{t\in[0,1]}\la\nabla^mW(tx), x^{\otimes m}\ra \leq \frac{c_m}{m!n^{m/2-1}}\|x\|^{m},\qquad\forall x\in\U.\eeq Using that $\|x\|\leq \s\sqrt d$ on $\U$ gives $|r_4(x)|\leq \frac{c_4}{4!n}s^4d^2 = \s^4\efour{\s}^2/4!$, proving~\eqref{sup4}. Next, take the $k$th power and then the $\gamma$-expectation of~\eqref{rm} to get
$$\|r_m\ind_{\U}\|_k \leq  \frac{c_m}{m!n^{m/2-1}}\E\l[\|Z\|^{mk}\r]^{1/k} \les_{m,k}\frac{c_m}{n^{m/2-1}}d^{m/2}.$$ Now, taking $m=4$ and $5$ proves~\eqref{E4} and~\eqref{E5}, respectively. 

\end{proof}
\begin{lemma}\label{lma:exp}
It holds
$$\|e^{\hat r_3}\ind_{\U}\|_k\leq \e\l(\l(\efour{\s}^2/12 + k\epsilon_3^2/8\r)\s^4\r).$$
\end{lemma}
\begin{proof}
First note that since $\hat r_3=p_3+\hat r_4$ and $\sup_{x\in\U}|\hat r_4(x)|\leq 2\sup_{x\in\U}|r_4(x)|$, we have
$$\|e^{\hat r_3}\ind_{\U}\|_k\leq \sup_{x\in\U}e^{2|r_4(x)|}\l(\int_{\U}e^{kp_3}d\gamma\r)^{\frac1k}.$$ Now,~\eqref{sup4} of Lemma~\ref{lma:r345} gives $|r_4(x)|\leq\efour{\s}^2\s^4/4!$ on $\U$. Next,~\eqref{sup3} of Lemma~\ref{lma:r345} shows that the Lipschitz constant of $p_3$ on $\U$ is bounded above by $L=\epsilon_3\s^2/2$. Furthermore, note that the measure $\gamma\vert_{\U}$ is $1$-strongly log concave in its domain of definition $\U$, and therefore satisfies a log Sobolev inequality LSI(1), by the Bakry-Emery criterion~\cite{bakry2014analysis}. Therefore using Herbst's argument on the exponential integrability of Lipschitz functions with respect to a measure satisfying LSI(1)~\cite[Proposition 5.4.1]{bakry2014analysis},
$$\int_{\U}e^{kp_3}d\gamma\leq \int e^{kp_3}d\gamma\vert_{\U}\leq \e\l(k\int p_3d\gamma\vert_\U+ k^2L^2/2\r) = \e(k^2\epsilon_3^2\s^4/8).$$ Taking the $(1/k)$th power concludes the proof. 
\end{proof}
\section{Moments of Gaussian tensor inner products}\label{app:hermite}

\subsection{Brief primer on Hermite polynomials}\label{app:hermite-primer}
Let $\gamma=(\gamma_1,\dots,\gamma_d)\in\mathbb N_{\geq0}^d$. We let $|\gamma|=\gamma_1+\dots+\gamma_d$, and $\gamma!=\gamma_1!\dots\gamma_d!$. Then
$$H_\gamma(x_1,\dots, x_d):=\prod_{i=1}^dH_{\gamma_i}(x_i),$$ where $H_k(x)$ is the order $k$ univariate Hermite polynomial. We have $H_0(x)=1, H_1(x)=x, H_2(x)=x^2-1, H_3(x)=x^3-3x$. We have
$$\E[H_\gamma(Z)H_{\gamma'}(Z)] = \delta_{\gamma,\gamma'}\gamma!.$$
Given $i,j,k\in [d]$, let $\gamma(ijk)=(\gamma_1,\dots,\gamma_d)$ be given by 
$$\gamma_\ell = \delta_{i\ell} + \delta_{j\ell}+\delta_{k\ell},\quad\ell=1,\dots,d.$$ In other words $\gamma_\ell$ is the number of times index $\ell\in[d]$ repeats within the string $ijk$. For example 
$$\gamma(111)=(3,0,\dots,0),\quad\gamma(113)=(2,0,1,0,\dots,0).$$ We define $\H_3(x)$ as the $d\times d\times d$ tensor, with entries
$$\H_3^{ijk}(x_1,\dots,x_d) = H_{\gamma(ijk)}(x_1,\dots,x_d).$$ One can show that $\H_3(x)=x^{\otimes 3}-3\mathrm{Sym}(x\otimes I_d)$. Therefore, for a symmetric tensor $S$ we have
$$\la S, \H_3(x)\ra = \la S, x^{\otimes 3}-3x\otimes I_d\ra = \la S, x^{\otimes 3}\ra-3\la S, x\otimes I_d\ra = \la S, x^{\otimes 3}\ra-3\la S, I_d\ra^\T x,$$ and hence
\beq\label{Symm}\la S, x^{\otimes 3}\ra = \la S, \H_3(x)\ra + 3\la S, I_d\ra^\T x.\eeq

\subsection{Moments via Hermite polynomials}\label{app:hermite:main} 
Let $\H_3(x) = x^{\otimes 3}-3\mathrm{Sym}(x\otimes I_d)$ and $\Q_3(x)=x^{\otimes3}$.
\begin{lemma}\label{lma:TH32}If $T$ is a symmetric $d\times d\times d$ tensor, then
\begin{align}
\|\la T, \H_3\ra\|_2^2 &= \E[\la T, \H_3(Z)\ra^2] = 3!\|T\|_F^2.\label{eq:TH32-1}
\end{align}
\end{lemma}
\begin{proof}
Note that given a $\gamma$ with $|\gamma|=3$, there are $3!/\gamma!$ tuples $(i,j,k)\in[d]^3$ for which $\gamma(ijk)=\gamma$. We let $T_{\gamma}$ denote $T_{ijk}$ for any $ijk$ for which $\gamma(ijk)=\gamma$. This is well-defined since $T$ is symmetric. Now, since both $T$ and $\H_3$ are symmetric tensors, we can write the inner product between $T$ and $\H_3$ by grouping together equal terms. In other words, for all $3!/\gamma!$ tuples $i,j,k$ such that $\gamma(ijk)=\gamma$, we have $T_{ijk}\H_3^{ijk}=T_{\gamma}H_{\gamma}$. Therefore,
$$\la T, \H_3(Z)\ra = \sum_{|\gamma|=3}\frac{3!}{\gamma!}T_\gamma H_\gamma(Z).$$ Using this formula, we get
\beqs
\E[\la T, \H_3(Z)\ra^2] &=\sum_{|\gamma|=3,|\gamma'|=3}^d\frac{3!}{\gamma!}\frac{3!}{\gamma'!}T_{\gamma}T_{\gamma'}\E[H_\gamma(Z)H_{\gamma'}(Z)]\\
&=\sum_{|\gamma|=3}^d\frac{3!}{\gamma!}\frac{3!}{\gamma!}T_{\gamma}^2\gamma!=3!\sum_{|\gamma|=3}^d\frac{3!}{\gamma!}T_{\gamma}^2 \\
&= 3!\sum_{i,j,k=1}^dT_{ijk}^2=3!\|T\|_F^2.
\eeqs
\end{proof}
\begin{lemma}\label{hypercon}
If $T$ is a symmetric $d\times d\times d$ tensor, then
\begin{align}
\|\la T, \H_3\ra\|_{k} &\leq (k-1)^{3/2} \|\la T,\H_3\ra\|_2.\label{eq:hypercon1}\\
\end{align}
\end{lemma}
\begin{proof}
Let $\mathcal L$ be the generator for the $d$-dimensional Ornstein-Uhlenbeck process. Then it is known that $(\mathcal LH_\gamma)(x)=-|\gamma|H_\gamma(x)$, i.e. the $H_\gamma$ are eigenfunctions of $\mathcal L$ with corresponding eigenvalues $-|\gamma|$. Hence, $P_t\la T,\H_3\ra= e^{-3t}\la T, \H_3\ra$ for any tensor $T$, where $P_t=e^{t\mathcal L}$. This is because $|\gamma|=3$ for all $H_\gamma$ making up the tensor $\H_3$. Now, by hypercontractivity (see e.g. Chapter 5.2.2 of~\cite{bakry2014analysis}), we have
$$e^{-3t}\|\la T, H_3\ra\|_{q(t)}=\|P_t\la T,\H_3\ra\|_{q(t)}\leq \|\la T,\H_3\ra\|_2,$$ where $q(t)=1+e^{2t}$. Setting $k=q(t)$ we get $e^{3t}=(k-1)^{3/2}$, so that
\beq
\|\la T, \H_3\ra\|_{2k} \leq (k-1)^{3/2} \|\la T,\H_3\ra\|_2,
\eeq as desired.
\end{proof}
\section{Auxiliary Results}\label{app:aux}

\begin{lemma}\label{aux:gamma}Let $a,b>0$. Then
\beqs
I:=\frac{1}{(2\pi)^{d/2}}\int_{\|x\|\geq a\sqrt d}^\infty e^{-b\sqrt d\|x\|}dx\leq \frac de\e\l(\l[\log\l(\frac{e}{b}\r)- \frac12ab\r]d\r)
\eeqs In particular, if $a \geq \frac{4}{b}\log\frac{e}{b}$ then 
$$I\leq \frac de e^{-abd/4}.$$
\end{lemma}
\begin{proof}
Switching to polar coordinates and then changing variables, we have
\beqs\label{ISd}
I&=\frac{S_{d-1}}{(2\pi)^{d/2}}\int_{a\sqrt d}^\infty u^{d-1}e^{-b\sqrt du}du= \frac{S_{d-1}}{(2\pi)^{d/2}(b\sqrt d)^{d}}\int_{abd}^\infty u^{d-1}e^{-u}du,
\eeqs where $S_{d-1}$ is the surface area of the unit sphere. Now, we have 
$$ \frac{S_{d-1}}{(2\pi)^{d/2}} = \frac{2\pi^{d/2}}{\Gamma(d/2)(2\pi)^{d/2}}\leq 2\frac{(2e/d)^{d/2-1}}{2^{d/2}}=\l(\frac ed\r)^{\frac d2-1},$$
using that $\Gamma(d/2)\geq (d/2e)^{d/2-1}$. To bound the integral in~\eqref{ISd}, we use Lemma~\ref{gamma} with $\lambda = abd$ and $c=d$. Combining the resulting bound with the above bound on $S_{d-1}/(2\pi)^{d/2}$, we get
\beqs
I &\leq \l(\frac ed\r)^{\frac d2-1}(b\sqrt d)^{-d}e^{-tabd}\l(\frac{d}{1-t}\r)^d= \frac de\l(\frac{\sqrt e}{b(1-t)}\r)^de^{-tabd}
\eeqs Taking $t=1/2$ and noting that $2\sqrt e\leq e$ concludes the proof.
\end{proof}
\begin{lemma}\label{gamma}For all $\lambda,c>0$ and $t\in(0,1)$ it holds
$$\int_\lambda^\infty u^{c-1}e^{-u}du\leq  e^{-\lambda t}\l(\frac{c}{1-t}\r)^{c}.$$
\end{lemma}
\begin{proof}
Let $X$ be a random variable with gamma distribution $\Gamma(c, 1)$. Then the desired integral is given by $\Gamma(c)\mathbb P(X\geq\lambda)$. 
Now, the mgf of $\Gamma(c,1)$ is $\E[e^{Xt}]=(1-t)^{-c}$, defined for $t<1$. Hence for all $t\in(0,1)$ we have
\beq\label{mgf}
\mathbb P(X\geq\lambda)\leq e^{-\lambda t}(1-t)^{-c}.\eeq  Multiplying both sides by $\Gamma(c)$ and using that $\Gamma(c)\leq c^c$ gives the desired bound.
\end{proof}
\begin{lemma}\label{TATF}Let $S$ be a symmetric $d\times d\times d$ tensor and $A$ be a symmetric $d\times d$ matrix. Then
\beq
\|\la S, A\ra\|\leq d\|A\|\|S\|,\qquad \|S\|_F \leq d\|S\|,
\eeq where $\|\cdot\|_F$ is the Frobenius norm.
\end{lemma}
\begin{proof}
Note that $\|\la S,  A\ra\| = \sup_{\|u\|=1}\la S, A\ra^\T u = \sup_{\|u\|=1}\la S, A\otimes u\ra$. Let $A=\sum_{i=1}^d\lambda_iv_iv_i^\T $ be the eigendecomposition of $A$, where the $v_i$ are unit vectors. Then
\beq\label{TA}\|\la S, A\ra\| = \sup_{\|u\|=1}\la S, A\otimes u\ra\leq\sum_{i=1}^d|\lambda_i|\l|\la S, v_i\otimes v_i\otimes u\ra\r|\leq d\|A\|\|S\|.\eeq For the second inequality, write $S=[A_1,\dots, A_d]$ where each $A_i\in\R^{d\times d}$. That is, $S_{ijk} = (A_i)_{jk}$. Note that $\|A_i\|\leq \|S\|$ for each $i=1,\dots,d$. We now have
\beq\label{TF}\|S\|_F^2=\sum_{i=1}^d\mathrm{Tr}(A_i^\T A_i)\leq\sum_{i=1}^dd\|A_i\|^2\leq d^2\|S\|^2.\eeq
\end{proof}
\bibliographystyle{plain} 
\bibliography{bibliogr_Lap}  

\end{document}